\documentclass[11pt]{article}
\usepackage[utf8]{inputenc}
\usepackage[T1]{fontenc}
\usepackage[margin=1in]{geometry}

\usepackage{CJKutf8} 
\usepackage{setspace}
\usepackage{amssymb,amsmath,mathtools,amsfonts}
\usepackage{amsthm,thmtools}
\newtheorem{theorem}{Theorem}[section]
\newtheorem{lemma}[theorem]{Lemma}

\newtheorem{corollary}[theorem]{Corollary}

\newtheorem{claim}{Claim}[lemma]
\theoremstyle{definition}
\newtheorem{definition}[theorem]{Definition}
\newtheorem{remark}[theorem]{Remark}
\newtheorem{examplex}[theorem]{Example}
\newenvironment{example}
  {\pushQED{\qed}\examplex}
  {\popQED\endexamplex}

\usepackage[numbers, square]{natbib}  

\usepackage{pifont} 
\usepackage{xcolor} 
\usepackage{tikz} 
\usetikzlibrary{positioning}
\usetikzlibrary{arrows.meta,shapes,snakes}
\usepackage{enumitem} 
\usepackage{soul} 
\usepackage{bm} 
\usepackage{subcaption} 
\usepackage[normalem]{ulem}
\usepackage[breaklinks=true]{hyperref} 
\hypersetup{
	colorlinks=true,
	citecolor=green!70!blue,
	linkcolor=blue,  
}

\usepackage{algorithm}
\usepackage{algorithmicx}
\usepackage{algpseudocode}
\algnewcommand\Input{\item[\textbf{Input:}]}
\algnewcommand\Output{\item[\textbf{Output:}]}
\algnewcommand{\IfThen}[2]{
	\State \algorithmicif\ #1\ \algorithmicthen\ #2
}
\algnewcommand{\IfThenElse}[3]{
	\State \algorithmicif\ #1\ \algorithmicthen\ #2 \ \algorithmicelse\ #3
}
\makeatletter
\newcommand{\algmargin}{\the\ALG@thistlm}   
\makeatother
\algnewcommand{\parState}[1]{\State%
	\parbox[t]{\dimexpr\linewidth-\algmargin}{\strut #1\strut}
}
\algnewcommand{\ForInline}[2]{%
    \State \algorithmicfor\ #1\ \algorithmicdo\ #2 \algorithmicend\ \algorithmicfor%
}
	
\def\N{{\mathbb N}}

\def\R{{\mathbb R}}
\def\SS{{\cal S}}
\def\B{{\cal B}}
\def\C{{\cal C}}
\def\I{{\cal I}}
\def\U{{\cal U}}
\def\D{{\cal D}}
\def\T{{\cal T}}
\def\O{{\cal O}}
\def\P{{\cal P}}
\def\H{{\cal H}}
\def\X{{\cal X}}
\def\L{{\cal L}}

\def\A{{\cal A}}
\def\F{{\cal F}}
\newcommand{\conv}{\textup{conv}}
\newcommand{\Chi}{\raisebox{0pt}[.9ex][.9ex]{$\chi$}}
\let\bar\overline

\let\tilde\widetilde
\newcommand{\textuptt}[1]{{\textup{\texttt{#1}}}}

\providecommand{\keywords}[1]
{
  \small	
  \noindent {\bf \textit{Keywords}:} #1
}

\title{Affinely representable lattices, stable matchings, \\and choice functions}

\author{Yuri Faenza\and Xuan Zhang}

\date{ IEOR Department, Columbia University \\ \{yf2414,xz2569\}@columbia.edu }

\begin{document}

\maketitle

\begin{abstract}
    Birkhoff's representation theorem~\cite{birkhoff1937rings} defines a bijection between elements of a distributive lattice and the family of upper sets of an associated poset. Although not used explicitly, this result is at the backbone of the combinatorial algorithm by Irving et al.~\cite{irving1987efficient} for maximizing a linear function over the set of stable matchings in Gale and Shapley's stable marriage model~\cite{gale1962college}. In this paper, we introduce a property of distributive lattices, which we term as affine representability, and show its role in efficiently solving linear optimization problems over the elements of a distributive lattice, as well as describing the convex hull of the characteristic vectors of lattice elements. We apply this concept to the stable matching model with path-independent quota-filling choice functions, thus giving efficient algorithms and a compact polyhedral description for this model. To the best of our knowledge, this model generalizes all models from the literature for which similar results were known, and our paper is the first that proposes efficient algorithms for stable matchings with choice functions, beyond classical extensions of the Deferred Acceptance algorithm. 
\end{abstract}


\newpage
\tableofcontents

\newpage
\setcounter{page}{1}

\section{Introduction}

Since Gale and Shapley's seminal publication~\cite{gale1962college}, the concept of stability in matching markets has been widely studied by the optimization community. With minor modifications, the one-to-many version of Gale and Shapley's original stable \emph{marriage} model is currently employed in the National Resident Matching Program~\cite{roth1984evolution}, which assigns medical residents to hospitals in the US, and for assigning eighth-graders to public high schools in many major cities in the US~\cite{abdulkadirouglu2003school}.

In this paper, matching markets have two sides, which we call firms $F$ and workers $W$. In the marriage model, every agent from $F\cup W$ has a \emph{strict preference list} that ranks agents in the opposite side of the market. The problem asks for a \emph{stable matching}, which is a matching where no pair of agents prefer each other to their assigned partners. A stable matching can be found efficiently via the Deferred Acceptance (DA) algorithm~\cite{gale1962college}. Although successful, the marriage model does not capture features that have become of crucial importance both inside and outside academia. For instance, there is growing attention to models that can increase diversity in school cohorts~\cite{nguyen2019stable,tomoeda2018finding}. Such constraints cannot be represented in the original model, or its one-to-many or many-to-many generalizations, since admission decisions with diversity concerns cannot be captured by a strict preference list. 

To model these and other markets, instead of ranking individual potential partners, each agent $a\in F\cup W$ is endowed with a \emph{choice function} $\C_a$ that picks a team she prefers the best from a given set of potential partners. See, e.g.,~\cite{echenique2015control,aygun2016dynamic,kamada2015efficient} for more applications of models with choice functions. Models with choice functions were first studied in~\cite{roth1984stability,kelso1982job} (see Section~\ref{sec:literature}). \emph{Mutatis mutandis}, one can define a concept of stability in this model as well (for this and the other technical definition mentioned below, see Section~\ref{sec:basics}). Two classical assumptions on choices functions are \emph{substitutability} and \emph{consistency}, under which the existence of stable matchings is guaranteed~\cite{hatfield2005matching,aygun2013matching}. Clearly, existence results are not enough for applications (and for optimizers). Interestingly, little is known about efficient algorithms in models with choice functions. Only extensions of the classical Deferred Acceptance algorithm for finding the one-side optimal matching have been studied for this model~\cite{roth1984stability,chambers2017choice}.

The goal of this paper is to study algorithms for optimizing a linear function $w$ over the set of stable matchings in models with choice functions, where $w$ is defined over firm-worker pairs. Such questions are classical in combinatorial optimization, see, e.g.,~\cite{schrijver2003combinatorial} (and~\cite{manlove2013algorithmics} for problems on matching markets). We focus on two models. The first model (\textsc{CM-Model}) assumes that all choice functions are substitutable, consistent, and cardinal monotone. The second model (\textsc{CM-QF-Model}) additional assumes that for one side of the market, choice functions are also \emph{quota-filling}. Both models generalize all classical models where agents have strict preference lists, on which results for the question above were known. For these models, Alkan~\cite{alkan2002class} has shown that stable matchings form a distributive lattice. As we argue next, this is a fundamental property that allows us to solve our optimization problem efficiently.

\subsection{Our contributions and techniques}\label{sec:techniques+results}

We give here a high-level description of our approach and results. For the standard notions of posets, distributive lattices, and related definitions see Section~\ref{sec:basic:poset}. All sets considered in this paper are finite. 

Let $\L=(\X,\succeq)$ be a distributive lattice, where the elements of $\X$ are distinct subsets of a base set $E$ and $\succeq$ is a partial order on $\X$. We refer to $S \in \X$ as an \emph{element} (of the lattice). Birkhoff's theorem~\cite{birkhoff1937rings} implies that we can associate\footnote{The result proved by Birkhoff is actually a bijection between the families of lattices and posets, but in this paper we shall not need it in full generality.} to $\L$ a poset $\B=(Y,\succeq^\star)$ such that there is a bijection $\psi: {\cal X} \rightarrow {\cal U}({\cal B})$, where ${\cal U}({\cal B})$ is the family of \emph{upper sets} of ${\cal B}$. $U\subseteq Y$ is an upper set of $\B$ if $y \in U$ and  $y' \succeq^\star y$ for some $y' \in Y$ implies $y' \in U$. We say therefore that $\B$ is a \emph{representation poset} for $\L$ with the \emph{representation function} $\psi$. See Example~\ref{ex:first} below. $\B$ may contain much fewer elements than the lattice ${\cal L}$ it represents, thus giving a possibly ``compact'' description of ${\cal L}$.

The representation poset $\B$ and the representation function $\psi$ are univocally defined per Birkhoff's theorem. Moreover, the representation function $\psi$ satisfies that for $S,S'\in \X$, $S\succeq S'$ if and only if $\psi(S)\subseteq \psi(S')$. Although $\B$ explains how elements of $\X$ are related to each other with respect to $\succeq$, it does not contain any information on which items from $E$ are contained in each lattice element. We introduce therefore Definition~\ref{def:representation-poset}. For $S\in \X$ and $U\in \U(\B)$, we write $\Chi^S \in \{0,1\}^E$ and $\Chi^U\in \{0,1\}^Y$ to denote their characteristic vectors, respectively.

\begin{definition} \label{def:representation-poset} 
    Let ${\cal L}=(\X,\succeq)$ be a distributive lattice on a base set $E$ and $\B=(Y,\succeq^\star)$ be a representation poset for ${\cal L}$ with representation function $\psi$. $\B$ is an \emph{affine representation} of ${\cal L}$ if there exists an affine function $g:\R^Y \rightarrow \R^{E}$ such that $g(\Chi^U)= \chi^{\psi^{-1}(U)}$, for all $U\in {\cal U}(\B)$. In this case, we also say that $\B$ \emph{affinely represents} ${\cal L}$ via function $g$ and that ${\cal L}$ is \emph{affinely representable}. 
\end{definition}

In Definition~\ref{def:representation-poset}, we can always assume $g(u)=Au + x^0$, where $A \in \{0,\pm 1\}^{E \times Y}$ and $x^0$ is the characteristic vector of the maximal element of $\L$.

\begin{example}\label{ex:first}

    Consider first the distributive lattice $\L=(\X,\succeq)$ whose Hasse diagram is given in the Figure~\ref{fig:aff-rep-yes}, with base set $E= \{1,2,3,4\}$.
    
    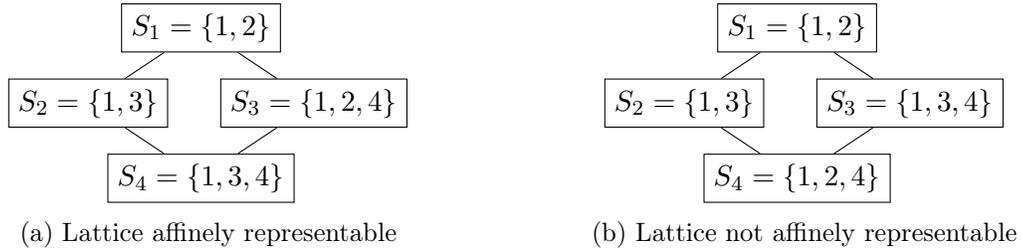
\begin{figure}[h]
    \centering 
    \begin{subfigure}{.48\textwidth}
    \centering
    \begin{tikzpicture}
        \node[rectangle, draw] (1) at (1,1) {$S_1=\{1,2\}$};
        \node[rectangle, draw] (2) at (-.5,0) {$S_2=\{1,3\}$};
        \node[rectangle, draw] (3) at (2.5,0) {$S_3=\{1,2,4\}$};
        \node[rectangle, draw] (4) at (1,-1) {$S_4=\{1,3,4\}$};
        \draw[] (1) -- (2);
        \draw[] (1) -- (3);
        \draw[] (4) -- (2);
        \draw[] (4) -- (3);
    \end{tikzpicture}
    \caption{Lattice affinely representable}\label{fig:aff-rep-yes}
    \end{subfigure}%
    \begin{subfigure}{.48\textwidth}
    \centering
    \begin{tikzpicture}
        \node[rectangle, draw] (1) at (1,1) {$S_1=\{1,2\}$};
        \node[rectangle, draw] (2) at (-.5,0) {$S_2=\{1,3\}$};
        \node[rectangle, draw] (3) at (2.5,0) {$S_3=\{1,3,4\}$};
        \node[rectangle, draw] (4) at (1,-1) {$S_4=\{1,2,4\}$};
        \draw[] (1) -- (2);
        \draw[] (1) -- (3);
        \draw[] (4) -- (2);
        \draw[] (4) -- (3);
    \end{tikzpicture}
    \caption{Lattice not affinely representable} \label{fig:aff-rep-no}
    \end{subfigure}%
    \caption{Lattices for Example~\ref{ex:first}.}
    \end{figure}

    The representation poset $\B=(Y,\succeq^\star)$ of ${\cal L}$ contains two non-comparable elements, $y_1$ and $y_2$. The representation function $\psi$ maps $S_i$ to $U_i$ for $i\in [4]$ with $U_1= \emptyset$, $U_2= \{y_1\}$, $U_3= \{y_2\}$, and $U_4= \{y_1, y_2\}$. That is, $\U(\B)=\{U_1, U_2, U_3, U_4\}$. One can think of $y_1$ as the operation of adding $\{3\}$ and removing $\{2\}$, and $y_2$ as the operation of adding $\{4\}$. $\B$ affinely represents ${\cal L}$ via the function $g(\Chi^U)=A\Chi^U+ \Chi^{S_1}$ where 

    \begin{figure}[H]
    \centering
    \begin{tikzpicture}
    \node[] (a) at (-1,0) {\normalsize $A =\begin{pmatrix} 0 & 0 \\ -1 & 0 \\ 1 & 0 \\ 0 & 1 \end{pmatrix},$};
    \node[left, right=0cm of a] (b) {\normalsize as};
    \node[left, right=0cm of b] {\normalsize
    $\begin{array}{lclclclc}
        g(\Chi^{U_1})^\intercal&= &(0,0,0,0) &+ &(1,1,0,0) &= (1,1,0,0) &= (\Chi^{S_1})^\intercal; \\[.5mm]
        g(\Chi^{U_2})^\intercal&= &(0,-1,1,0) &+ &(1,1,0,0) &= (1,0,1,0) &= (\Chi^{S_2})^\intercal; \\[.5mm]
        g(\Chi^{U_3})^\intercal&= &(0,0,0,1) &+ &(1,1,0,0) &= (1,1,0,1) &= (\Chi^{S_3})^\intercal; \\[.5mm]
        g(\Chi^{U_4})^\intercal&= &(0,-1,1,1) &+ &(1,1,0,0) &= (1,0,1,1) &= (\Chi^{S_4})^\intercal.
    \end{array}$
    };
    \end{tikzpicture}
    \end{figure}

    Next consider the distributive lattice $\L'$ whose Hasse diagram is presented in Figure~\ref{fig:aff-rep-no}. Note that the same poset $\B$ represents $\L'$ with the same representation function $\psi$. Nevertheless, $\L'$ is not affinely representable. If it is and such a function $g(\Chi^U)=A\Chi^U+ \Chi^{S_1}$ exists, then since $(\Chi^{U_1} + \Chi^{U_4})^\intercal = (1,1) = (\Chi^{U_2} + \Chi^{U_3})^\intercal$, we must have
    \begin{equation*}
        \Chi^{S_1}+ \Chi^{S_4}= (\Chi^{S_1}+ A\Chi^{U_1})+ (\Chi^{S_1}+ A\Chi^{U_4})=  (\Chi^{S_1}+ A\Chi^{U_2})+ (\Chi^{S_1}+ A\Chi^{U_3})= \Chi^{S_2}+ \Chi^{S_3}.
    \end{equation*}
    However, this is clearly not the case as $(\Chi^{S_1} + \Chi^{S_4})^\intercal = (2,2,0,1)$ but $(\Chi^{S_2} + \Chi^{S_3})^\intercal = (2,0,2,1)$.
\end{example}

As we show next, affine representability allows one to efficiently solve linear optimization problems over elements of a distributive lattice. In particular, it generalizes properties that are at the backbone of algorithms for optimizing a linear function over the set of stable matchings in the marriage model and its one-to-many and many-to-many generalizations (see, e.g., \cite{irving1987efficient,bansal2007polynomial}). For instance, in the marriage model, the base set $E$ is the set of potential pairs of agents from two sides of the market, $\X$ is the set of stable matchings, and for $S,S' \in \X$, we have $S\succeq S'$ if every firm prefers its partner in $S$ to its partner in $S'$. Elements of its representation poset are certain (trading) cycles, called \emph{rotations}. 

\begin{lemma}\label{thm:reduction-intro}
    Suppose we are given a poset ${\cal B}=(Y,\succeq^\star)$ that affinely represents a lattice ${\cal L}=(\X,\succeq)$ with representation function $\psi$. Let $w:E\rightarrow \R$ be a linear function over the base set $E$ of $\L$. Then the problem $\max \{w^\intercal  \Chi^S : S \in \X\}$ can be solved in time \texttt{min-cut}$(|Y|+2)$, where \texttt{min-cut}$(k)$ is the time complexity required to solve a minimum cut problem with nonnegative weights in a digraph with $k$ nodes.  
\end{lemma}

\begin{proof}
    Let $g(u)=Au+x^0$ be the affine function from the definition of affine representability. We have:
    $$\max_{S \in \X} w^\intercal  \Chi^S = \max_{ U \in {\cal U}({\cal B})} w^\intercal  g(\chi^U) = \max_{U \in {\cal U}({\cal B})} w^\intercal  (A\chi^U + x^0) = w^\intercal  x^0 + \max_{U \in {\cal U}({\cal B})} (w^\intercal A)\chi^U.$$
    Our problem boils down therefore to the optimization of a linear function over the upper sets of ${\cal B}$. It is well-known that the latter problem is equivalent to computing a minimum cut in a digraph with $|Y|+2$ nodes~\cite{picard1976maximal}.
\end{proof}

We want to apply Lemma~\ref{thm:reduction-intro} to the \textsc{CM-QF-Model} model. Observe that a choice function may be defined on all the (exponentially many) subsets of agents from the opposite side. We avoid this computational concern by modeling choice functions via an oracle model. That is, choice functions can be thought of as agents' private information. The complexity of our algorithms will therefore be expressed in terms of $|F|$, $|W|$, and the time required to compute the choice function $\C_{a}(X)$ of an agent $a \in F \cup W$, where the set $X$ is in the domain of $\C_a$. The latter running time is denoted by \textuptt{oracle-call} and we assume it to be independent of $a$ and $X$. Our first result is the following.

\begin{theorem}\label{thm:main-intro}
    The distributive lattice $({\cal S},\succeq)$ of stable matchings in the \textsc{CM-Model} is affinely representable. Its representation poset $(\Pi,\succeq^\star)$ has $O(|F||W|)$ elements. This representation poset, as well as its representation function $\psi$ and affine function $g(u)= Au + x^0$, can be computed in time $O(|F|^3 |W|^3 \textuptt{oracle-call})$ for the \textsc{CM-QF-Model}. Moreover, matrix $A$ has full column rank.
\end{theorem}

In Theorem~\ref{thm:main-intro}, we assumed that operations, such as comparing two sets and obtaining an entry from the set difference of two sets, take constant time. If this is not the case, a factor mildly polynomial in $|F|\cdot|W|$ needs to be added to the running time. Observe that Theorem~\ref{thm:main-intro} is the union of two statements. First, the distributive lattice of stable matchings in the \textsc{CM-Model} is affinely representable. Second, this representation and the corresponding functions $\psi$ and $g$ can be found efficiently for the \textsc{CM-QF-Model}. Those two results are proved in Section~\ref{sec:affine-representability} and Section~\ref{sec:algo}, respectively. Combining Theorem~\ref{thm:main-intro}, Lemma~\ref{thm:reduction-intro} and algorithms for min-cut (see, e.g., \cite{schrijver2003combinatorial}), we obtain the following.

\begin{corollary}\label{cor:main-intro}
    The problem of optimizing a linear function over the set of stable matchings in the \textsc{CM-QF-Model} can be solved in time $O(|F|^3 |W|^3 \textuptt{oracle-call})$.
\end{corollary}

As an interesting consequence of studying a distributive lattice via the poset that affinely represents it, one immediately obtains a linear description of the convex hull of the characteristic vectors of elements of the lattice (see Section~\ref{sec:poly}). In contrast, most stable matching literature (see Section~\ref{sec:literature}) has focused on deducing linear descriptions for special cases of our model via ad-hoc proofs, independently of the lattice structure. 

\begin{theorem}\label{thm:poly-intro}
    Let ${\cal L}=(\X,\succeq)$ be a distributive lattice and ${\cal B}=(Y,\succeq^\star)$ be a poset that affinely represents it via function $g(u)=Au + x^0$. Then the extension complexity of $\conv(\X)\coloneqq \conv\{\Chi^S: S \in \X\}$ is $O(|Y|^2)$. If moreover $A$ has full column rank, then $\conv(\X)$ has $O(|Y|^2)$ facets.
\end{theorem}

Theorem~\ref{thm:main-intro} and Theorem~\ref{thm:poly-intro} imply the following description of the stable matching polytope $\conv(\SS)$, i.e., the convex hull of the characteristic vectors of stable matchings. 

\begin{corollary}
    $\conv(\SS)$ has $ O(|F|^2|W|^2)$ facets in the \textsc{CM-Model}.
\end{corollary}

We next give an example of a lattice represented via a non-full-column rank matrix $A$. 

\begin{example}\label{ex:not-aff-isomorphic}
    Consider the distributive lattice given in Figure~\ref{fig:non-full-rank-X}. It can be represented via the poset $\B=(Y,\succeq^\star)$ that contains three elements $y_1$, $y_2$, and $y_3$ where $y_1\succeq^\star y_2\succeq^\star y_3$. The upper sets of $\B$ are $\U(\B)= \{\emptyset, \{y_1\}, \{y_1, y_2\}, \{y_1, y_2, y_3\}\}$. In addition, $\B$ affinely represents ${\cal L}$ via the function $g(\Chi^U)= A\Chi^U+ \Chi^{S_1}$, where $A$ is given in Figure~\ref{fig:non-full-rank-A}. Matrix $A$ clearly does not have full column rank.
\end{example}

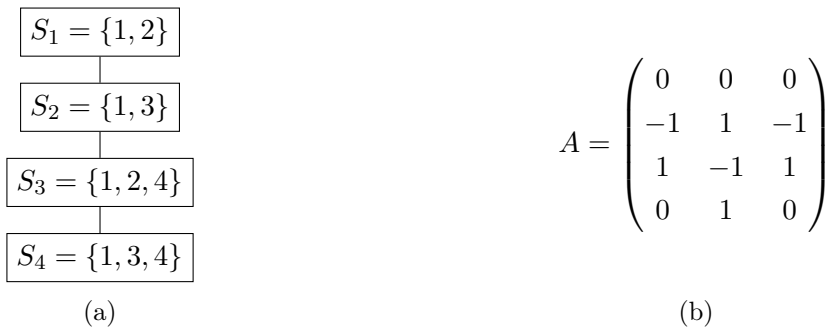
\begin{figure}[H]
    \centering
    \begin{subfigure}{.48\textwidth}
    \centering
    \begin{tikzpicture}
        \node[rectangle, draw] (1) at (1,1) {\normalsize $S_1=\{1,2\}$};
        \node[rectangle, draw] (2) at (1,0) {\normalsize $S_2=\{1,3\}$};
        \node[rectangle, draw] (3) at (1,-1) {\normalsize $S_3=\{1,2,4\}$};
        \node[rectangle, draw] (4) at (1,-2) {\normalsize $S_4=\{1,3,4\}$};
        \draw[] (1) -- (2);
        \draw[] (2) -- (3);
        \draw[] (3) -- (4);
    \end{tikzpicture}
    \caption{} \label{fig:non-full-rank-X}
    \end{subfigure}%
    \begin{subfigure}{.48\textwidth}
    \centering
    \begin{tikzpicture}
        \node[] at (0,1.2) {};
        \node[] at (0,-2.2) {};
        \node[] at (0,-.5) {\normalsize $A=\begin{pmatrix}0 & 0 & 0 \\ -1 & 1 & -1 \\ 1 & -1 & 1 \\ 0 & 1 & 0\end{pmatrix}$};
    \end{tikzpicture}
    \caption{} \label{fig:non-full-rank-A}
    \end{subfigure}%
    \caption{Affine representation with non-full-column-rank matrix $A$} \label{fig:non-full-rank}
\end{figure}

Lastly, in Section~\ref{sec:representation-and-algo}, we discuss alternative ways to represent choice functions, dropping the oracle-model assumption. Interestingly, we show that choice functions in the \textsc{CM-Model} (i.e., substitutable, consistent, and cardinal monotone) do not have polynomial-size representation because the number of possible choice functions in such a model is doubly-exponential in the size of acceptable partners.

\subsection{Relationship with the literature} \label{sec:literature} 

Gale and Shapley~\cite{gale1962college} introduced the one-to-one stable marriage (\textsc{SM-Model}) and the one-to-many stable admission model (\textsc{SA-Model}), and presented an algorithm which finds a stable matching. McVitie and Wilson~\cite{mcvitie1971stable} proposed the break-marriage procedure that allows us to find the full set of stable matchings. Irving et al.~\cite{irving1987efficient} presented an efficient algorithm for the maximum-weighted stable matching problem with weights over pairs of agents, utilizing the fact stable matchings form a distributive lattice~\cite{knuth1976marriages} and that its representation poset -- an affine representation following our terminology -- can be constructed efficiently via the concept of rotations~\cite{irving1986complexity}. The above-mentioned structural and algorithm results were shown for its many-to-many generalization (\textsc{MM-Model}) by Ba{\"\i}ou and Balinski~\cite{baiou2000many}, and Bansal et al.~\cite{bansal2007polynomial}. A complete survey of results on these models can be found, e.g., in~\cite{gusfield1989stable,manlove2013algorithmics}.

For models with substitutable and consistent choice functions, Roth~\cite{roth1984stability} proved that stable matchings always exist by generalizing the algorithm presented in~\cite{gale1962college}. Blair~\cite{blair1988lattice} proved that stable matchings form a lattice, although not necessarily distributive. Alkan~\cite{alkan2001preferences} showed that if choice functions are further assumed to be quota-filling, the lattice is distributive.  Results on (non-efficient) enumeration algorithms for certain choice functions appeared in~\cite{martinez2004algorithm}.

It is then natural to investigate whether algorithms from~\cite{bansal2007polynomial,irving1986complexity} can be directly extended to construct the representation poset in the \textsc{CM-QF-Mode} or the more general \textsc{CM-Model}. However, their definition of rotation and techniques rely on the fact that there is a strict ordering of partners, which is not available with choice functions. This, for instance, leads to the fact that the symmetric difference of two stable matchings that are adjacent in the Hasse Diagram of the lattice is a simple cycle, which is not always true in the \textsc{CM-Model} (see Example~\ref{eg:p-set}). We take then a more fundamental approach by showing a carefully defined ring of sets is isomorphic to the set of stable matchings, and thus we can construct the rotation poset following a maximal chain of the stable matching lattice. This approach conceptually follows the one by Gusfield and Irving~\cite{gusfield1989stable} for the \textsc{SM-Model} and leads to a generalization of the break-marriage procedure from~\cite{mcvitie1971stable}. Again, proofs in \cite{gusfield1989stable,mcvitie1971stable} heavily rely on the strict ordering of partners, while we need to tackle the challenge of not having one.

Besides the combinatorial perspective, another line of research focuses on the polyhedral aspects. Linear descriptions of the convex hull of the characteristic vectors of stable matchings are provided for the \textsc{SM-Model}~\cite{vate1989linear,rothblum1992characterization,roth1993stable}, the \textsc{SA-Model}~\cite{baiou2000stable}, and the \textsc{MM-Model}~\cite{fleiner2003stable}. In this paper, we provide a polyhedral description for the \textsc{CM-QF-Model}, by drawing connections between the order polytope (i.e., the convex hull of the characteristic vectors of upper sets of a poset) and Birkhoff's representation theorem of distributive lattices. 
A similar approach has been proposed in~\cite{aprile20182}: their result can be seen as a specialization of Theorem~\ref{thm:main-intro} to the \textsc{SM-Model}.

\section{Basics} \label{sec:basics}

\subsection{Posets, lattices, and distributivity} \label{sec:basic:poset}

A set $X$ endowed with a partial order relation $\ge$, denoted as $(X,\ge)$, is called a \emph{partially ordered set (poset)}.  When the partial order $\ge$ is clear from context, we often times simply use $X$ to denote the poset $(X, \ge)$. Let $a,a'\in X$, if $a'> a$, we say $a'$ is a \emph{predecessor} of $a$ in poset $(X,\ge)$, and $a$ is a \emph{descendant} of $a'$ in poset $(X,\ge)$. If moreover, there is no $b\in X$ such that $a' > b > a$, we say that $a'$ an \emph{immediate predecessor} of $a$ in poset $(X,\ge)$ and that $a$ is an \emph{immediate descendant} of $a'$ in poset $(X,\ge)$. If $a\not\ge a'$ and $a'\not\ge a$, we say $a$ and $a'$ are \emph{incomparable}.

For a subset $S\subseteq X$, an element $a\in X$ is said to be an \emph{upper bound} (resp.~\emph{lower bound}) of $S$ if for all $b\in S$, $a\ge b$ (resp.~$b\ge a$). An upper bound (resp.~lower bound) $a'$ of $S$ is said to be its \emph{least upper bound} or \emph{join} (resp.~\emph{greatest lower bound} or \emph{meet}), if $a \ge a'$ (resp.~$a'\ge a$) for each upper bound (resp.~lower bound) $a$ of $S$.

A \emph{lattice} is a poset for which every pair of elements has a join and a meet, and for every pair those are unique by definition. Thus, two binary operations are defined over a lattice: join and meet. A lattice is \emph{distributive} where the operations of join and meet distribute over each other.

For $n \in \N$, we denote by $[n]$ the set $\{1,\cdots,n\}$. Two lattices are said to be \emph{isomorphic} if there is a structure-preserving mapping between them that can be reversed by an inverse mapping. Such a structure-preserving mapping is called an \emph{isomorphism} between the two lattices.

\subsection{The firm-worker models} \label{sec:basic:model}
    
Let $F$ and $W$ denote two disjoint finite sets of agents, say firms and workers respectively. Associated with each firm $f\in F$ is a \emph{choice function} $\C_f: 2^{W(f)} \rightarrow 2^{W(f)}$ where $W(f)\subseteq W$ is the set of \emph{acceptable} partners of $f$ and $\C_f$ satisfies the property that for every $S\subseteq W(f)$, $\C_f(S)\subseteq S$. Similarly, a choice function $\C_w: 2^{F(w)} \rightarrow 2^{F(w)}$ is associated to each worker $w$. We assume that for every firm-worker pair $(f,w)$, $f\in F(w)$ if and only if $w\in W(f)$. We let $\C_W$ and $\C_F$ denote the collection of firms' and workers' choice functions respectively. A \emph{matching market} (or an instance) is a tuple $(F, W, \C_F, \C_W)$. Following~\cite{alkan2002class}, we define below the properties of \emph{substitutability}, \emph{consistency}, and \emph{cardinal monotonicity (law of aggregate demand)} for choice function $\C_a$  of an agent $a$.

\begin{definition}[Substitutability]
	An agent $a$'s choice function $\C_a$ is substitutable if for any set of partners $S$, $b\in \C_a(S)$ implies that for all $T\subseteq S$, $b\in \C_a(T\cup \{b\})$.
\end{definition}

\begin{definition}[Consistency]
	An agent $a$'s choice function $\C_a$ is consistent if for any sets of partners $S$ and $T$, $\C_a(S) \subseteq T \subseteq S$ implies $\C_a(S) = \C_a(T)$.
\end{definition}

\begin{definition}[Cardinal monotonicity]
    An agent $a$'s choice function $\C_a$ is cardinal monotone if for all sets of partners $S\subseteq T$, we have $|\C_a(S)|\le |\C_a(T)|$.
\end{definition}

Intuitively, substitutability implies that if an agent is selected from a set of candidates, she will also be selected from a smaller subset; consistency is also called ``irrelevance of rejected contracts''; and cardinal monotonicity implies that the size of the image of the choice function is monotone with respect to set inclusion. 

Aizerman and Malishevski~\cite{aizerman1981general} showed that a choice function is substitutable and consistent if and only if it is \emph{path-independent}.

\begin{definition}[Path-independence]
	An agent $a$'s choice function $\C_a$ is path-independent if for any sets of partners $S$ and $T$, $\C_a(S\cup T) = \C_a\big( \C_a(S)\cup T\big)$.
\end{definition}

We next prove a few properties of path-independent choice functions.

\begin{lemma} \label{lem:PI-repeat}
    Let $\C:2^A\rightarrow 2^A$ be a path-independent choice function and let $A_1, A_2\subseteq A$. If $\C(A_1\cup \{a\})=\C(A_1)$ for every $a\in A_2\setminus A_1$, then $\C(A_1\cup A_2)= \C(A_1)$.
\end{lemma}

\begin{proof}
    Assume $A_2\setminus A_1= \{a_1,a_2,\cdots,a_t\}$. Then, by repeated application of the path independence property,
    \begin{align*}
        \C(A_1\cup A_2) &= \C(A_1\cup \{a_1, a_2, \cdots, a_t\})= \C(\C(A_1\cup \{a_1\}) \cup\{a_2, \cdots, a_t\}) \\
        &= \C(\C(A_1) \cup \{a_2, \cdots, a_t\}) = \C(A_1 \cup \{a_2, a_3, \cdots, a_t\}) = \cdots  = \C(A_1).
    \end{align*}
\end{proof}

\begin{corollary} \label{cor:PI-repeat}
    Let $\C:2^A\rightarrow 2^A$ be a path-independent choice function and let $A_1, A_2\subseteq A$. If $a\notin \C(A_1\cup \{a\})$ for every $a\in A_2\setminus A_1$, then $\C(A_1\cup A_2)= \C(A_1)$.
\end{corollary}

\begin{proof}
    By the consistency property of $\C$, $a\notin \C(A_1\cup \{a\})$ implies $\C(A_1\cup \{a\}) = \C(A_1)$. Lemma~\ref{lem:PI-repeat} then applies directly.
\end{proof}

\begin{lemma} \label{lem:dom-P-pre}
    Let $\C: 2^A\rightarrow 2^A$ be a path-independent choice function and let $A_1,A_2\subseteq A$, $a \in A$. Assume $\C(A_1\cup A_2)=A_1$ and $a\in \C(A_1\cup \{a\})$. Then, $a\in \C(A_2\cup \{a\})$.
\end{lemma}

\begin{proof}
    By path-independence, we have that $\C(A_1\cup A_2\cup \{a\}) = \C(\C(A_1\cup A_2)\cup \{a\}) = \C(A_1\cup \{a\})$ and thus $a\in \C(A_1\cup A_2\cup \{a\})$. Also, by path-independence, we have $\C(A_1\cup A_2\cup \{a\}) = \C\big(\C(A_1 \setminus \{a\}) \cup \C(A_2\cup \{a\})\big)$. Since $a\notin \C(A_1 \setminus \{a\})$, it must be that $a\in \C(A_2\cup \{a\})$.
\end{proof}

A \emph{matching} $\mu$ is a mapping from $F\cup W$ to $2^{F\cup W}$ such that for all $w\in W$ and $f\in F$, (1) $\mu(w)\subseteq F(w)$; (2) $\mu(f)\subseteq W(f)$; and (3) $w\in \mu(f)$ if and only if $f\in \mu(w)$. A matching can also be viewed as a collection of firm-worker pairs. That is, $\mu\equiv \{(f,w): f\in F, w\in \mu(f)\}$. Thus, we use $(f,w)\in \mu$, $w\in \mu(f)$, and $f\in \mu(w)$ interchangeably. We say a matching $\mu$ is \emph{individually rational} if for every agent $a$, $\C_a(\mu(a))=\mu(a)$. An acceptable firm-worker pair $(f,w)\notin \mu$ is called a \emph{blocking pair} if $w\in \C_f(\mu(f)\cup \{w\})$ and $f\in \C_w(\mu(w)\cup \{f\})$, and when such pair exists, we say $\mu$ is \emph{blocked by} the pair or the pair blocks $\mu$. A matching $\mu$ is \emph{stable} if it is individually rational and it admits no blocking pairs. If $f$ is matched to $w$ in some stable matching, we say that $(f,w)$ is a \emph{stable pair} and that $f$ (resp.~$w$) is a \emph{stable partner} of $w$ (resp.~$f$). We denote by $\SS(\C_F, \C_W)$ the set of stable matchings in the market $(F,W,\C_F, \C_W)$, and when the market is clear from the context we abbreviate $\SS:=\SS(\C_F, \C_W)$.
 
Alkan~\cite{alkan2002class} showed the following.

\begin{theorem}[\cite{alkan2002class}] \label{thm:alkan}
    Consider a matching market $(F,W,\C_F,\C_W)$ and assume $\C_F$ and $\C_W$ are substitutable, consistent, and cardinal monotone. Then $\SS(\C_F, \C_W)$ is a distributive lattice under the partial order relation $\succeq$ where $\mu_1\succeq \mu_2$ if for all $f\in F$, $\C_f(\mu_1(f))\cup \mu_2(f)) = \mu_1(f)$. The join (denoted by $\vee$) and meet (denoted by $\wedge$) operations of the lattice are defined component-wise. That is, for all $f\in F$:
	\begin{align*}
	(\mu_1\vee \mu_2)(f) &\coloneqq \mu_1(f)\vee \mu_2(f) \coloneqq \C_f(\mu_1(f)\cup \mu_2(f)), \\
	(\mu_1 \wedge \mu_2)(f) &\coloneqq\mu_1(f)\wedge \mu_2(f) \\
	&\coloneqq \big(\big( \mu_1(f) \cup \mu_2(f)\big) \setminus (\mu_1\vee \mu_2)(f) \big) \cup \big(\mu_1(f) \cap \mu_2(f)\big).
	\end{align*}
    Moreover, $\SS(\C_F,\C_W)$ satisfies the \emph{polarity} property: $\mu_1\succeq \mu_2$ if and only if $\C_w(\mu_1(w))\cup \mu_2(w)) = \mu_2(w)$ for every worker $w\in W$.
\end{theorem}

Because of the lattice structure, the firm- and worker-optimal stable matchings are well-defined, and we denote them respectively by $\mu_F$ and $\mu_W$. In addition, Alkan~\cite{alkan2002class} showed two properties, which we call \emph{concordance} (Proposition 7, \cite{alkan2002class}) and \emph{equal-quota} (Proposition 6, \cite{alkan2002class}), satisfied by the family of sets of partners under all stable matchings for every agent $a$. Let $\Phi_a \coloneqq \{\mu(a): \mu\in \SS(\C_F,\C_W)\}$. Then for all $S, T\in \Phi_a$, 
\begin{equation} \tag{concordance} \label{eq:concordance}
    S\cap T \subseteq S\vee T
\end{equation} 
and 
\begin{equation} \tag{equal-quota} \label{eq:equal-quota}
    |S|=|T| \eqqcolon \bar q_a.
\end{equation} 

Instead of cardinal monotonicity, an earlier paper of Alkan~\cite{alkan2001preferences} considers a more restrictive property of choice functions, called \emph{quota-filling}.

\begin{definition}[Quota-filling]
    An agent $a$'s choice function $\C_a$ is quota-filling if there exists $q_a\in \N$ such that for any set of partners $S$, $|\C_a(S)|=\min(q_a, |S|)$. We call $q_a$ the \emph{quota} of agent $a$.
\end{definition}

Intuitively, quota-filling means that an agent has a number of positions and she tries to fill these positions as many as possible. Note that quota-filling implies cardinal monotonicity. Let $q_a$ denote the quota of each agent $a\in F\cup W$. 

Our results from Section~\ref{sec:affine-representability} assume path-independence (i.e., substitutability and consistency) and cardinal monotonicity. In Section~\ref{sec:algo}, we will restrict our model by replacing cardinal monotonicity with quota-filling for one side of the market. These two models are what we call the \textsc{CM-Model} and the \textsc{CM-QF-Model}, respectively.

\subsection{MC-representation for path-independent choice functions}\label{sec:mc-representation}

We now introduce an alternative, equivalent description of choice functions for the model studied in this paper that we will use in examples throughout the paper, and investigate more in detail in Section~\ref{sec:representation-and-algo}.

Aizerman and Malishevski~\cite{aizerman1981general} showed that a choice function $\C_a$ is path-independent if and only if there exists a finite sequence of $p(\C_a)\in \N$ preference relations over acceptable partners, denoted as $\{\ge_{a,i}\}_{i\in [p(\C_a)]}$ indexed by $i$, such that for every subset of acceptable partners $S$, $\C_a(S) = \cup_{i\in [p(\C_a)]} \{x^*_{a,i}\}$, where $x^*_{a,i}=\max(S,\ge_{a,i})$ is the maximum element\footnote{If $S=\emptyset$, then $\max(S,\ge_{a,i})$ is defined to be $\emptyset$.} of $S$ according to $\ge_{a,i}$. We call this sequence of preference relations the \emph{Maximizer-Collecting representation} (MC-representation) of choice function $\C_a$. Note that for distinct $i_1,i_2\in [p(\C_a)]$, it is possible to have $x^*_{a,i_1}= x^*_{a,i_2}$.

Conceptually, one can view the MC-representation as follows: a firm is a collection of \emph{positions}, each of which has its own preference relation; a worker is a collection of \emph{personas}, each of whom also has his or her own preference relation. Each firm hires the best candidate for each position, and the same candidate can be hired for two positions if (s)he is the best for both. A symmetric statement holds for workers and personas.

\begin{remark}
    We would like to again highlight the differences between MC-representation of choice functions and the representation, in the \textsc{MM-Model}, by a single preference list $\ge_a$ together with a quota $q_a$. In particular, in the \textsc{MM-Model}, $\C_a(S)= \cup_{i\in [q_a]} \{\tilde x_{a,i}\}$, where $\tilde x_{a,i}=\max(S\setminus \{\tilde x_{a,j}: j\in [i-1]\}, \ge_a)$. Note that here for distinct $i_1,i_2\in [q_a]$, $\tilde x_{a,i_1} \neq \tilde x_{a,i_2}$ unless both are $\emptyset$.
\end{remark}

\section{Affine representability of the stable matching lattice} \label{sec:affine-representability}

In this section, we show that the distributive lattice of stable matchings in the model by~\cite{alkan2002class} is affinely representable. An algorithm to construct an affine representation is given in Section~\ref{sec:algo} where we additionally impose the quota-filling property upon choice functions of agents in one side of the markets. The proof of this section proceeds as follows. First, we show in Section~\ref{sec:sm-ring-of-sets} that the lattice of stable matchings $(\SS, \succeq)$ is isomorphic to a lattice $(\P, \subseteq)$ belonging to a special class, that is called \emph{ring of sets}. In Section~\ref{sec:affine-ring-of-sets}, we then show that ring of sets are always affinely representable.  In Section~\ref{sec:compact-representation}, we show a poset $(\Pi,\succeq^\star)$ representing $(\SS, \succeq)$. Lastly, in Section~\ref{sec:affine-repr-final}, we show how to combine all those results and ``translate'' the affine representability of $(\P, \subseteq)$ to  the affine representability of $(\SS, \succeq)$, concluding the proof.

\subsection{Isomorphism between the stable matching lattice and a ring of sets} \label{sec:sm-ring-of-sets}

A family $\H=\{H_1, H_2, \cdots, H_k\}$ of subsets of a \emph{base set} $B$ is a \emph{ring of sets} over $B$ if $\H$ is closed under set union and set intersection~\cite{birkhoff1937rings}. Note that a ring of sets is a distributive lattice with the partial order relation $\subseteq$, and the join and meet operations corresponds to set intersection and set union, respectively. An example of a ring of sets is given in Example~\ref{eg:ring-of-sets}.

In this and the following section, we fix a matching market $(F,W,\C_F,\C_W)$ and assume that $\C_F$ and $\C_W$ are path-independent and cardinal monotone (i.e., the framework of~\cite{alkan2002class}). Let $\phi(a)$ denote the set of stable partners of agent $a$. That is, $\phi(a)\coloneqq \{b: b\in \mu(a) \text{ for some } \mu\in \SS\}$. For a stable matching $\mu$, define $$P_f(\mu)\coloneqq \{w\in \phi(f): w\in \C_f(\mu(f)\cup \{w\})\},$$ and define the \emph{P-set} of $\mu$ as $$P(\mu)\coloneqq \{(f,w): f\in F, w\in P_f(\mu)\}.$$

The goal of this section is to show the following theorem, which gives a representation of the stable matching lattice as a ring of sets. Let $\P(\C_F, \C_W)$ denote the set $\{P(\mu): \mu\in \SS(\C_F, \C_W)\}$, and we often abbreviate $\P\coloneqq \P(\C_F, \C_W)$.

\begin{theorem} \label{thm:iso-ros}
    Assume $\C_F$ and $\C_W$ are path-independent and cardinal monotone. Then, 
    \begin{enumerate}
        \item[(1)] the mapping $P:\SS\rightarrow \P$ is a bijection; 
        \item[(2)] $(\P, \subseteq)$ is isomorphic to $(\SS, \succeq)$. That is, for two stable matchings $\mu_1, \mu_2\in \SS$, we have $\mu_2\succeq \mu_1$ if and only if $P(\mu_2)\subseteq P(\mu_1)$. Moreover, $P(\mu_1\vee \mu_2) =P(\mu_1)\cap P(\mu_2)$ and $P(\mu_1\wedge \mu_2) =P(\mu_1)\cup P(\mu_2)$. In particular, $(\P, \subseteq)$ is a ring of sets over the base set $\{(f,w): f\in F, w\in \phi(f)\}$.
    \end{enumerate}
\end{theorem}

\begin{remark} \label{rmk:P-set-defn}
    An isomorphism between the lattice of stable matchings and a ring of sets (also called P-set) is proved in the \textsc{SM-Model} by Gusfield and Irving~\cite{gusfield1989stable} as well. However, they define $P(\mu) := \{(f,w): f\in F, w\ge_f \mu(f)\}$, hence including firm-worker pairs that are not stable. We show in Example~\ref{eg:p-set} that in our more general setting, the P-set by~\cite{gusfield1989stable} is not a ring of sets. As a consequence, while in their model the construction of the P-set for a given stable matching is immediate, in ours it is not, since we need to know first which pairs are stable.
\end{remark}

\begin{lemma} \label{lem:dom-P}
    Let $\mu_1$ and $\mu_2$ be two stable matchings such that $\mu_2\succeq \mu_1$. Then, $P_f(\mu_2)\subseteq P_f(\mu_1)$ for every firm $f$.
\end{lemma}

\begin{proof}
    Since $\mu_2\succeq \mu_1$, we have that $\C_f(\mu_2(f) \cup \mu_1(f))=\mu_2(f)$. The claim then follows from Lemma~\ref{lem:dom-P-pre}.
\end{proof}

\begin{lemma} \label{lem:exist-dom-P}
    Let $\mu_1$ be a stable matching such that $w\in P_f(\mu_1)$ for some firm $f$ and worker $w$. Then, there exists a stable matching $\mu_2$ such that $\mu_2\succeq \mu_1$ and $w\in \mu_2(f)$.
\end{lemma}

\begin{proof}
    By definition of $P_f(\mu_1)$, we know there exists a stable matching $\mu_1'$ such that $w\in \mu_1'(f)$. Let $\mu_2\coloneqq \mu_1\vee \mu_1'$. We want to show that $w\in \mu_2(f)$. If $w\in \mu_1(f)$, then the claim follows due to the~\ref{eq:concordance} property. So assume $w\notin \mu_1(f)$ and also assume by contradiction that $w\notin \mu_2(f)$. Then, we must have $w\in (\mu_1\wedge \mu_1')(f)$ by definition of the meet. Since $\mu_1\succeq \mu_1\wedge \mu_1'$, we have $\C_f(\mu_1(f) \cup (\mu_1\wedge \mu_1')(f)) = \mu_1(f)$. However, applying path-independence and consistency, we have
    \begin{align*} 
         \C_f(\mu_1(f) \cup (\mu_1\wedge \mu_1')(f)) &= \C_f\big( \C_f(\mu_1(f) \cup (\mu_1\wedge \mu_1')(f) \setminus \{w\})\cup \{w\}\big) \\
         &= \C_f(\mu_1(f) \cup \{w\}) \neq \mu_1(f),
    \end{align*}
	which is a contradiction.
\end{proof}

\begin{lemma} \label{lem:exist-dom-P-middle}
    Let $\mu_1$ and $\mu_2$ be two stable matchings such that $\mu_2\succeq \mu_1$. Assume $w\in P_f(\mu_1)\setminus P_f(\mu_2)$ for some firm $f$. Then, there exists a stable matching $\bar\mu_1$ with $\mu_2\succeq \bar\mu_1\succeq \mu_1$ such that $w\in \bar\mu_1(f)$.
\end{lemma}

\begin{proof}
    By Lemma~\ref{lem:exist-dom-P}, there exists a stable matching $\bar\mu_2 \succeq \mu_1$ such that $w\in \bar\mu_2(f)$. Let $\bar\mu_1\coloneqq \bar\mu_2 \wedge \mu_2$ and we claim that $\bar\mu_1$ is the desired matching. First, by definition of meet, we have $\mu_2\succeq \bar\mu_1\succeq \mu_1$. Since $w\notin P_f(\mu_2)$, by the contrapositive of the substitutability property, we have $w\notin \C_f(\mu_2(f)\cup \bar\mu_2(f))$, which implies that $w\notin (\mu_2\vee \bar\mu_2)(f)$. Therefore, $w\in \bar\mu_1(f)$, again by the definition of meet.
\end{proof}

\begin{lemma} \label{lem:1to1-ros}
    Let $\mu_1$ and $\mu_2$ be two stable matchings. Then, $$P(\mu_1\vee \mu_2) = P(\mu_1) \cap P(\mu_2) \quad \textup{and} \quad P(\mu_1\wedge \mu_2) = P(\mu_1) \cup P(\mu_2).$$
\end{lemma}

\begin{proof}
    Fix a firm $f$, and we want to show $P_f(\mu_1\vee \mu_2) = P_f(\mu_1) \cap P_f(\mu_2)$ and $P_f(\mu_1\wedge \mu_2) = P_f(\mu_1) \cup P_f(\mu_2)$. If $\mu_1(f)=\mu_2(f)$, then the claim is obviously true. Thus, for the following, we assume $\mu_1(f)\neq \mu_2(f)$. We first show that $P_f(\mu_1\vee \mu_2)\subseteq P_f(\mu_1)\cap P_f(\mu_2)$. Since $\mu_1\vee \mu_2 \succeq \mu_1, \mu_2$, the claim follows from Lemma~\ref{lem:dom-P}. Next, we show that $P_f(\mu_1\vee \mu_2)\supseteq P_f(\mu_1)\cap P_f(\mu_2)$. If $P_f(\mu_1)\cap P_f(\mu_2)=\emptyset$, then the claim follows trivially. So we assume $P_f(\mu_1)\cap P_f(\mu_2)\neq \emptyset$ and let $w\in P_f(\mu_1)\cap P_f(\mu_2)$. By Lemma~\ref{lem:exist-dom-P}, there exists a stable matching $\bar\mu_1$ such that $\bar\mu_1 \succeq \mu_1$ and $w\in \bar\mu_1(f)$. Similarly, there exists a stable matching $\bar\mu_2$ such that $\bar\mu_2 \succeq \mu_2$ and $w\in \bar\mu_2(f)$. Consider the stable matching $\bar\mu_1 \vee \bar\mu_2$. Because of the~\ref{eq:concordance} property, $w\in (\bar\mu_1 \vee \bar\mu_2)(f)$. In addition, by transitivity of $\succeq$, we have that $\bar\mu_1 \vee \bar\mu_2 \succeq \mu_1 ,\mu_2$ and thus $\bar\mu_1 \vee \bar\mu_2 \succeq \mu_1\vee \mu_2$ by minimality of $\mu_1\vee\mu_2$. Hence, by Lemma~\ref{lem:dom-P}, $w\in P_f(\mu_1 \vee\mu_2)$. This concludes the first part of the thesis.

    For the second half, we first show $P_f(\mu_1\wedge \mu_2)\subseteq P_f(\mu_1)\cup P_f(\mu_2)$. Let $w\notin P_f(\mu_1)\cup P_f(\mu_2)$, we want to show that $w\notin P_f(\mu_1\wedge \mu_2)$. Assume by contradiction that $w\in P_f(\mu_1\wedge \mu_2)$. $w\notin P_f(\mu_1)\cup P_f(\mu_2)$ implies $w\notin \mu_1(f)$ and $w\notin \mu_2(f)$ and thus, $w\notin (\mu_1\wedge \mu_2)(f)$. By Lemma~\ref{lem:exist-dom-P-middle}, for both $i\in \{1,2\}$, there exists a stable matching $\bar\mu_i$ such that $\mu_i\succeq \bar\mu_i \succeq \mu_1\wedge\mu_2$ and $w\in \bar\mu_i(f)$. Note that $\mu_1\wedge \mu_2 \succeq \bar\mu_1 \wedge \bar\mu_2 \succeq \mu_1\wedge\mu_2$, where the first relation holds because $\mu_i\succeq \bar \mu_i$ for both $i\in \{1,2\}$, and the second relation holds because $\bar \mu_1, \bar \mu_2 \succeq \mu_1\wedge \mu_2$. Hence, $\bar\mu_1 \wedge \bar\mu_2 = \mu_1\wedge\mu_2$. However, by the~\ref{eq:concordance} property over $\bar\mu_1$ and $\bar\mu_2$, we have $w\in (\bar\mu_1 \wedge \bar\mu_2)(f) = (\mu_1\wedge \mu_2)(f)$, which is a contradiction. 
    
    Lastly, we show $P_f(\mu_1\wedge \mu_2)\supseteq P_f(\mu_1)\cup P_f(\mu_2)$. Let $w\in P_f(\mu_1)\cup P_f(\mu_2)$ and wlog assume $w\in P_f(\mu_1)$. Since $\mu_1\succeq \mu_1\wedge \mu_2$, by Lemma~\ref{lem:dom-P}, we have $w\in P_f(\mu_1\wedge \mu_2)$.
\end{proof}

\begin{lemma} \label{lem:dom-P-uniq}
    Let $\mu_1$ and $\mu_2$ be two stable matchings such that $\mu_2\succ \mu_1$ and assume that $\mu_1(f)\neq \mu_2(f)$ for some $f\in F$. Then, $P_f(\mu_1)\neq P_f(\mu_2)$.
\end{lemma}

\begin{proof}
    Assume by contradiction that $P_f(\mu_1)=P_f(\mu_2)$. Let $w\in \mu_1(f)\setminus \mu_2(f)$. $w$ exists because $\mu_1(f)\neq\mu_2(f)$ and $|\mu_1(f)|=|\mu_2(f)|$ due to the~\ref{eq:equal-quota} property. Since the stable matching lattice $(\SS, \succeq)$ has the polarity property as shown in Theorem~\ref{thm:alkan}, we have that $\C_w(\mu_1(w)\cup \mu_2(w))=\mu_1(w)$ and thus, by substitutability, we have $f\in \C_w(\mu_2(w)\cup\{f\})$. On the other hand, $w\in \mu_1(f)$ implies that $w\in P_f(\mu_1)= P_f(\mu_2)$. Since $w\notin \mu_2(f)$, this means that $(f,w)$ is a blocking pair of $\mu_2$, which contradicts the stability assumption.
\end{proof}

\begin{lemma} \label{lem:P-injective}
    Let $\mu_1$ and $\mu_2$ be two distinct stable matchings and assume that $\mu_1(f)\neq \mu_2(f)$ for some $f\in F$. Then, $P_f(\mu_1)\neq P_f(\mu_2)$.
\end{lemma}

\begin{proof}
    Assume by contradiction that $P_f(\mu_1)=P_f(\mu_2)$. Then, we have $P_f(\mu_1\vee \mu_2)= P_f(\mu_1\wedge \mu_2)$ by Lemma~\ref{lem:1to1-ros}. However, $\mu_1(f)\neq \mu_2(f)$ implies that $(\mu_1\vee\mu_2)(f) \neq (\mu_1\wedge\mu_2)(f)$, which contradicts Lemma~\ref{lem:dom-P-uniq} since $\mu_1\vee\mu_2 \succ \mu_1\wedge \mu_2$.
\end{proof}

\begin{proof}[Proof of Theorem~\ref{thm:iso-ros}] 
    For (1), note that the mapping $P$ is onto by definition. It is therefore a bijection since it is also injective as shown in Lemma~\ref{lem:P-injective}. Next, we show (2). One direction of the first statement is shown in Lemma~\ref{lem:dom-P}. Conversely, if $P(\mu_2)\subseteq P(\mu_1)$, then by Lemma~\ref{lem:1to1-ros}, $P(\mu_1\vee \mu_2) = P(\mu_1)\cap P(\mu_2) = P(\mu_2)$. Hence, by Lemma~\ref{lem:P-injective}, we have $\mu_1\vee \mu_2 = \mu_2$ and thus, $\mu_2 \succeq \mu_1$. The second statement of (2) follows from Lemma~\ref{lem:1to1-ros}. The third follows from the second and the fact that stable matchings form a distributive lattice (Theorem~\ref{thm:alkan}).
\end{proof}

\begin{example} \label{eg:p-set}
    Consider the following instance with $4$ firms and $5$ workers. Agents' choice functions are given below in their MC-representations. For instance, the first position of firm $f_1$ prefers $w_1$ the most and prefers $w_2$ the least. 
    \[\begin{array}{ccl}
        f_1: & \ge_{f_1,1}: & w_1\ w_5\ w_3\ w_4\ w_2 \\
        & \ge_{f_1,2}: & w_2\ w_5\ w_4\ w_3\ w_1 \\
        & \ge_{f_1,3}: & w_1\ w_2\ w_3\ w_4\ w_5 \\
        f_2: & \ge_{f_2,1}: & w_4\ w_2\ w_1\ w_3\ w_5 \\
        f_3: & \ge_{f_3,1}: & w_3\ w_1\ w_2\ w_4\ w_5 \\
        f_4: & \ge_{f_4,1}: & w_5\ w_1\ w_2\ w_3\ w_4
    \end{array} \qquad \qquad \qquad 
    \begin{array}{ccl}
        w_1: & \ge_{w_1,1}: & f_3\ f_1\ f_2\ f_4 \\
        w_2: & \ge_{w_2,1}: & f_2\ f_1\ f_3\ f_4 \\
        w_3: & \ge_{w_3,1}: & f_1\ f_3\ f_2\ f_4 \\
        w_4: & \ge_{w_4,1}: & f_1\ f_2\ f_3\ f_4 \\
        w_5: & \ge_{w_5,1}: & f_4\ f_1\ f_2\ f_3 \\
        & 
    \end{array}\]
    There are four stable matchings in this instance:
    \begin{align*}
        \mu_F &= (f_1, w_1), (f_1, w_2), (f_2, w_4), (f_3, w_3), (f_4, w_5); \\ 
        \mu_1 &= (f_1, w_1), (f_1, w_4), (f_2, w_2), (f_3, w_3), (f_4, w_5); \\
        \mu_2 &= (f_1, w_2), (f_1, w_3), (f_2, w_4), (f_3, w_1), (f_4, w_5); \\
        \mu_W &= (f_1, w_3), (f_1, w_4), (f_2, w_2), (f_3, w_1), (f_4, w_5).
    \end{align*}
    Note that $\mu_1$ and $\mu_2$ are not comparable. Their corresponding P-sets are
    \begin{equation*}
    \resizebox{.97\textwidth}{!}{$
    \begin{aligned}
        P(\mu_F) &= (f_1, w_1), (f_1, w_2), (f_2, w_4), (f_3, w_3), (f_4, w_5); \\ 
        P(\mu_1) &= (f_1, w_1), (f_1, w_2), (f_1, w_4), (f_2, w_2), (f_2, w_4), (f_3, w_3), (f_4, w_5); \\
        P(\mu_2) &= (f_1, w_1), (f_1, w_2), (f_1, w_3), (f_2, w_4), (f_3, w_1), (f_3, w_3), (f_4, w_5); \\
        P(\mu_W) &= (f_1, w_1), (f_1, w_2), (f_1, w_3), (f_1, w_4), (f_2, w_2), (f_2, w_4), (f_3, w_1), (f_3, w_3), (f_4, w_5).
    \end{aligned}$
    }
    \end{equation*}
    One can easily check that the claims given in Lemma~\ref{lem:1to1-ros} are true. Note that if we follow the definition given in~\cite{gusfield1989stable} and include the pair $(f_1,w_5)$ in $P(\mu_1)$ and $P(\mu_2)$. Then Lemma~\ref{lem:1to1-ros} no longer holds since $w_5\notin P_{f_1}(\mu_F) = P_{f_1}(\mu_1 \vee \mu_2)$. 
\end{example}

\subsection{Affine representability of ring of sets via the poset of minimal differences} \label{sec:affine-ring-of-sets}

We now recall (mostly known) facts about posets representing ring of sets, and observe that the \emph{affine} representability of ring of sets easily follows from those. 

Fix a ring of sets $(\H,\subseteq)$ over a base set $B$, and let $H_0$ and $H_z$ denote respectively the unique minimal and maximal elements of $\H$. That is, for all $H\in \H$, we have $H_0\subseteq H\subseteq H_z$. For $a\in H_z$, let $H(a)$ denote the unique inclusion-wise minimal set among all sets in $\H$ that contain $a$, where uniqueness follows from the fact that $\H$ is closed under set intersection. That is, $$H(a)\coloneqq \bigcap \{H\in \H: a\in H\}.$$ In addition, define the set $\I(\H)$ of the \emph{irreducible} elements of $\H$ as follows $$\I(\H) \coloneqq \{H\in \H: \exists \; a\in H_z \text{ s.t. } H=H(a)\}.$$ Since $\I(\H)$ is a subset of $\H$, we can view $\I(\H)$ as a poset under the set containment relation.

For $H\in \I(\H)$, let $K(H)\coloneqq \{a\in H_z: H(a)=H\}$ denote the \emph{centers} of $H$. Note that $K(H_0) = H_0$. Define $\D(\H)$ as the set of centers of irreducible elements of $\H$ without the set $H_0$. Formally, $$\D(\H)\coloneqq \{K(H): H\in \I(\H), H\neq H_0\}.$$ 

Immediately from the definition of centers, we obtain the following.

\begin{lemma}\label{lem:center-uniq}
    Let $a\in B$. There is at most one $K_1\in \D(\H)$ such that $a\in K_1$. In particular, $|\D(\H)|=O(|B|)$.
\end{lemma}

For $K_1\in \D(\H)$, let $I(K_1)$ denote the irreducible element from $\I(\H)$ such that $K(I(K_1))= K_1$. Let $\sqsupseteq$ be a partial order over the set $\D(\H)$ that is inherited from the set containment relation of the poset $\I(\H)$. That is, for $K_1,K_2\in \D(\H)$, we have $K_1 \sqsupseteq K_2$ if and only if $I(K_1)\subseteq I(K_2)$.

\begin{theorem}[\cite{birkhoff1937rings}] \label{thm:birkhoff}
    Let $(\H,\subseteq)$ be a ring of sets. Then, $(\D(\H),\sqsupseteq)$ is a representation poset for $(\H,\subseteq)$ with representation function $\psi_\H$, where $\psi_\H^{-1}(\bar\D) = \bigcup \{K_1: K_1\in \bar\D\} \cup H_0$ for any upper set $\bar\D$ of $(\D(\H), \sqsupseteq)$, and $H_0$ is the minimal element of $\H$. 
\end{theorem}

Lemma~\ref{lem:center-uniq} and Theorem~\ref{thm:birkhoff} directly imply the following.

\begin{theorem}\label{thm:affine-representaibility-ring-of-sets}
    Let $(\H,\subseteq)$ be a ring of sets over base set $B$. Then, $(\D(\H), \sqsupseteq)$ affinely represents $(\H,\subseteq)$ via affine function $g(u)=Au+x^0$, where $x^0$ is the characteristic vector of the minimal element of $\H$, and $A\in \{0,1\}^{B\times \D(\H)}$ has columns $\Chi^{K_1}$ for each $K_1\in \D(\H)$. Moreover, $A$ has full column rank.
\end{theorem}

\begin{proof}
    Because of the representation function $\psi_\H$ given in Theorem~\ref{thm:birkhoff}, it is clear that $g(\Chi^U) = \Chi^{\psi_\H^{-1}(U)}$ for every upper set $U\in \U((\D(\H), \sqsupseteq))$. Note that every row of $A$ has at most one non-zero entry due to Lemma~\ref{lem:center-uniq}, and every column of $A$ contains at least one non-zero entry by definition. Therefore, $A$ has full column rank.
\end{proof}

\begin{lemma} \label{lem:unq-upperset}
    Let $(\H,\subseteq)$ be a ring of sets with minimal element $H_0$, and let $H\in \H$. If $H= \bigcup \{K_1: K_1\in \bar\D\} \cup H_0$ for some subset $\bar\D$ of $\D(\H)$, then $\bar\D$ is an upper set of $(\D(\H),\sqsupseteq)$.
\end{lemma}

\begin{proof}
    By Lemma~\ref{lem:center-uniq}, there is at most one subset of ${\cal D}({\cal H})$ whose union of the elements together with $H_0$ gives $H$. On the other hand, Theorem~\ref{thm:birkhoff} implies that there exists one such subset which is also an upper set of $(\D(\H),\sqsupseteq)$. The claim follows thereafter.
\end{proof}

We elucidate in Example~\ref{eg:ring-of-sets} the definitions and facts above.

\begin{example} \label{eg:ring-of-sets}
    Consider the Hasse diagram of the ring of sets given in Figure~\ref{fig:ros-rep} with base set $B=\{a,b,c,d,e,f\}$ and $\H=\{H_1, \cdots, H_7\}$. The irreducible elements are $\I(\H) = \{H_1, H_2, H_3, H_5, H_7\}$. The center(s) of each irreducible element is underlined, and $\D(\H)= \{\{b\}, \{c\}, \{d,e\}, \{f\}\}$. The poset of $\D(\H)$ is represented in Figure~\ref{fig:pos-md}. The upper sets of poset $\D(\H)$ corresponding to $H_1, \cdots, H_7$ in the exact order are: $\emptyset$; $\{\{b\}\}$; $\{\{c\}\}$; $\{\{b\}, \{c\}\}$; $\{\{c\}, \{d,e\}\}$; $\{\{b\}, \{c\}, \{d,e\}\}$; and $\{\{b\}, \{c\}, \{d,e\}, \{f\}\}$. Affine function is $g(u)=Au+x^0$ with $(x^0)^\intercal=(1,0,0,0,0,0)$ and matrix $A$ given below in Figure~\ref{fig:matrix-A-ros}. Note that columns of $A$ correspond to $\{b\}$, $\{c\}$, $\{d,e\}$, $\{f\}$ in this order. 
\end{example}
    
\begin{figure}[ht]
    \begin{subfigure}{.44\textwidth}
    \centering
    \begin{tikzpicture}[scale=.7]
        \node[rectangle, draw] (0) at (2,4) {\ul{$a$}};
        \node[right= 0mm of 0] {\small $H_1$};
        \node[rectangle, draw] (1) at (0,3) {$a$ \ul{$b$}};
        \node[left= 0mm of 1] {\small $H_2$};
        \node[rectangle, draw] (2) at (4,3) {$a$ \ul{$c$}};
        \node[right= 0mm of 2] {\small $H_3$};
        \node[rectangle, draw] (3) at (2,2) {$a \; b \; c$};
        \node[left= 0mm of 3] {\small $H_4$};
        \node[rectangle, draw] (4) at (6,2) {$a \; c$ \ul{$d \; e$}};
        \node[right= 0mm of 4] {\small $H_5$};
        \node[rectangle, draw] (5) at (4,1) {$a \; b \; c \; d \; e$};
        \node[left= 0mm of 5] {\small $H_6$};
        \node[rectangle, draw] (6) at (4,-.5) {$a \; b \; c \; d \; e$ \ul{$f$}};
        \node[left= 0mm of 6] {\small $H_7$};
        \draw[] (0) -- (1);
        \draw[] (0) -- (2);
        \draw[] (1) -- (3);
        \draw[] (2) -- (3);
        \draw[] (2) -- (4);
        \draw[] (3) -- (5);
        \draw[] (4) -- (5);
        \draw[] (5) -- (6);
    \end{tikzpicture}
    \caption{$(\H, \subseteq)$} \label{fig:ros-rep}
    \end{subfigure}
    \begin{subfigure}{.29\textwidth}
    \centering
    \begin{tikzpicture}[scale=.7]
        \node[] at (0,4.7) {};
        \node[] at (0,0) {};
        \node[circle, draw, scale=.5] (1) at (0,3) {};
        \node[] at (-.4, 3) {$b$};
        \node[circle, draw, scale=.5] (2) at (2,3) {};
        \node[] at (2.4, 3) {$c$};
        \node[circle, draw, scale=.5] (3) at (2,2) {};
        \node[] at (2.5, 2) {$d \; e$};
        \node[circle, draw, scale=.5] (4) at (1,1) {};
        \node[] at (.6, 1) {$f$};
        \draw[] (1) -- (4);
        \draw[] (2) -- (3);
        \draw[] (3) -- (4);
    \end{tikzpicture}
    \caption{$(\D(\H),\sqsupseteq)$} \label{fig:pos-md}
    \end{subfigure}
    \begin{subfigure}{.24\textwidth}
    \centering
    \begin{tikzpicture}[scale=.8]
        \node[] at (0,2.3) {$A=\begin{pmatrix} 
        0 & 0 & 0 & 0 \\ 1 & 0 & 0 & 0 \\ 0 & 1 & 0 & 0 \\ 0 & 0 & 1 & 0 \\ 0 & 0 & 1 & 0 \\ 0 & 0 & 0 & 1 \\ 
        \end{pmatrix}$};
        \node[] at (0,4.5) {};
        \node[] at (0,0.3) {};
    \end{tikzpicture}
    \caption{Matrix $A$} \label{fig:matrix-A-ros}
    \end{subfigure}
    \caption{Hasse diagrams of a ring of sets and its representation poset, as well as the matrix $A$ for affine representability for Example~\ref{eg:ring-of-sets}.} 
\end{figure}
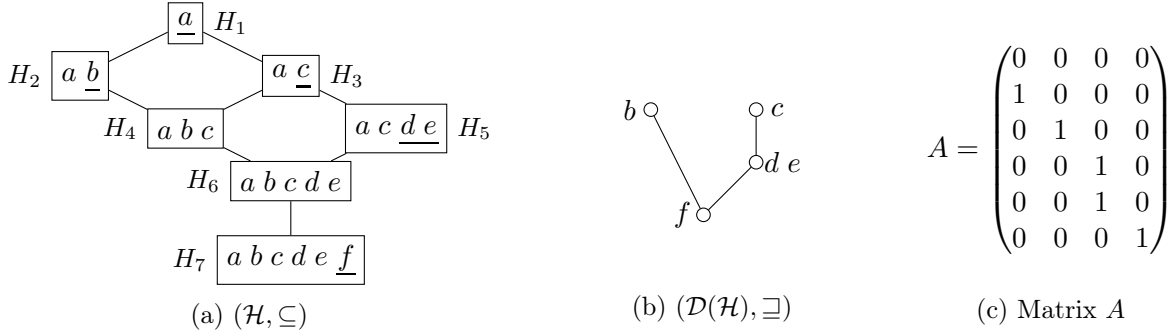

Alternatively, one can view $\D(\H)$ as the set of \emph{minimal differences} between elements of $\H$. The following lemma is established directly from Lemma 2.4.3 and Corollary 2.4.1 of \cite{gusfield1989stable}.

\begin{lemma} \label{lem:center-md} 
    $\D(\H)= \{H\setminus H': H' \text{ is an immediate predecessor of }H \text{ in } (\H, \subseteq)\}$.
\end{lemma}

A direct consequence of Lemma~\ref{lem:center-uniq} and Lemma~\ref{lem:center-md} is the following.

\begin{lemma} \label{lem:md-imply-immediate}
    Let $H', H\in \H$. If $H'\subseteq H$ and $H\setminus H'\in \D(\H)$, then $H'$ is an immediate predecessor of $H$ in $(\H,\subseteq)$.
\end{lemma}

\begin{proof}
    Let $K_1\coloneqq H\setminus H'$. Assume by contradiction that there exists $\bar H\in \H$ with $H'\subsetneq \bar H\subsetneq H$. Then, because of Lemma~\ref{lem:center-md}, there exists a center $K_2\in \D(\H)$ such that $\emptyset\neq K_2 \subsetneq K_1$. However, this contradicts Lemma~\ref{lem:center-uniq}. 
\end{proof}

\subsection{Representation of $({\cal S,\succeq})$ via the poset of rotations} \label{sec:compact-representation}

As discussed in Section~\ref{sec:affine-ring-of-sets}, the poset $(\D(\P),\sqsupseteq)$ associated with $(\P,\subseteq)$ provides a compact representation of $(\P,\subseteq)$ and can be used to \emph{reconstruct} $\P$ via Theorem~\ref{thm:affine-representaibility-ring-of-sets}. In this section, we show how to associate with $(\SS, \succeq)$ a poset that is isomorphic to $(\D(\P), \sqsupseteq)$, which can be used to reconstruct $\SS$. The precise statement is give in Theorem~\ref{thm:rp-isom-birkhoff} below. 

For $\mu, \mu'\in \SS$, with $\mu'$ being an immediate predecessor of $\mu$ in the stable matching lattice, let 
$$\rho^{+} (\mu', \mu) =\{(f,w): f\in F, w\in \mu(f)\setminus \mu'(f)\}$$ 
and 
$$\rho^{-} (\mu', \mu) =  \{(f,w): f\in F, w\in \mu'(f)\setminus \mu(f)\}.$$ 
Note that by definition, $$\mu= \mu'\triangle \rho^-(\mu', \mu) \triangle \rho^+(\mu', \mu) = \mu' \setminus \rho^-(\mu', \mu) \cup \rho^+(\mu', \mu).$$
We call $\rho(\mu', \mu)\coloneqq (\rho^{+} (\mu', \mu), \rho^{-} (\mu', \mu))$ a \emph{rotation} of $(\SS, \succeq)$. Let $\Pi(\SS)$ denote the set of rotations of $(\SS, \succeq)$. That is, $$\Pi(\SS) \coloneqq \{\rho(\mu',\mu): \mu' \textup{ is an immediate predecessor of } \mu \textup{ in } (\SS, \succeq)\}.$$ 

\begin{remark} 
    It is interesting to compare rotations in the current model~\cite{alkan2002class} with the analogous concept in the \textsc{MM-Model}. While in the latter case, rotations are simple cycles in the associated bipartite graph of agents~\cite{baiou2000many}, this may not be the case for our model, as Example~\ref{ex:rotations} shows.
\end{remark}

\begin{example}\label{ex:rotations}
    Consider the two stable matchings $\mu'$ and $\bar\mu$ shown in Example~\ref{eg:break-marriage}, where $\mu'$ is an immediate predecessor of $\bar\mu$. As  shown in Figure~\ref{fig:eg-matching-ncycle}, their symmetric difference is not a simple cycle. In Figure~\ref{fig:symmetric-difference}, solid lines are edges from $\mu'$ and dashed lines are those from $\bar\mu$.
\end{example}

\begin{figure}[ht]
    \centering
    \begin{subfigure}{.33\textwidth}
    \centering
    \begin{tikzpicture}[scale=.7]
        \node[rectangle, draw] (f1) at (0,4) {$f_1$};
        \node[rectangle, draw] (f2) at (0,3) {$f_2$};
        \node[rectangle, draw] (f3) at (0,2) {$f_3$};
        \node[rectangle, draw] (f4) at (0,1) {$f_4$};
        \node[rectangle, draw] (w1) at (3,4) {$w_1$};
        \node[rectangle, draw] (w2) at (3,3) {$w_2$};
        \node[rectangle, draw] (w3) at (3,2) {$w_3$};
        \node[rectangle, draw] (w4) at (3,1) {$w_4$};
        \draw[] (f1) -- (w2);
        \draw[] (f1) -- (w4);
        \draw[] (f2) -- (w1);
        \draw[] (f2) -- (w2);
        \draw[] (f3) -- (w3);
        \draw[] (f3) -- (w4);
        \draw[] (f4) -- (w1);
        \draw[] (f4) -- (w3);
    \end{tikzpicture}
    \caption{stable matching $\mu'$}
    \end{subfigure}%
    \begin{subfigure}{.33\textwidth}
    \centering
    \begin{tikzpicture}[scale=.7]
        \node[rectangle, draw] (f1) at (0,4) {$f_1$};
        \node[rectangle, draw] (f2) at (0,3) {$f_2$};
        \node[rectangle, draw] (f3) at (0,2) {$f_3$};
        \node[rectangle, draw] (f4) at (0,1) {$f_4$};
        \node[rectangle, draw] (w1) at (3,4) {$w_1$};
        \node[rectangle, draw] (w2) at (3,3) {$w_2$};
        \node[rectangle, draw] (w3) at (3,2) {$w_3$};
        \node[rectangle, draw] (w4) at (3,1) {$w_4$};
        \draw[] (f1) -- (w3);
        \draw[] (f1) -- (w4);
        \draw[] (f2) -- (w1);
        \draw[] (f2) -- (w4);
        \draw[] (f3) -- (w2);
        \draw[] (f3) -- (w3);
        \draw[] (f4) -- (w1);
        \draw[] (f4) -- (w2);
    \end{tikzpicture}
    \caption{stable matching $\bar\mu$}
    \end{subfigure}%
    \begin{subfigure}{.33\textwidth}
    \centering
    \begin{tikzpicture}[scale=.7]
        \node[rectangle, draw] (f1) at (0,0) {$f_1$};
        \node[rectangle, draw] (f2) at (2,0) {$f_2$};
        \node[rectangle, draw] (f3) at (2,2) {$f_3$};
        \node[rectangle, draw] (f4) at (0,2) {$f_4$};
        \node[rectangle, draw] (w2) at (1,1) {$w_2$};
        \node[rectangle, draw] (w3) at (-1,3) {$w_3$};
        \node[rectangle, draw] (w4) at (3,3) {$w_4$};
        \draw[thick] (f1) -- (w2);
        \draw[thick, dashed] (f1) to[in=-120, out=180] (w3);
        \draw[thick] (f2) -- (w2);
        \draw[thick, dashed] (f2) to[in=-60, out=0] (w4);
        \draw[thick] (f3) -- (w4);
        \draw[thick, dashed] (f3) -- (w2);
        \draw[thick] (f4) -- (w3);
        \draw[thick, dashed] (f4) -- (w2);
    \end{tikzpicture}
    \caption{symmetric difference $\bar\mu \triangle \mu'$} \label{fig:symmetric-difference}
    \end{subfigure}
    \caption{Two stable matchings neighboring in $(\SS,\succeq)$ and their symmetric difference.} \label{fig:eg-matching-ncycle}
\end{figure}
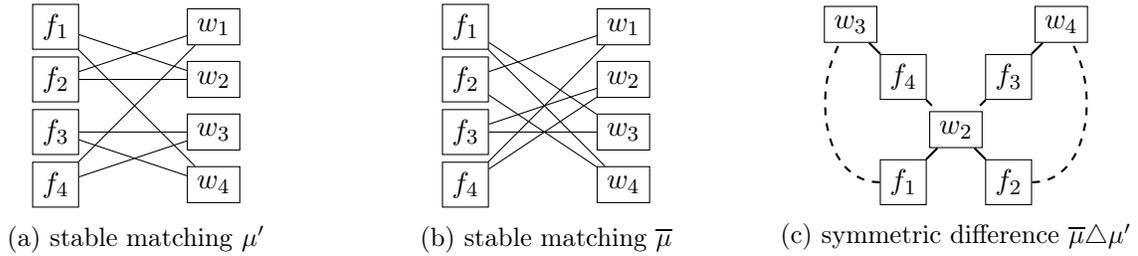

In the following, we focus on proving a bijection between $\D(\P)$ and $\Pi(\SS)$, and we often abbreviate $\Pi\coloneqq \Pi(\SS)$ and $\D\coloneqq \D(\P)$. In particular, we show the following.

\begin{theorem}\label{thm:rp-isom-birkhoff}
    Assume $\C_F$ and $\C_W$ are path-independent and cardinal monotone. Then, 
    \begin{enumerate}
        \item[(1)] the mapping $Q:\Pi\rightarrow \D$, with $Q(\rho)=\rho^+$, is a bijection; 
        \item[(2)] $(\D, \sqsupseteq)$ is isomorphic to the \emph{rotation poset} $(\Pi, \succeq^\star)$ where for two rotations $\rho_1, \rho_2\in \Pi$, $\rho_1\succeq^\star \rho_2$ if $Q(\rho_1)\sqsupseteq Q(\rho_2)$; 
        \item[(3)] $(\Pi, \succeq^\star)$ is a representation poset for $(\SS,\succeq)$ with representation function $\psi_\SS$ such that for any upper set $\bar\Pi$ of $(\Pi, \succeq^\star)$, $P(\psi_\SS^{-1} (\bar\Pi))= \psi_\P^{-1} (\{Q(\rho): \rho\in \bar\Pi\})$ where $\psi_\P$ is the representation function of $(\P, \subseteq)$ per Theorem~\ref{thm:birkhoff}; and $\psi_\SS^{-1}(\bar\Pi) = \big(\triangle_{\rho\in\bar\Pi} (\rho^{-} \triangle \rho^{+}) \big) \triangle \mu_F$, where $\triangle$ is the symmetric difference operator. Moreover, equivalently, we have $\psi_\SS^{-1}(\bar\Pi)=\mu_F \cup (\bigcup_{\rho\in\bar\Pi}\rho^+)\setminus (\bigcup_{\rho\in\bar\Pi}\rho^-)$.
    \end{enumerate}
\end{theorem}

\begin{lemma} \label{lem:additional}
    Let $\mu, \mu'\in \SS$ such that $\mu'\succ \mu$. If $w\in \mu(f)\setminus \mu'(f)$ for some $f$, then $w\notin P_f(\mu')$.
\end{lemma}

\begin{proof}
    Since $\mu'\succ \mu$, we have $\C_f(\mu'(f)\cup \mu(f)) = \mu'(f)$. By path-independence and consistency, we have
    $$w\notin \mu'(f) = \C_f(\mu'(f)\cup \mu(f))= \C_f(\C_f(\mu'(f) \cup \mu(f) \setminus \{w\}) \cup\{w\}) = \C_f(\mu'(f)\cup \{w\}).$$
    Therefore, $w\notin P_f(\mu')$, concluding the proof.
\end{proof}

\begin{lemma} \label{lem:rot-in-f}
    Let $\mu, \mu'\in \SS$ such that $\mu'$ is an immediate predecessor of $\mu$ in the stable matching lattice. Then, $\mu(f)\setminus \mu'(f) = P_f(\mu)\setminus P_f(\mu')$ for all $f\in F$. In particular, $P(\mu)\setminus P(\mu') = \rho^{+} (\mu', \mu)$.
\end{lemma}

\begin{proof}
   Fix a firm $f$. $\mu(f)\setminus \mu'(f) \subseteq P_f(\mu)\setminus P_f(\mu')$ follows by definition and from Lemma~\ref{lem:additional}. For the reverse direction, assume by contradiction that there exists $w\in P_f(\mu)\setminus P_f(\mu')$ but $w\notin \mu(f)\setminus \mu'(f)$. Since $w\notin P_f(\mu')$ implies that $w\notin \mu'(f)$ by definition of $P_f(\cdot)$, we also have $w\notin \mu(f)$. By Lemma~\ref{lem:exist-dom-P-middle}, there exists a stable matching $\bar \mu$ such that $\mu'\succeq \bar \mu\succeq \mu$ and $w\in \bar \mu(f)$. However, since $\mu'$ is an immediate predecessor of $\mu$ in the stable matching lattice, we either have $\bar \mu=\mu$ or $\bar \mu=\mu'$. However, both are impossible since we deduced $w\notin \mu(f)\cup \mu'(f)$.
\end{proof}

\begin{lemma} \label{lem:out-cannot-be-in}
    Let $\mu_1,\mu_2,\mu_3\in \SS$ such that $\mu_1\succ \mu_2\succ \mu_3$. If $w\in \mu_1(f)\setminus\mu_2(f)$ for some firm $f$, then $w\notin \mu_3(f)$.
\end{lemma} 

\begin{proof}
    First, note that $\mu_1\succ \mu_2$ implies $w\in \mu_1(f) = \C_f(\mu_1(f)\cup \mu_2(f))$. Thus, by substitutability, we have $w\in\C_f(\mu_2(f)\cup \{w\})$. Assume by contradiction that $w\in \mu_3(f)$. Then, applying Lemma~\ref{lem:additional} on $\mu_2$ and $\mu_3$, we have that $w\notin P_f(\mu_2)$, which is a contradiction.
\end{proof}

\begin{lemma} \label{lem:rot-out-f-dom}
   Let $\mu_1,\mu_1', \mu_2,\mu_2'\in \SS$ and assume that $\mu_1', \mu_2'$ are immediate predecessors of $\mu_1, \mu_2$ in the stable matching lattice, respectively. In addition, assume that $\mu_1\succ \mu_2$. If $P(\mu_1)\setminus P(\mu_1') = P(\mu_2)\setminus P(\mu_2')$, then $\mu_1'(f) \setminus \mu_1(f) = \mu_2'(f) \setminus \mu_2(f)$ for all firm $f\in F$.
\end{lemma}

\begin{proof}
    Fix a firm $f$. Due to Lemma~\ref{lem:rot-in-f}, we know $\mu_1(f) \setminus \mu_1'(f) = \mu_2(f) \setminus \mu_2'(f)$. By the~\ref{eq:equal-quota} property, we have $|\mu_1(f)| = |\mu_1'(f)|$ and $|\mu_2(f)|= |\mu_2'(f)|$. Thus, $|\mu_1'(f) \setminus \mu_1(f)| = |\mu_2'(f) \setminus \mu_2(f)|$ ($\natural$). If $\mu_1'(f) \setminus \mu_1(f)= \emptyset$, the claim follows immediately, and thus, in the following, we assume $\mu_1'(f) \setminus \mu_1(f) \neq \emptyset$. Assume by contradiction that there exists $w\in \mu_1'(f) \setminus \mu_1(f)$ but $w\notin \mu_2'(f) \setminus \mu_2(f)$. Since $\mu_1\succ \mu_2$ and $\mu_i'\succ \mu_i$ for $i\in \{1,2\}$, by Theorem~\ref{thm:iso-ros}, we have $P(\mu_1)\subsetneq P(\mu_2)$ and $P(\mu_i')\subsetneq P(\mu_i)$ for $i\in \{1,2\}$.  Therefore, $P(\mu_1')\subsetneq P(\mu_2')$ due to the assumption that $P(\mu_1)\setminus P(\mu_1') = P(\mu_2)\setminus P(\mu_2')$. Again by Theorem~\ref{thm:iso-ros}, we have $\mu_1'\succ \mu_2'$. Hence, $\mu_1\vee \mu_2' = \mu_1'$ and we must have $w\in \mu_2'(f)$ and thus, $w\in \mu_2(f)$. However, since $\mu_1'\succ \mu_1\succ \mu_2$ and $w\in \mu_1'(f) \setminus \mu_1(f)$, we can apply Lemma~\ref{lem:out-cannot-be-in} and conclude that $w\notin \mu_2(f)$, which is a contradiction. This shows $\mu_1'(f) \setminus \mu_1(f) \subseteq \mu_2'(f) \setminus \mu_2(f)$. Together with ($\natural$), we have $\mu_1'(f) \setminus \mu_1(f) = \mu_2'(f) \setminus \mu_2(f)$.
\end{proof}

\begin{lemma} \label{lem:set-operation}
    Let $A,B,A',B'$ be sets such that $A\subseteq A'$ and $B\subseteq B'$. In addition, assume that $A'\setminus A = B'\setminus B$. Then, $(A'\cap B') \setminus (A\cap B) = A'\setminus A$.
\end{lemma}

\begin{proof}
    Let $X\coloneqq A'\setminus A = B'\setminus B$. Notice that $A'=A\sqcup X$ and $B'=B\sqcup X$, where $\sqcup$ is the disjoint union operator. Therefore, we have $A\cap B = (A'\setminus X)\cap (B'\setminus X) = (A'\cap B')\setminus X$ and the claim follows. 
\end{proof}

\begin{lemma} \label{lem:rot-out-f}
   Let $\mu_1, \mu_1', \mu_2, \mu_2'\in \SS$ and assume that $\mu_1', \mu_2'$ are immediate predecessors of $\mu_1, \mu_2$ in the stable matching lattice, respectively. If $P(\mu_1)\setminus P(\mu_1') = P(\mu_2)\setminus P(\mu_2')$, then $\mu_1'(f) \setminus \mu_1(f) = \mu_2'(f) \setminus \mu_2(f)$ for every firm $f$. In particular, $\rho^{-}(\mu_1', \mu_1) = \rho^{-}(\mu_2', \mu_2)$.
\end{lemma}

\begin{proof}
    We first consider the case where $\mu_1=\mu_2$. By Lemma~\ref{lem:dom-P}, we have $P(\mu_i')\subseteq P(\mu_i)$ for $i\in \{1,2\}$. Therefore,
    $$P(\mu_1') = P(\mu_1)\setminus (P(\mu_1)\setminus P(\mu_1')) =  P(\mu_2)\setminus (P(\mu_2)\setminus P(\mu_2')) = P(\mu_2'),$$
    where the second equality is due to our assumptions that $\mu_1=\mu_2$ and $P(\mu_1)\setminus P(\mu_1') = P(\mu_2)\setminus P(\mu_2')$. Thus, $\mu_1'=\mu_2'$ because of Theorem~\ref{thm:iso-ros}, and the thesis then follows. Since the cases when $\mu_1\succ \mu_2$ or $\mu_2\succ \mu_1$ have already been considered in Lemma~\ref{lem:rot-out-f-dom}, for the following, we assume that $\mu_1$ and $\mu_2$ are not comparable. Let $\mu_3\coloneqq \mu_1\vee\mu_2$ and $\mu_3'\coloneqq \mu_1'\vee\mu_2'$. Note that $\mu_3' \succeq \mu_3$. Then, applying Lemma~\ref{lem:1to1-ros} and Lemma~\ref{lem:set-operation}, we have $$P(\mu_3)\setminus P(\mu_3') = (P(\mu_1) \cap P(\mu_2)) \setminus (P(\mu_1') \cap P(\mu_2')) = P(\mu_1)\setminus P(\mu_1').$$ By Theorem~\ref{thm:iso-ros}, Lemma~\ref{lem:center-md} and Lemma~\ref{lem:md-imply-immediate}, we also have that $\mu_3'$ is an immediate predecessor of $\mu_3$ in the stable matching lattice. Note that by construction, we have $\mu_3\succ \mu_1$ and $\mu_3\succ \mu_2$ since $\mu_1$ and $\mu_2$ are incomparable. Applying Lemma~\ref{lem:rot-out-f-dom} on $\mu_1$ and $\mu_3$ as well as on $\mu_2$ and $\mu_3$, we have $\mu_1'(f)\setminus \mu_1(f) = \mu_3'(f) \setminus \mu_3(f) = \mu_2'(f) \setminus \mu_2(f)$ for all firm $f\in F$, as desired.
\end{proof}

\begin{theorem} \label{thm:1to1-rot-md}
    Let $\mu_1, \mu_1', \mu_2,\mu_2'\in \SS$ and assume that $\mu_1', \mu_2'$ are immediate predecessors of $\mu_1, \mu_2$ in the stable matching lattice, respectively. Then, $P(\mu_1)\setminus P(\mu_1') = P(\mu_2)\setminus P(\mu_2')$ if and only if $\rho(\mu_1', \mu_1) = \rho(\mu_2', \mu_2)$.
\end{theorem}

\begin{proof}
    For the ``only if'' direction, assume $P(\mu_1) \setminus P(\mu_1')= P(\mu_2) \setminus P(\mu_2')$. Then, $\rho^+(\mu_1', \mu_1)= \rho^+(\mu_2', \mu_2)$ by Lemma~\ref{lem:rot-in-f} and $\rho^-(\mu_1', \mu_1)= \rho^-(\mu_2', \mu_2)$ by Lemma~\ref{lem:rot-out-f}. Thus, $\rho(\mu_1', \mu_1) = \rho(\mu_2', \mu_2)$. For the ``if'' direction, assume $\rho(\mu_1', \mu_1) = \rho(\mu_2', \mu_2)$. Then, immediately from Lemma~\ref{lem:rot-in-f}, we have that $P(\mu_1) \setminus P(\mu_1')= \rho^+(\mu_1', \mu_1) = \rho^+(\mu_2', \mu_2) = P(\mu_2) \setminus P(\mu_2')$.
\end{proof}

\begin{remark}
    In the \textsc{SM-Model} with P-sets defined as by Gusfield and Irving~\cite{gusfield1989stable} stated in Remark~\ref{rmk:P-set-defn}, Theorem~\ref{thm:1to1-rot-md} immediately follows from the definition of P-set. In fact, one can explicitly and uniquely construct $\rho(\mu', \mu)$ from $P(\mu)\setminus P(\mu')$. In particular, $\rho^{+}(\mu', \mu)$ is the set of edges $(f,w)$ such that $P_f(\mu)\neq P_f(\mu')$ and $w$ is the least preferred partner of $f$ among $P_f(\mu)\setminus P_f(\mu')$, and $\rho^{-}(\mu', \mu)$ is the set of edges $(f,w)$ such that $P_f(\mu)\neq P_f(\mu')$ and $w$ is the partner that, in the preference list $\ge_f$, is immediately before the most preferred partner of $f$ among $P_f(\mu)\setminus P_f(\mu')$.
\end{remark}

\begin{proof}[Proof of Theorem~\ref{thm:rp-isom-birkhoff}]
    Because of Theorem~\ref{thm:iso-ros} and Lemma~\ref{lem:rot-in-f}, for every $K_1\in \D$,  there exist stable matchings $\mu'$ and $\mu$ with $\mu'$ being an immediate predecessor of $\mu$ such that $K_1= P(\mu)\setminus P(\mu') = \rho^+(\mu', \mu)$. Thus, the mapping $Q$ is onto.  Theorem~\ref{thm:1to1-rot-md} further implies that $Q$ is injective. Hence, the mapping $Q$ is a bijection. This bijection and the definition of $\succeq^\star$ immediately imply that $(\D, \sqsupseteq)$ is isomorphic to $(\Pi, \succeq^\star)$. Together with the isomorphism between $(\SS,\succeq)$ and $(\P,\subseteq)$, and the fact that $(\D,\sqsupseteq)$ is a representation poset of $(\P,\subseteq)$, we deduce a bijection between elements of $(\SS,\succeq)$ and upper sets of $(\Pi,\succeq^\star)$. That is, $(\Pi,\succeq^\star)$ is a representation poset of $(\SS,\succeq)$ and its representation function $\psi_\SS$ satisfies that for every $\mu\in \SS$, $\{Q(\rho): \rho\in \psi_\SS(\mu)\} = \psi_\P(P(\mu))$. It remains to show that the formula for the inverse of $\psi_{\SS}$ given in the statement of the theorem is correct. Let $\mu\in \SS$ and let $\mu_0, \mu_1, \cdots, \mu_k$ be a sequence of stable matchings such that $\mu_{i-1}$ is an immediate predecessor of $\mu_i$ in $(\SS, \succeq)$ for all $i\in [k]$, $\mu_0=\mu_F$ and $\mu_k=\mu$. In addition, let $\rho_i= \rho(\mu_{i-1}, \mu_i)$ for all $i\in [k]$. Note that $\mu= \mu_F\triangle (\rho_1^{-} \triangle \rho_1^{+}) \triangle (\rho_2^{-} \triangle \rho_2^{+}) \triangle \cdots \triangle (\rho_k^{-} \triangle \rho_k^{+})$ ($\natural$). By Theorem~\ref{thm:iso-ros}, $P(\mu_0)\subseteq P(\mu_1) \subseteq \cdots \subseteq P(\mu_k)$, and thus,
    $$P(\mu) = P(\mu_0) \cup \big( P(\mu_1) \setminus P(\mu_0) \big) \cup \big( P(\mu_2) \setminus P(\mu_1) \big) \cup \cdots \cup \big( P(\mu_k) \setminus P(\mu_{k-1}) \big).$$
    Therefore, by Lemma~\ref{lem:rot-in-f}, $P(\mu) = P(\mu_F) \cup Q(\rho_1) \cup \cdots \cup Q(\rho_k)$. By Lemma~\ref{lem:unq-upperset}, we know that $\{Q(\rho_i): i\in [k]\}$ is an upper set of $\D$ and thus, $\psi_\P(P(\mu)) = \{Q(\rho_i): i\in [k]\}$ due to Theorem~\ref{thm:birkhoff}. Hence, $\psi_\SS(\mu) = \{\rho_i: i\in [k]\}$. The inverse of $\psi_\SS$ must be as in the first definition in the thesis so that  ($\natural$) holds. 
    
    Let $(f,w)$ be a firm-worker pair. If $(f,w)\in \rho_i^-$ for some $i\in [k]$, then $(f,w)\notin \mu$ due to Lemma~\ref{lem:out-cannot-be-in}. In addition, because of Lemma~\ref{lem:center-uniq} and the bijection $Q$, $\mu_F$, $\rho_1^+$, $\rho_2^+$, $\cdots$, $\rho_i^+$ are disjoint. Hence, if $(f,w)\in \mu_F\cup (\bigcup \{\rho_i^+: i\in[k]\})$ but $(f,w)\notin \bigcup \{\rho_i^-: i\in[k]\}$, then $(f,w)\in \mu$. The second definition of $\psi_\SS$ from the thesis follows immediately from these facts and the previous definition.
\end{proof}

\begin{example} \label{eg:rotation-poset}
    Consider the following instance where each agent has a quota of $2$. 
     \[\begin{array}{ccl}
        f_1: & \ge_{f_1,1}: & w_4\ w_2\ w_1\ w_3 \\
        & \ge_{f_1,2}: & w_3\ w_1\ w_2\ w_4 \\ 
        & \\
        f_2: & \ge_{f_2,1}: & w_2\ w_3\ w_4\ w_1 \\
        & \ge_{f_2,2}: & w_1\ w_4\ w_3\ w_2 \\
        & \\
        f_3: & \ge_{f_3,1}: & w_1\ w_2\ w_4\ w_3 \\
        & \ge_{f_3,2}: & w_3\ w_4\ w_2\ w_1 \\
        f_4: & \ge_{f_4,1}: & w_4\ w_3\ w_1\ w_2 \\
        & \ge_{f_4,2}: & w_2\ w_1\ w_3\ w_4
    \end{array} \qquad \qquad \qquad 
    \begin{array}{ccl}
        w_1: & \ge_{w_1,1}: & f_2\ f_1\ f_3\ f_4 \\
        & \ge_{w_1,2}: & f_4\ f_1\ f_3\ f_2 \\
        & \ge_{w_1,3}: & f_2\ f_4\ f_3\ f_1 \\
        w_2: & \ge_{w_2,1}: & f_1\ f_2\ f_4\ f_3 \\
        & \ge_{w_2,2}: & f_3\ f_2\ f_4\ f_1 \\
        & \ge_{w_2,3}: & f_1\ f_3\ f_4\ f_2 \\
        w_3: & \ge_{w_3,1}: & f_4\ f_3\ f_1\ f_2 \\
        & \ge_{w_3,2}: & f_2\ f_1\ f_3\ f_4 \\
        w_4: & \ge_{w_4,1}: & f_2\ f_3\ f_4\ f_1 \\
        & \ge_{w_4,2}: & f_1\ f_4\ f_3\ f_2 
    \end{array}\]
    
    The stable matchings of this instance and their corresponding P-sets are listed below. To be concise, for matching $\mu$, we list the assigned partners of firms $f_1, f_2, f_3, f_4$ in the exact order. Similarly, for P-set $P(\mu)$, we list in the order of $P_{f_1}(\mu)$, $P_{f_2}(\mu)$, $P_{f_3}(\mu)$, $P_{f_4}(\mu)$ and replace $\{w_1,w_2,w_3,w_4\}$ with $W$.
    \[\resizebox{.97\textwidth}{!}{$\begin{array}{cc}
    \begin{aligned}
        \mu_F&= (\{w_3, w_4\}, \{w_1, w_2\}, \{w_1, w_3\}, \{w_2, w_4\}) \\
        \mu_1&= (\{w_3, w_4\}, \{w_1, w_2\}, \{w_1, w_4\}, \{w_2, w_3\}) \\
        \mu_2&= (\{w_3, w_4\}, \{w_1, w_2\}, \{w_2, w_3\}, \{w_1, w_4\}) \\
        \mu_3&= (\{w_2, w_4\}, \{w_1, w_3\}, \{w_2, w_3\}, \{w_1, w_4\}) \\
        \mu_4&= (\{w_3, w_4\}, \{w_1, w_2\}, \{w_2, w_4\}, \{w_1, w_3\}) \\
        \mu_W&= (\{w_2, w_4\}, \{w_1, w_3\}, \{w_2, w_4\}, \{w_1, w_3\})
    \end{aligned} &
    \begin{aligned}
        P(\mu_F)&= (\{w_3, w_4\}, \{w_1, w_2\}, \{w_1, w_3\}, \{w_2, w_4\}) \\
        P(\mu_1)&= (\{w_3, w_4\}, \{w_1, w_2\}, \{w_1, w_3, w_4\}, \{w_2, w_3, w_4\}) \\
        P(\mu_2)&= (\{w_3, w_4\}, \{w_1, w_2\}, \{w_1, w_2, w_3\}, \{w_1, w_2, w_4\}) \\
        P(\mu_3)&= (\{w_2, w_3, w_4\}, \{w_1, w_2, w_3\}, \{w_1, w_2, w_3\}, \{w_1, w_2, w_4\}) \\
        P(\mu_4)&= (\{w_3, w_4\}, \{w_1, w_2\}, W, W) \\
        P(\mu_W)&= (\{w_2, w_3, w_4\}, \{w_1, w_2, w_3\}, W, W)
    \end{aligned}
    \end{array}$}
    \]

    The stable matching lattice $(\SS, \succeq)$ and the rotation poset $(\Pi, \succeq^\star)$ are shown in Figure~\ref{fig:eg-sml-rp}.
    
    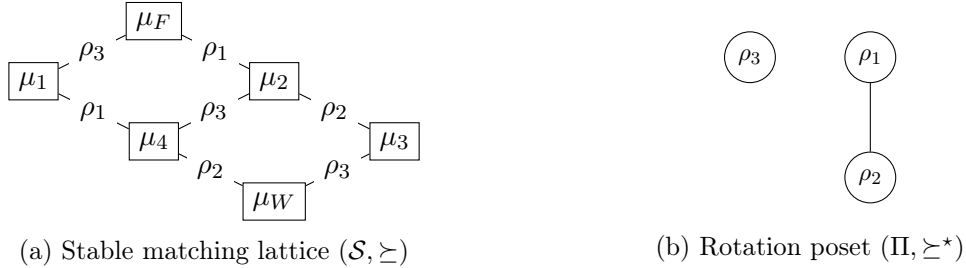
\begin{figure}[ht]
    \centering
    \begin{subfigure}{.58\textwidth}
        \centering
        \begin{tikzpicture}[scale=.8]
            \node[rectangle, draw] (0) at (2,4) {$\mu_F$};
            \node[rectangle, draw] (1) at (0,3) {$\mu_1$};
            \node[rectangle, draw] (2) at (4,3) {$\mu_2$};
            \node[rectangle, draw] (3) at (2,2) {$\mu_4$};
            \node[rectangle, draw] (4) at (6,2) {$\mu_3$};
            \node[rectangle, draw] (5) at (4,1) {$\mu_W$};
            \draw[] (0) -- (1) node [midway, fill=white] {$\rho_3$};
            \draw[] (0) -- (2) node [midway, fill=white] {$\rho_1$};
            \draw[] (1) -- (3) node [midway, fill=white] {$\rho_1$};
            \draw[] (2) -- (3) node [midway, fill=white] {$\rho_3$};
            \draw[] (2) -- (4) node [midway, fill=white] {$\rho_2$};
            \draw[] (3) -- (5) node [midway, fill=white] {$\rho_2$};
            \draw[] (4) -- (5) node [midway, fill=white] {$\rho_3$};
        \end{tikzpicture}
        \caption{Stable matching lattice $(\SS, \succeq)$} \label{fig:eg-sml}
    \end{subfigure}%
    \begin{subfigure}{.38\textwidth}
        \centering
        \begin{tikzpicture}[scale=.8]
            \node[] at (0,4.2) {};
            \node[] at (0,1) {};
            \node[circle, draw, scale=.85] (1) at (0,3.5) {$\rho_3$};
            \node[circle, draw, scale=.85] (2) at (2,3.5) {$\rho_1$};
            \node[circle, draw, scale=.85] (3) at (2,1.5) {$\rho_2$};
            \draw[] (2) -- (3);
        \end{tikzpicture}
        \caption{Rotation poset $(\Pi, \succeq^\star)$} \label{fig:eg-rp}
    \end{subfigure} 
    \caption{The stable matching lattice and its rotation poset of Example~\ref{eg:rotation-poset}} \label{fig:eg-sml-rp}
    \end{figure} 
    
    Due to Theorem~\ref{thm:iso-ros} and Theorem~\ref{thm:rp-isom-birkhoff}, one can also view Figure~\ref{fig:eg-sml} and Figure~\ref{fig:eg-rp} respectively as the ring of sets $(\P, \subseteq)$ and the poset of minimal differences $(\D, \sqsupseteq)$. 
    
    Below, we list the rotations in $\Pi$ and their corresponding minimal differences in $\D$. In addition, we label in Figure~\ref{fig:eg-sml} the edges of the Hasse Diagram by these rotations.
    \begin{align*}
        \rho_1:\ Q(\rho_1) =\rho^{+}_1 &= \{\{f_3, w_2\}, \{f_4, w_1\}\}; \ \rho^{-}_1 = \{\{f_3, w_1\}, \{f_4, w_2\}\} \\
        \rho_2:\ Q(\rho_2) =\rho^{+}_2 &= \{\{f_1, w_2\}, \{f_2, w_3\}\}; \ \rho^{-}_2 = \{\{f_1, w_3\}, \{f_2, w_2\}\} \\
        \rho_3:\ Q(\rho_3) =\rho^{+}_3 &= \{\{f_3, w_4\}, \{f_4, w_3\}\}; \ \rho^{-}_3 = \{\{f_3, w_3\}, \{f_4, w_4\}\} 
    \end{align*} 
\end{example}

\subsection{Concluding the proof for the first part of Theorem~\ref{thm:main-intro}} \label{sec:affine-repr-final}

Because of Theorem~\ref{thm:rp-isom-birkhoff}, part (3), we know that poset $(\Pi, \succeq^\star)$ represents lattice $({\cal S},\succeq)$. Let $\psi_\SS$ be the representation function as defined in Theorem~\ref{thm:rp-isom-birkhoff}. We denote by $E\subseteq F\times W$ the set of acceptable firm-worker pairs. Hence, $E$ is the base set of lattice $(\SS, \succeq)$. We deduce the following, proving the structural statement from Theorem~\ref{thm:main-intro}.

\begin{lemma} \label{lem:upperset-contained-dom}
    Let $\bar\Pi_1, \bar\Pi_2$ be two upper sets of $(\Pi, \succeq^\star)$ and let $\mu_i= \psi_\SS^{-1}(\bar\Pi_i)$ for $i\in \{1,2\}$. If $\bar\Pi_1 \subseteq \bar\Pi_2$, then $\mu_1\succeq \mu_2$.
\end{lemma}

\begin{proof}
    Let $\bar \D_i\coloneqq \{Q(\rho): \rho\in \bar\Pi_i\}$ and let $P_i\coloneqq \psi_\P^{-1}(\bar \D_i)$ for $i\in\{1,2\}$. Since $\bar\Pi_1 \subseteq \bar\Pi_2$, we have $\bar\D_1 \subseteq \bar\D_2$ and thus subsequently $P_1\subseteq P_2$. Since $\psi_\P^{-1} (\bar \D_i) = P(\psi_\SS^{-1} (\bar\Pi_i))$ by Theorem~\ref{thm:rp-isom-birkhoff},  $P_i=P(\mu_i)$ for both $i=1,2$. Therefore, by Theorem~\ref{thm:iso-ros}, $\mu_1\succeq \mu_2$.
\end{proof}

\begin{lemma} \label{lem:in-then-out}
    Let $\rho_1, \rho_2\in \Pi$. If $\rho_1^+\cap \rho_2^-\neq \emptyset$, then $\rho_1 \succ^\star \rho_2$.
\end{lemma}

\begin{proof}
    Assume by contradiction that $\rho_1\not \succ^\star \rho_2$, that is, either $\rho_2 \succ^\star \rho_1$ or that they are not comparable. Let $\bar\Pi_1\coloneqq \{\rho\in \Pi: \rho\succeq \rho_2\}$ be the inclusion-wise smallest upper set of $\Pi$ that contains $\rho_2$, let $\bar\Pi_0\coloneqq \bar\Pi_1\setminus \{\rho_2\}$, and let $\bar\Pi_2\coloneqq \{\rho\in \Pi: \rho\succeq \rho_1 \textup{ or } \rho\succeq \rho_2\}$ be the inclusion-wise smallest upper set of $\Pi$ that contains both $\rho_1$ and $\rho_2$. Note that $\bar\Pi_0 \subsetneq \bar\Pi_1 \subsetneq \bar\Pi_2$, where the second strict containment is due to our assumption that $\rho_1 \not\succ^\star \rho_2$ and thus $\rho_1\notin \bar\Pi_1$. For $i\in \{0,1,2\}$, let $\mu_i \coloneqq (\triangle_{\rho \in \bar\Pi_i} (\rho^- \triangle \rho^+)) \triangle \mu_F$. Since $\bar \Pi_i$ is an upper set of $(\Pi,\succeq^\star)$, $\mu_i$ is a stable matching by Theorem~\ref{thm:rp-isom-birkhoff}. Moreover,  $\mu_0\succ \mu_1\succ \mu_2$ by Lemma~\ref{lem:upperset-contained-dom}. Let $(f,w)\in \rho_1^+\cap \rho_2^-$. Since $\rho(\mu_0,\mu_1)=\rho_2$, we have $(f,w)\in \mu_0\setminus \mu_1$. Since $\rho_1$ is a $\succeq$-minimal element in $\bar\Pi_2$, $\bar\Pi_2\setminus \{\rho_1\}$ is also an upper set of $\Pi$. Then, $\mu_2'\coloneqq (\triangle_{\rho \in \bar\Pi_2\setminus \{\rho_1\}} (\rho^- \triangle \rho^+)) \triangle \mu_F$ is a stable matching by Theorem~\ref{thm:rp-isom-birkhoff}, and $\mu_2 = \mu_2'\setminus \rho_1^- \cup \rho_1^+$. Thus, we have $(f,w)\in \mu_2$. Together, we have $w\in (\mu_0(f)\cap \mu_2(f)) \setminus \mu_1(f)$. However, this contradicts Lemma~\ref{lem:out-cannot-be-in}. 
\end{proof}

\begin{theorem} \label{thm:main-first-half}
    The rotation poset $(\Pi, \succeq^\star)$ affinely represents the stable matching lattice $(\SS, \succeq)$ with affine function $g(u)=Au + \Chi^{\mu_F}$, where $A\in \{0,\pm 1\}^{E\times \Pi}$ is matrix with columns $\Chi^{\rho^+} - \Chi^{\rho-}$ for each $\rho\in \Pi$. Moreover, $|\Pi|=O(|F||W|)$ and matrix $A$ has full column rank.
\end{theorem}

\begin{proof}
    The first claim follows immediately because by Theorem~\ref{thm:rp-isom-birkhoff}, part (3), $\chi^{\mu} = A \chi^{\psi_\SS(\mu)} + \chi^{\mu_F}$, for any stable matching $\mu$. Because of Theorem~\ref{thm:rp-isom-birkhoff}, $|\Pi|=|\D|$. In addition, by Lemma~\ref{lem:center-uniq}, we have $|\D|=|E|=O(|F||W|)$. Thus, $|\Pi|= O(|F||W|)$. Finally, we show that matrix $A$ has full column rank. Assume by contradiction that there is a non-zero vector $\lambda\in \R^\Pi$ such that $\sum_{\rho\in \Pi} \lambda_\rho (\Chi^{\rho^+} - \Chi^{\rho-}) = \mathbf{0}$. Let $\tilde\Pi\coloneqq \{\rho\in \Pi: \lambda_\rho\neq 0\}$ denote the set of rotations whose corresponding coefficients in $\lambda$ are non-zero. Let $\rho_1$ be a minimal rotation (w.r.t. $\succeq^\star$) in $\tilde\Pi$ and let $(f,w)$ be a firm-worker pair in $\rho_1^+$. Because of Lemma~\ref{lem:center-uniq} and the bijection $Q$, there is no rotation $\rho\neq \rho_1$ such that $(f,w)\in \rho^+$. Therefore, there must exist a rotation $\rho_2\in \tilde\Pi$ with $(f,w)\in \rho_2^-$. Note that we must have $\rho_1 \succ^\star \rho_2$ due to Lemma~\ref{lem:in-then-out}. However, this contradicts the choice of $\rho_1$.
\end{proof}

\section{Algorithms} \label{sec:algo}

Because of Theorem~\ref{thm:main-first-half}, in order to conclude the proof of Theorem~\ref{thm:main-intro}, we are left to explicitly construct $(\Pi, \succeq^\star)$. That is, we need to find elements of $\Pi$, and how they relate to each other via $\succeq^\star$. We fix an instance $(F,W,\C_F,\C_W)$ and abbreviate ${\cal S}:={\cal S}(\C_F,\C_W)$. 

In this section, we further assume workers' choice functions to be quota-filling. Under this additional assumption, for each worker $w\in W$, the family of sets of partners $w$ is assigned to under all stable matchings (denoted as $\Phi_w$) satisfies an additional property, which we call the \emph{full-quota}\footnote{Note that the full-quota property is analogous to the \emph{Rural Hospital Theorem}~\cite{roth1986allocation} in the \textsc{SA-Model} where agents have preferences over individual partners instead of over sets of partners.} property (see Lemma~\ref{lem:full-quota}). Recall that $q_w$ denote the quota of worker $w$ and $\bar q_w$ is the number of firms matched to $w$ under every stable matching, which is constant due to the~\ref{eq:equal-quota} property (i.e., $|S|=\bar q_w$ for all $S\in \Phi_w$).

\begin{lemma} \label{lem:full-quota}
    For every worker $w\in W$, if $\bar q_w< q_w$, then $w$ is matched to the same set of firms in all stable matchings. That is,
    \begin{equation} \tag{full-quota} \label{eq:full-quota}
    \bar q_w < q_w \implies |\Phi_w|=1. 
    \end{equation}
\end{lemma} 

\begin{proof}
    Assume by contradiction that $\bar q_w < q_w$ but $|\Phi_w|>1$. Let $S_1, S_2$ be two distinct elements from $\Phi_w$ and let $\mu_i$ be the matching such that $\mu_i(w)=S_i$ for $i=1,2$. Note that due to the~\ref{eq:equal-quota} property, we have $|S_1|= |S_2|= \bar q_w$. Consider the stable matching $\mu\coloneqq \mu_1\wedge \mu_2$. Then, 
    $$|\mu(w)|= |\C_w(\mu_1(w) \cup \mu_2(w))|= |\C_w(S_1 \cup S_2)| = \min(|S_1\cup S_2|, q_w|) > \bar q_w,$$
    where the first equality is by Theorem~\ref{thm:alkan} and the last two relations are by quota-filling. However, this contradicts the~\ref{eq:equal-quota} property since $\mu$ is a stable matching.
\end{proof}

Our approach to construct $(\Pi,\succeq^\star)$ is as follows. First, we recall Roth's adaptation of the Deferred Acceptance algorithm to find a firm- or worker-optimal stable matching (Section~\ref{sec:DA}). Second, we feed the output of Roth's algorithm to an algorithm that produces a maximal chain $C_1,C_2,\dots,C_k$ of $({\cal S},\succeq)$ and the set $\Pi$ (Section~\ref{sec:chain}). In Section~\ref{sec:find-from-chain}, we give an algorithm that, given a maximal chain of a ring of sets, constructs the partial order of the poset of minimal differences. This and previous facts are then exploited in Section~\ref{sec:irreducible} to construct the partial order $\succeq^\star$ on elements of $\Pi$. We sum up our algorithm in Section~\ref{sec:algo-summary}, where we show that the overall running time is $O(|F|^3 |W|^3 \textuptt{oracle-call})$.

We start with a definition and properties which will be used in later algorithms. For a matching $\mu$, let
$$\bar{X}_f(\mu) \coloneqq \{w\in W(f): \C_f(\mu(f)\cup\{w\}) = \mu(f)\},$$ and define the \emph{closure} of $\mu$, denoted by $\bar X(\mu)$, as the collection of sets $\{\bar X_f(\mu): f\in F\}$. Note that $\mu(f)\subseteq \bar X_f(\mu)$ for every firm $f$ and individually rational matching $\mu$.

\begin{lemma} \label{lem:choice-on-closure}
    Let $\mu$ be an individually rational matching. Then, for every firm $f$, we have $\C_f(\bar X_f(\mu)) = \mu(f)$.
\end{lemma}

\begin{proof}
    Fix a firm $f$. Since $\mu$ is individually rational, $\C_f(\mu(f))=\mu(f)$. The claim then follows from a direct application of Lemma~\ref{lem:PI-repeat} with $A_1=\mu(f)$ and $A_2 = \bar X_f(\mu)$.
\end{proof}

\begin{lemma} \label{lem:closure-contain-dominated}
    Let $\mu_1, \mu_2 \in \SS(\C_F,\C_W)$ such that $\mu_1 \succeq \mu_2$. Then, for every firm $f$, $\mu_2(f)\subseteq \bar X_f(\mu_1)$.
\end{lemma}

\begin{proof}
    Since $\mu_1 \succeq \mu_2$, we have $\C_f(\mu_1(f)\cup \mu_2(f))=\mu_1(f)$ for every firm $f$. Thus, by the consistency property of $\C_f$, for every $w\in \mu_2(f)$, we have $\C_f(\mu_1(f)\cup \{w\})=\mu_1(f)$. The claim follows.
\end{proof}

\begin{lemma} \label{lem:time-basic-ops}
    The following three operations can be performed in polynomial times: 
    \begin{enumerate}[leftmargin=1cm]
        \item[(1).] given a matching $\mu$, computing its closure $\bar X(\mu)$ can be performed in time  $O(|F||W|\textuptt{oracle-call})$;
        \item[(2).] given a matching $\mu$, deciding whether it is stable can be performed in time $O(|F||W| \textuptt{oracle-call})$;
        \item[(3).] given stable matchings $\mu,\mu' \in {\cal S}$, deciding whether $\mu' \succeq \mu$ can be performed in time $O(|F| \textuptt{oracle-call})$.
    \end{enumerate}
\end{lemma}

\begin{proof}
    (1). For any firm $f$, computing $\bar X_f(\mu)$ requires $O(|W|)$ \textuptt{oracle-calls} by definition and thus, computing the closure of $\mu$ takes $O(|F||W|)$ \textuptt{oracle-calls}. (2). To check if a matching $\mu$ is stable, we need to check first if it is individually rational, which takes $O(|F|+|W|)$ \textuptt{oracle-calls}, and then to check if it admits any blocking pair, which takes $O(|F||W|)$ $\textuptt{oracle-calls}$. (3). To decide if $\mu'\succeq \mu$, one need to check if for every firm $f\in F$, $\C_f(\mu'(f)\cup \mu(f))=\mu'(f)$, and this takes $O(|F|)$ \textuptt{oracle-calls}.
\end{proof}

\subsection{Deferred acceptance algorithm} \label{sec:DA}

The deferred acceptance algorithm introduced in~\cite{roth1984stability}\footnote{The model considered in~\cite{roth1984stability} is more general than our setting here, where choice functions are only assumed to be substitutable and consistent, not necessarily quota-filling.} can be seen as a generalization of the algorithm proposed in~\cite{gale1962college}. For the following, we assume that firms are the proposing side. Initially, for each firm $f$, let $X_f\coloneqq W(f)$, i.e., the set of acceptable workers of $f$. At every step, every firm $f$ \emph{proposes} to workers in $\C_f(X_f)$. Then, every worker $w$ considers the set of firms $X_w$ who made a proposal to $w$, \emph{temporarily accepts} $Y_w\coloneqq \C_w(X_w)$, and \emph{rejects} the rest. Afterwards, each firm $f$ removes from $X_f$ all workers that rejected $f$. The \emph{firm-proposing} algorithm iterates until there is no rejection. Hence, throughout the algorithm, $X_f$ denotes the set of acceptable workers of $f$ that have not rejected $f$.  A formal description is given in Algorithm~\ref{alg:DA}. 

\begin{algorithm}[ht]
    \caption{Firm-proposing DA algorithm for an instance $(F,W,\C_F,\C_W)$.} \label{alg:DA}
 	\begin{algorithmic}[1]
 	    \normalsize
     	\State initialize the step count $s\gets 0$
     	\ForInline{\textbf{each} firm $f$}{initialize $X^{(s)}_f \gets W(f)$}
     	\Repeat
     	\For{\textbf{each} worker $w$}
     	\State  $X^{(s)}_w \gets \{f\in F: w\in \C_f(X^{(s)}_f)\}$
     	\State $Y^{(s)}_w \gets \C_w(X^{(s)}_w)$
     	\EndFor
     	\For{\textbf{each} firm $f$}
     	\State update $X^{(s+1)}_f \gets X^{(s)}_f \setminus \{w\in W: f\in X^{(s)}_w \setminus Y^{(s)}_w\}$ \label{step:DA-remove}
     	\EndFor 
     	\State update the step count $s\gets s+1$
     	\Until{$X^{(s)}_f= X^{(s-1)}_f$ for every firm $f$}
     	\Output{matching $\bar\mu$ with $\bar\mu(w)= Y^{(s-1)}_w$ for every worker $w$} 
    \end{algorithmic}
\end{algorithm}

Note that for every step $s$ other than the final step, there exists a firm $f \in F$ such that $X_f^{(s)} \subsetneq X_f^{(s-1)}$. Therefore, the algorithm terminates, since there is a finite number of firms and workers. Moreover, the output has interesting properties. 

\begin{theorem}[Theorem 2,~\cite{roth1984stability}]
    Let $\bar \mu$ be the output of Algorithm~\ref{alg:DA} over a matching market $(F, W, \C_F, \C_W)$ assuming $\C_F, \C_W$ are path-independent. Then, $\bar \mu=\mu_F$.
\end{theorem}

Due to the symmetry between firms and workers in a market where the only assumption on choice functions is path-independence, swapping the role of firms and workers in Algorithm~\ref{alg:DA}, we have the \emph{worker-proposing} deferred acceptance algorithm, which outputs $\mu_W$.

\subsection{Constructing $\Pi$ via a maximal chain of $(\SS, \succeq)$ } \label{sec:chain}

Let $(\H,\subseteq)$ be a ring of sets. A \emph{chain} $C_0, \cdots, C_k$ in $(\H,\subseteq)$ is an ordered subset of $\H$ such that $C_{i-1}$ is a predecessor of $C_i$ in $(\H,\subseteq)$ for all $i\in [k]$. The chain is \emph{complete} if moreover $C_{i-1}$ is an immediate predecessor of $C_i$ for all $i \in [k]$; it is \emph{maximal} if it is complete, $C_0=H_0$ and $C_k=H_z$. Consider $K\in \D(\H)$. If $K=C_i\setminus C_{i-1}$ for some $i\in [k]$, then we say that the chain \emph{contains} the minimal difference $K$. We start with the theorem below, where it is shown that the set $\D(\H)$ can be obtained by following any maximal chain of $(\H, \subseteq)$.

\begin{theorem}[Theorem 2.4.2, \cite{gusfield1989stable}] \label{thm:chain-md}
    Let $H',H\in \H$ such that $H'\subseteq H$. Then, there exists a complete chain from $H'$ to $H$ in $(\H,\subseteq)$, and every such chain contains exactly the same set of minimal differences. In particular, for any maximal chain $(C_0, \cdots, C_k)$ in $(\H,\subseteq)$, we have $\{C_i \setminus C_{i-1}: i\in [k]\} = \D(\H)$ and $k=|\D(\H)|$.
\end{theorem}

In this section, we present Algorithm~\ref{alg:immediate-descendant} that, on inputs $\mu'$, outputs a stable matching $\mu$ that is an immediate descendant of $\mu'$ in $(\SS, \succeq)$. Then, using Algorithm~\ref{alg:immediate-descendant} as a subroutine, Algorithm~\ref{alg:max-chain} gives a maximal chain of $(\SS, \succeq)$. 

We start by extending to our setting the \emph{break-marriage} idea proposed by McVitie and Wilson~\cite{mcvitie1971stable} for finding the full set of stable matchings in the one-to-one stable marriage model. Given a stable matching $\mu'$ and a firm-worker pair $(f',w') \in \mu'\setminus \mu_W$, the break-marriage procedure, denoted as \textuptt{break-marriage($\mu',f',w'$)}, works as follows. We first initialize $X_f$ to be $\bar X_f(\mu')$ for every firm $f\neq f'$, while we set $X_{f'}=\bar X_{f'}(\mu')\setminus\{w'\}$. We then restart the deferred acceptance process. The algorithm continues in iterations as in the \textbf{repeat} loop of Algorithm~\ref{alg:DA}, with the exception that worker $w'$ temporarily accepts $Y_{w'}\coloneqq \C_{w'}(X_{w'} \cup\{f'\}) \setminus \{f'\}$. As an intuitive explanation, this acceptance rule of $w'$ ensures that for the output matching $\bar\mu$, we have $\C_{w'}(\bar\mu(w') \cup \mu'(w')) =\bar\mu(w')$, as we show in Lemma~\ref{lem:bm-firm-worse}. The formal break-marriage procedure is summarized in Algorithm~\ref{alg:BM}. See Example~\ref{eg:break-marriage} for a demonstration. Note that by choice of the pair $(f',w')$, we have $|\mu'(w')| = q_{w'}$. 

\begin{algorithm}[ht]
    \caption{\textuptt{break-marriage($\mu',f',w'$)}, with $(f',w')\in \mu' \setminus \mu_W$ and $\mu'\in {\cal S}$} \label{alg:BM}
	\begin{algorithmic}[1]
	    \normalsize
	    \ForInline{\textbf{each} firm $f\neq f'$}{initialize $X^{(0)}_f\gets \bar X_f(\mu')$}
     	\State initialize $X^{(0)}_{f'}\gets \bar X_{f'}(\mu') \setminus \{w'\}$ \label{step:BM-remove-init}
     	\State initialize the step count $s\gets 0$
     	\Repeat
     	\For{\textbf{each} worker $w$}
     	\State $X^{(s)}_w\gets \{f\in F: w\in \C_f(X^{(s)}_f)\}$
     	\IfThenElse{$w\neq w'$}{$Y^{(s)}_w \gets \C_w(X^{(s)}_w)$}{$Y^{(s)}_w \gets \C_w(X^{(s)}_w\cup \{f'\}) \setminus \{f'\}$}
     	\EndFor
     	\For{\textbf{each} firm $f$}
     	\State update $X^{(s+1)}_f \gets X^{(s)}_f \setminus \{w\in W: f\in X^{(s)}_w\setminus Y^{(s)}_w\}$ \label{step:BM-remove}
     	\EndFor
     	\State update the step count $s\gets s+1$
     	\Until{$X^{(s-1)}_f= X^{(s)}_f$ for every firm $f$}
     	\Output{matching $\bar\mu$ with $\bar\mu(w) =Y^{(s-1)}_w$ for every worker $w$}
    \end{algorithmic}
\end{algorithm}

With the same reasoning as for the DA algorithm, the \textuptt{break-marriage($\mu', f', w'$)} procedure is guaranteed to terminate. Let $s^\star$ be the value of step count $s$ at the end of the algorithm. Note that, for every firm $f \in F$, we have 
\begin{equation} \label{eq:bm-Xf-nested}
    \bar X_f(\mu') \supseteq X^{(0)}_f \supseteq  X^{(1)}_f \supseteq \dots \supseteq X^{(s^\star)}_f,
\end{equation}
where the first containment is an equality unless $f=f'$. In particular, \eqref{eq:bm-Xf-nested} implies that $f'\notin X_{w'}^{(s)}$ for all $s\in \{0,1,\cdots, s^\star\}$. Also note that the termination condition implies 
\begin{equation} \label{eq:bm-term-f}
    \bar\mu(f)= \C_f(X_f^{(s^\star)}) = \C_f(X_f^{(s^\star-1)}) 
\end{equation}
for every firm $f$, while for every worker $w\neq w'$ it implies that 
\begin{equation} \label{eq:bm-term-w}
    \bar\mu(w)= Y_w^{(s^\star-1)} =\C_w(X_w^{(s^\star-1)})= X_w^{(s^\star-1)}.
\end{equation}

Let $(f,w)\in F\times W$, we say $f$ is \emph{rejected by $w$ at step $s$} if $f\in X_w^{(s)}\setminus Y_w^{(s)}$, and we say $f$ is \emph{rejected by $w$} if $f$ is rejected by $w$ at some step during the break-marriage procedure. Note that a firm $f$ is rejected by all and only the workers in $X_f^{(0)} \setminus X_f^{(s^\star)}$.

In the following, we prove Theorem~\ref{thm:bm-immediate}.

\begin{theorem} \label{thm:bm-immediate}
    Let $\mu',\mu \in \SS(\C_F,\C_W)$ and assume $\mu'$ is an immediate predecessor of $\mu$ in the stable matching lattice. Pick $(f', w')\in \mu'\setminus \mu$ and let $\bar \mu$ be the output matching of \textuptt{break-marriage($\mu', f', w'$)}. Then, $\bar\mu = \mu$.
\end{theorem}

We start by outlining the proof steps of Theorem~\ref{thm:bm-immediate}. We first show in Lemma~\ref{lem:bm-ind-rational} that the output matching $\bar\mu$ of \textuptt{break-marriage($\mu', f', w'$)} is individually rational. We then show in Lemma~\ref{lem:bm-success} that under a certain condition (i.e., the break-marriage operation being \emph{successful}), $\bar\mu$ is a stable matching and $\mu' \succ \bar\mu$. Lastly, we show that under the assumptions in the statement of Theorem~\ref{thm:bm-immediate}, the above-mentioned condition is satisfied and $\bar\mu \succeq \mu$.

\begin{lemma} \label{lem:bm-ind-rational}
    Let $\mu'\in \SS$ be a stable matching that is not the worker-optimal stable matching $\mu_W$ and let $(f',w')\in \mu’\setminus \mu_W$. Consider the \textuptt{break-marriage($\mu', f', w'$)} procedure with output $\bar\mu$. Then, $\bar \mu$ is individually rational.
\end{lemma}

\begin{proof}
    By~\eqref{eq:bm-term-f} and~\eqref{eq:bm-term-w} above, for every agent $a\in F\cup W\setminus \{w'\}$, $\bar \mu(a)= \C_a(X_a^{(s^\star-1)})$ and thus, $\C_a(\bar \mu(a)) = \C_a(\C_a(X_a^{(s^\star-1)})) =\C_a(X_a^{(s^\star-1)}) = \bar \mu(a)$, where the second equality is due to path-independence. For worker $w'$, note that $X^{(s^\star-1)}_{w'} = Y^{(s^\star-1)}_{w'} = \bar\mu(w')= \C_{w'} (X^{(s^\star-1)}_{w'}\cup \{f'\}) \setminus \{f'\}$, where the first equality is due to the termination criterion. Then, by the substitutability property, with $T=X^{(s^\star-1)}_{w'}$ and $S=X^{(s^\star-1)}_{w'} \cup \{f'\}$, we have that for every firm $f\in \bar\mu(w')$, $f\in \C_{w'}(X^{(s^\star-1)}_{w'})$ holds. Thus, $\bar\mu(w')\subseteq \C_{w'}(\bar\mu(w'))$. Since $\C_{w'}(X)\subseteq X$ for any $X$ in the domain of $\C_{w'}$, we have $\bar\mu(w')= \C_{w'}(\bar\mu(w'))$. Therefore, $\bar\mu$ is individually rational.
\end{proof}

\begin{lemma} \label{lem:bm-firm-worse}
     Consider the \textuptt{break-marriage($\mu', f', w'$)} procedure with output matching $\bar\mu$. Then, for every firm $f$, $\C_f(\bar\mu(f) \cup \mu'(f)) = \mu'(f)$.
\end{lemma}

\begin{proof}
    For a firm $f$, we have 
    \begin{align*}
    C_f(\bar\mu(f) \cup \mu'(f)) &= \C_f(\C_f(X_f^{(s^\star)}) \cup \C_f(\bar X_f(\mu'))\\
    &= \C_f(X_f^{(s^\star)} \cup \bar X_f(\mu'))= \C_f(\bar X_f(\mu')) = \mu'(f),
    \end{align*}
    where the first and last equality holds since $\mu'(f)=C_f(\bar X_f(\mu'))$ by Lemma~\ref{lem:choice-on-closure} and $\bar\mu(f)=C_f(X_f^{(s^\star)})$ by~\eqref{eq:bm-term-f}, the second equality is by path-independence, and the third equality is due to $X_f^{(s^\star)} \subseteq X_f^{(0)} \subseteq \bar X_f(\mu')$ by~\eqref{eq:bm-Xf-nested}.
\end{proof}

The following two properties of the break-marriage procedure are direct consequences of the path-independence assumption imposed on choice functions. These properties are also true for the deferred acceptance algorithm, as shown in~\cite{roth1984stability}. Let $f\in F$ and $w\in W$ be an arbitrary firm and worker. Lemma~\ref{lem:offer-remain-open} states that once $f$ proposes to $w$ in some step of the algorithm, it will keep proposing to $w$ in future steps until $w$ rejects $f$. Lemma~\ref{lem:rejection-final} states that once $w$ rejects $f$, $w$ would never accept $f$ in later steps even if the proposal is offered again.

\begin{lemma} \label{lem:offer-remain-open}
    For all $s\in [s^\star-1]$ and $w \in W$, we have $Y^{(s-1)}_w \subseteq X^{(s)}_w$.
\end{lemma}

\begin{proof}
    Let $f\in Y^{(s-1)}_w$. By construction, we have $w \in \C_f(X_f^{(s-1)})\cap X_f^{(s)}$. Since $X_f^{(s)}\subseteq X_f^{(s-1)}$ by~\eqref{eq:bm-Xf-nested}, we deduce that $w \in \C_f(X^{(s)}_f)$ by the substitutability property. Hence, $f \in X_w^{(s)}$ by definition.
\end{proof}

\begin{lemma} \label{lem:rejection-final}
   Let $s\in [s^\star-1]$, $f \in F$, and $w \in W$. Assume $f\in X^{(s-1)}_w \setminus Y^{(s-1)}_w$, i.e., $f$ is rejected by $w$ at step $s-1$. If $w\neq w'$, then for every step $s'\ge s$, $f\notin \C_w(X^{(s')}_w \cup \{f\})$; and if $w=w'$, then for every step $s'\ge s$, $f\notin \C_w(X^{(s')}_w\cup \{f'\} \cup \{f\})$.
\end{lemma}

\begin{proof}
    By construction, $w \notin X_f^{(s)}$. Hence, $f\notin X_w^{(s')}$ for all $s'\ge s$ because of~\eqref{eq:bm-Xf-nested} and the definition of $X_w^{(s')}$. Fix a value of $s'\geq s$. First consider the case when $w\neq w'$. By repeated application of the path-independence property of $\C_w$ and Lemma~\ref{lem:offer-remain-open}, we have 
    \begin{align*}
        \C_w(X_w^{(s')} \cup \{f\})&= \C_w(X_w^{(s')} \cup Y_w^{(s'-1)} \cup \{f\}) = \C_w(X_w^{(s')} \cup \C_w(Y_w^{(s'-1)} \cup \{f\})) \\
        &= \C_w(X_w^{(s')} \cup \C_w(\C_w(X_w^{(s'-1)}) \cup \{f\}))\\
        &= \C_w(X_w^{(s')} \cup \C_w(X_w^{(s'-1)} \cup \{f\})) \\
        & = \cdots \\ &= \C_w(\underbrace{X_w^{(s')}}_{\not\ni f} \cup \underbrace{X_w^{(s'-1)}}_{\not\ni f} \cup \cdots \cup \underbrace{\C_w(X_w^{(s-1)} \cup \{f\})}_{= \C_w(X_w^{(s-1)}) =Y_w^{(s-1)} \not\ni f}).
    \end{align*}
    Therefore, $f\notin \C_w(X_w^{(s')} \cup \{f\})$ as desired. We next consider the case where $w=w'$. Since $w \notin X_{f'}^{(0)}$ by construction, we have $w \notin X_{f'}^{(s-1)}$ by~\eqref{eq:bm-Xf-nested}, which then implies $f' \notin X_{w}^{(s-1)}$ by definition. Thus, we have $f\neq f'$. Again, by repeated application of the path-independence property of $\C_w$ and Lemma~\ref{lem:offer-remain-open}, we have 
    \begin{equation*}
    \resizebox{.97\textwidth}{!}{$\begin{aligned}
        \C_w(X^{(s')}_w\cup \{f'\} \cup \{f\})&= \C_w(X_w^{(s')} \cup Y_w^{(s'-1)} \cup \{f'\} \cup \{f\}) \\
        &= \C_w(X_w^{(s')} \cup \{f'\} \cup \C_w(Y_w^{(s'-1)} \cup \{f'\} \cup \{f\}))\\
        &= \C_w(X_w^{(s')} \cup \{f'\} \cup \C_w((\C_w( X_w^{(s'-1)} \cup\{f'\})\setminus \{f'\}) \cup\{f'\} \cup \{f\}))\\
        &= \C_w(X_w^{(s')} \cup \{f'\} \cup \C_w( X_w^{(s'-1)} \cup\{f'\} \cup \{f\})) \\ 
        & = \cdots \\
        &= \C_w(\underbrace{X_w^{(s')}}_{\not\ni f} \cup \underbrace{X_w^{(s'-1)}}_{\not\ni f} \cup \cdots \cup \{f'\} \cup \underbrace{\C_w(X_w^{(s-1)} \cup \{f'\}\cup \{f\})}_{= \C_w(X_w^{(s-1)} \cup \{f'\}) \not\ni f}).
    \end{aligned}$
    }
    \end{equation*}
    Therefore, $f\notin \C_w(X_w^{(s')} \cup \{f'\} \cup \{f\})$ as desired in this case as well.
\end{proof}

We say the break-marriage procedure \textuptt{break-marriage($\mu', f', w'$)} is \emph{successful} if $f'\notin \C_{w'}(X_{w'}^{(s^\star-1)} \cup \{f'\})$. We next show that when the procedure is successful, the output matching is stable. 

\begin{remark}
    For the \textsc{SM-Model}, McVitie and Wilson~\cite{mcvitie1971stable} defines the break-marriage procedure \textuptt{break-marriage($\mu', f', w'$)} to be successful if $w'$ receives a proposal from a firm that $w'$ prefers to $f'$. To translate this condition the \textsc{CM-QF-Model}, we interpret it as the follows: if $w'$ were to choose between this proposal and $f'$, $w'$ would not choose $f'$.
\end{remark}

\begin{lemma} \label{lem:bm-success}
    If \textuptt{break-marriage($\mu', f', w'$)} is successful, then the output matching $\bar\mu$ is stable. Moreover, $\mu'\succ \bar\mu$.
\end{lemma}

\begin{proof} 
    Since \textuptt{break-marriage($\mu', f', w'$)} is successful, applying the consistency property with $T=X^{(s^\star-1)}_{w'}$ and $S=T\cup \{f'\}$, we have $\C_{w'}(X^{(s^\star-1)}_{w'} \cup \{f'\})= \C_{w'}(X^{(s^\star-1)}_{w'})$ and thus, $Y_{w'}^{(s^\star-1)} = \C_{w'}(X_{w'}^{(s^\star-1)})$. In addition, by the termination condition, $Y_{w'}^{(s^\star-1)} = X_{w'}^{(s^\star-1)}$. Therefore, we have the following identity
    \begin{equation} \label{eq:bm-success-1}
        \bar\mu(w')= Y_{w'}^{(s^\star-1)} = X_{w'}^{(s^\star-1)} = \C_{w'}(X^{(s^\star-1)}_{w'}) = \C_{w'}(X^{(s^\star-1)}_{w'} \cup \{f'\}),
    \end{equation}
    which is similar to~\eqref{eq:bm-term-w} for other workers.
    
    \begin{claim} \label{cl:reject-final-reject}
        Let $(f,w) \in F \times W$. If $f$ is rejected by $w$ during the break-marriage procedure, then $f\notin \C_w(\bar\mu(w) \cup \{f\})$.
    \end{claim}
    
    \begin{proof}
        If $w\neq w'$, then by Lemma~\ref{lem:rejection-final}, $f\notin \C_w(X^{(s^\star-1)}_w \cup \{f\}) = \C_w(\bar\mu(w) \cup \{f\})$ where the equality is due to~\eqref{eq:bm-term-w}. This is also true if $w=w'$ because again by Lemma~\ref{lem:rejection-final}, 
        \begin{equation*}
            \resizebox{.97\textwidth}{!}{$f\notin \C_{w'}(X^{(s^\star-1)}_{w'} \cup \{f'\} \cup \{f\}) = \C_{w'}(\C_{w'}(X^{(s^\star-1)}_{w'} \cup \{f'\}) \cup \{f\}) = \C_{w'}(\bar\mu(w') \cup \{f\}),$}
        \end{equation*}
        where the first equality is by path-independence, and the second equality by~\eqref{eq:bm-success-1}.
    \end{proof}
    
    \begin{claim} \label{cl:bm-W-better}
        $\C_w(\mu'(w) \cup \bar\mu(w)) = \bar\mu(w)$ for all $w \in W$.
    \end{claim}
    
    \begin{proof}
        Let $f\in \mu'(w)\setminus \bar\mu(w)$, and suppose first $(f,w)\neq (f',w')$. Because of Lemma~\ref{lem:choice-on-closure} and Lemma~\ref{lem:offer-remain-open}, $f$ must be rejected by $w$ during the break-marriage procedure since otherwise $f\in X_w^{(s)}$ for all $s\in [s^\star]\cup \{0\}$, which in particular implies $w\in \bar\mu(f)$ due to~\eqref{eq:bm-term-w}. Then, by Claim~\ref{cl:reject-final-reject}, $f\notin \C_w(\bar\mu(w)\cup \{f\})$. Next assume $(f,w)=(f',w')$. By~\eqref{eq:bm-success-1}, we know that $X_w^{(s^\star -1)}=\bar \mu(w)$. Since \textuptt{break-marriage($\mu', f', w'$)} is successful, we have  $f' \notin \C_{w}(X_w^{(s^\star -1)}\cup \{f'\})=\C_{w}(\bar \mu(w)\cup \{f'\}$). We conclude that in both cases, $\C_w(\bar\mu(w)\cup \{f\}) =\C_w(\bar\mu(w))$ by consistency. Thus, we can apply Lemma~\ref{lem:PI-repeat} with $A_1=\bar\mu(w)$ and $A_2=\mu'(w)$ and conclude that $\C_w(\mu'(w)\cup \bar\mu(w))= \C_w(\bar\mu(w))$. The claim then follows from Lemma~\ref{lem:bm-ind-rational}.
    \end{proof}
    
    \noindent Fix an acceptable firm-worker pair $(f,w) \notin \bar\mu$. We show that $(f,w)$ does not block $\bar\mu$. Assume by contradiction that $f\in \C_w(\bar\mu(w)\cup \{f\})$ ($\dagger$) and $w\in \C_f(\bar\mu(f)\cup \{w\})$ ($\ddagger$). We claim that $(f,w)\notin \mu'$. If this is not the case, the consistency property of $\C_w$, with $S=\mu'(w)\cup \bar\mu(w)$ and $T=\bar\mu(w) \cup \{f\}$, implies $\C_w(\bar\mu(w) \cup \{f\})= \C_w(\mu'(w)\cup \bar\mu(w)) = \bar\mu(w)$, where the last equality is by Claim~\ref{cl:bm-W-better}. Thus, $f\notin \C_w(\bar\mu(w) \cup \{f\})$, which contradicts our assumption ($\dagger$). Thus, $(f,w)\notin \mu'$. Note that in particular, $(f,w)\neq (f',w')$. By Lemma~\ref{lem:dom-P-pre} and Claim~\ref{cl:bm-W-better}, ($\dagger$) implies $f\in \C_w(\mu'(w)\cup \{f\})$. Hence, we must have $w\notin \C_f(\mu'(f) \cup \{w\})$ since $\mu'$ is stable, i.e., not blocked by $(f,w)$. This implies $\C_f(\mu'(f) \cup \{w\}) = \C_f(\mu'(f)) = \mu'(f)$ due to the consistency property or $\C_f$ and the fact that $\mu'$ is individually rational. Thus, $w\in \bar X_f(\mu')=X_f^{(0)}$ since $f\neq f'$.
    
    Suppose first $w\notin X_f^{(s^\star)}$. Then, worker $w$ rejected firm $f$ during the break-marriage procedure. This implies $f\notin \C_w(\bar\mu(w)\cup \{f\})$ by Claim~\ref{cl:reject-final-reject}, contradicting assumption ($\dagger$). Suppose next $w\in X_f^{(s^\star)}$. Since $(f,w)\notin \bar\mu$, we have $w\notin \bar\mu(f) = \C_f(X_f^{(s^\star)})$, where the equality is due to~\eqref{eq:bm-term-f}. Then by the consistency property, with $S=X_f^{(s^\star)}$ and $T=\bar\mu(f)\cup \{w\}$, we have that $w \notin \C_f(\bar \mu(f)\cup \{w\})$. However, this contradicts ($\ddagger$). Therefore, $\bar\mu$ must be stable. 
    
    By Lemma~\ref{lem:bm-firm-worse}, $\mu'\succeq \bar\mu$. Moreover, we have $\mu'\neq \bar\mu$ since $f'\in \mu'(w')\setminus \bar\mu(w')$. Hence, $\mu'\succ \mu$ as desired.
\end{proof}

We now give the proof of Theorem~\ref{thm:bm-immediate}.

\begin{proof}[Proof of Theorem~\ref{thm:bm-immediate}]
    Note that by Lemma~\ref{lem:closure-contain-dominated}, $\mu(f)\subseteq \bar X_f(\mu')$ for every $f\in F$. We start by showing that during the break-marriage procedure, for every firm $f$, no worker in $\mu(f)$ rejects $f$. Assume by contradiction that this is not true. Let $s'$ be the first step where such a rejection happens, with firm $f_1$ being rejected by worker $w_1\in \mu(f_1)$. Hence, $f_1\in X_{w_1}^{(s')} \setminus Y_{w_1}^{(s')}$.
    
    \begin{claim}\label{cl:proof-of-break-marriage}
        There exists a firm $f_2\in Y_{w_1}^{(s')} \setminus \mu(w_1)$ such that $f_2\in \C_{w_1}(\mu(w_1) \cup\{f_2\})$. 
    \end{claim}
    
    \begin{proof}
        Assume by contradiction that such a firm $f_2$ does not exist. We first consider the case when $w_1\neq w'$. By Corollary~\ref{cor:PI-repeat} with $A_1=\mu(w_1)$ and $A_2=Y_{w_1}^{(s')}$, we have $\C_{w_1}(\mu(w_1) \cup Y_{w_1}^{(s')}) = \C_{w_1}(\mu(w_1)) =\mu(w_1)$, where the last equality is because $\mu$ is individually rational. Hence, $f_1 \in \C_{w_1}(\mu(w_1) \cup Y_{w_1}^{(s')})$, and using substitutability, we deduce $f_1\in \C_{w_1}(Y_{w_1}^{(s')} \cup \{f_1\})$. However, using consistency, with $T=Y_{w_1}^{(s')} \cup \{f_1\}$ and $S=X_{w_1}^{(s')}$, we conclude $\C_{w_1}(Y_{w_1}^{(s')} \cup \{f_1\}) = \C_{w_1}(X_{w_1}^{(s')}) = Y_{w_1}^{(s')} \not\ni f_1$, a contradiction.
        
        We next consider the case when $w_1=w'$. Note that $f_1\neq f'$, because $(f',w')\notin \mu$ by choice of $(f',w')$. Since $\mu'\succeq\mu$, by Theorem~\ref{thm:alkan}, $\C_{w'}(\mu'(w') \cup \mu(w')) = \mu(w')$. Thus, by the consistency property, with $S=\mu'(w')\cup \mu(w')$ and $T=\mu(w')\cup \{f'\}$, we have $\C_{w'}(\mu(w') \cup \{f'\}) = \mu(w')\not\ni f'$. As in the  case $w_1\neq w'$, by Corollary~\ref{cor:PI-repeat} with $A_1=\mu(w')$ and $A_2=Y_{w_1}^{(s')} \cup \{f'\}$ and the fact that $\mu$ is individually rational, $\mu(w')= \C_{w'}(\mu(w')) =\C_{w'}(\mu(w') \cup \{f'\} \cup Y_{w'}^{(s')}) $. Then, since $f_1 \in \mu(w')\cap X_{w'}^{(s')}$, by substitutability and path independence, we have: 
        \begin{align*}
        f_1\in \C_{w'}(Y_{w'}^{(s')} \cup \{f'\} \cup \{f_1\}) &= \C_{w'}(\C_{w'}(X_{w'}^{(s')} \cup \{f'\}) \setminus \{f'\} \cup \{f'\} \cup \{f_1\}) \\
        &= \C_{w'}(X_{w'}^{(s')} \cup \{f'\}).
        \end{align*}
        However, since $f_1\notin Y_{w'}^{(s')}$ by our choice and $f_1\neq f'$, we should have $f_1\notin \C_{w'}(X_{w'}^{(s')} \cup \{f'\})$, which is again a contradiction. 
    \end{proof}
    
    \noindent Now let $f_2$ be the firm whose existence is guaranteed by Claim~\ref{cl:proof-of-break-marriage}. In particular, $f_2 \in Y_{w_1}^{(s')}$ implies $w_1 \in \C_{f_2}(X_{f_2}^{(s')}) \subseteq X_{f_2}^{(s')}$. Note that by our choice of $f_1$, $\mu(f_2)\subseteq X_{f_2}^{(s')}$. Therefore, by substitutability and $w_1\in \C_{f_2}(X_{f_2}^{(s')})$, we have $w_1\in \C_{f_2}(\mu(f_2)\cup \{w_1\})$. However, this means that $(f_2,w_1)$ is a blocking pair of $\mu$, which contradicts stability of $\mu$. Thus, for every firm $f\in F$, no worker in $\mu(f)$ rejects $f$ during the break-marriage procedure as we claimed, which, together with the fact that $\mu(f)\subseteq \bar X_f(\mu')$, implies $\mu(f)\subseteq X_f^{(s^\star)}$. By path-independence and~\eqref{eq:bm-term-f}, we have that for every firm $f$:
    \begin{align} \label{eq:bm-immediate-1}
    \begin{split}
        \C_f(\bar\mu(f) \cup \mu(f)) &= \C_f(\C_f(X_f^{(s^\star)}) \cup \mu(f)) = \C_f(\C_f(X_f^{(s^\star)} \cup \mu(f))) \\
        &= \C_f(X_f^{(s^\star)}) = \bar\mu(f).
    \end{split}
    \end{align}
    Moreover, 
    \begin{equation} \label{eq:three-mu-card-same}
        |\mu(f)|=|\mu'(f)|=|\bar \mu(f)|, \; \forall f\in F
    \end{equation} 
    because 
    \begin{align*} 
    	|\mu(f)|= |\mu'(f)|&= |\C_f(\bar\mu(f) \cup \mu'(f))|\geq |\C_f(\bar\mu(f))|=|\bar \mu(f)| \\
	    &= |\C_f(\bar\mu(f) \cup \mu(f))|\geq |\C_f(\mu(f))|= |\mu(f)|,
    \end{align*}
    where the first equality is due to the~\ref{eq:equal-quota} property, the second and the fourth equalities are by Lemma~\ref{lem:bm-firm-worse} and~\eqref{eq:bm-immediate-1} respectively, the remaining two equalities are due to the fact that $\bar\mu$ and $\mu$ are individually rational, and the two inequalities hold because of cardinal monotonicity.
    
    We next show that the break-marriage procedure is successful. Consider the following two cases for a worker $w\neq w'$. The first is when $|\mu'(w)|<q_w$. By the~\ref{eq:full-quota} property, $w$ has the same set of partners in all stable matchings. In particular, $\mu'(w)=\mu(w)$. We claim that only firms from $\mu(w)$ proposes to $w$ during the break-marriage procedure. Assume by contradiction that a firm $f\notin \mu(w)$ proposes to $w$ at step $s$ (i.e., $w\in \C_f(X_f^{(s)})$). Then, since $\bar\mu(f) = \C_f(X_f^{(s^\star)}) \subseteq X_f^{(s^\star)} \subseteq X_f^{(s)}$ due to~\eqref{eq:bm-Xf-nested} and~\eqref{eq:bm-term-f}, by substitutability, we have $w\in \C_f(\bar\mu(f)\cup\{w\})$ and thus, $w\in C_f(\mu(f)\cup\{w\})$ because of~\eqref{eq:bm-immediate-1} and Lemma~\ref{lem:dom-P-pre}. Since $|\mu(w)|<q_w$, we also have that $f\in C_w(\mu(w)\cup \{f\})$ by the quota-filling property of $\C_w$. However, this contradicts stability of $\mu$. Therefore, $Y_{w}^{(s)}= X_{w}^{(s)} = \mu'(w)$ for all $s\in \{0,1,\cdots,s^\star\}$ by Lemma~\ref{lem:choice-on-closure} and Lemma~\ref{lem:offer-remain-open}. Hence, $\bar\mu(w)=\mu'(w)$ by~\eqref{eq:bm-term-w}. 
    
    We next consider the second case for worker $w\neq w'$, which is when $|\mu'(w)|=q_w$, and we claim that $|Y_w^{(s)}|=q_w$ for all $s\in \{0\} \cup [s^\star]$. We will show this by induction. For the base case with $s=0$, we want to show that $X_w^{(0)}\supseteq \mu'(w)$ because then we have $|X_w^{(0)}|\ge q_w$ and thus $|Y_w^{(0)}|=q_w$ by quota-filling. Let $f\in \mu'(w)$. If $f\neq f'$, then by Lemma~\ref{lem:choice-on-closure}, we have $w\in \C_f(X_f^{(0)})$; and if $f=f'$, by substitutability of $\C_{f'}$, we also have $w\in \C_f(X_f^{(0)})$ since $w\neq w'$. Hence, $f\in X_w^{(0)}$ by definition of $X_w^{(0)}$. For the inductive step, assume that $|Y_w^{(s-1)}|=q_w$ and we want to show that $|Y_w^{(s)}|=q_w$. Because of Lemma~\ref{lem:offer-remain-open}, $X_w^{(s)}\supseteq Y_w^{(s-1)}$. Hence, similar to the base case, we have $|X_w^{(s)}|\ge q_w$ and subsequently $|Y_w^{(s)}|=q_w$ by quota-filling. Therefore, $|\bar\mu(w)| = |\mu'(w)|$ by~\eqref{eq:bm-term-w}. 
    
    Combining both cases, we have $|\bar\mu(w)| = |\mu'(w)|$ for every worker $w\neq w'$. Together with~\eqref{eq:three-mu-card-same}, we have:
    \begin{align*}
        \sum_{w \in W\setminus\{w'\}} |\bar \mu(w)| + |\bar \mu(w')| &=\sum_{w \in W} |\bar \mu(w)| = \sum_{f \in F} |\overline \mu(f)|= \sum_{f \in F} |\mu(f)| \\
        &=\sum_{w \in W} |\mu(w)|=\sum_{w \in W\setminus\{w'\}} |\mu(w)| + |\mu(w')|.
    \end{align*}
    Hence, we must also have  $|\bar\mu(w')| = |\mu'(w')| = q_{w'}$. Therefore, $f'\not\in \C_{w'}(X_{w'}^{(s^\star-1)} \cup \{f'\})$ because otherwise $|\bar\mu(w')| = |\C_{w'}(X_{w'}^{(s^\star -1)} \cup \{f'\}) \setminus \{f'\}| \le q_{w'}-1$, where the inequality is by quota-filling. Hence, the break-marriage procedure is successful.
     
    Finally, by Lemma~\ref{lem:bm-success}, we have $\bar\mu\in \SS$ and $\mu'\succ \bar\mu$. We also have $\bar\mu \succeq \mu$ by~\eqref{eq:bm-immediate-1}. Therefore, it must be that $\bar\mu=\mu$ since $\mu$ is an immediate descendant of $\mu'$ in $\SS$.
\end{proof}

\begin{example} \label{eg:break-marriage}
    Consider the following instance adapted from the one given in~\cite{martinez2004algorithm}. One can check that every choice function is quota-filling with quota $2$.
     \[\begin{array}{ccl}
        f_1: & \ge_{f_1,1}: & w_1\ \underline{w_4}\ w_3\ w_2 \\
        & \ge_{f_1,2}: & \underline{w_2}\ w_3\ w_4\ w_1 \\
        & \\
        f_2: & \ge_{f_2,1}: & \underline{w_1}\ w_3\ w_4\ w_2 \\
        & \ge_{f_2,2}: & \underline{w_2}\ w_4\ w_3\ w_1 \\
        f_3: & \ge_{f_3,1}: & \underline{w_3}\ w_1\ w_2\ w_4 \\
        & \ge_{f_3,2}: & \underline{w_4}\ w_2\ w_1\ w_3 \\
        & \\
        f_4: & \ge_{f_4,1}: & \underline{w_3}\ w_2\ w_1\ w_4 \\
        & \ge_{f_4,2}: & w_4\ \underline{w_1}\ w_2\ w_3 \\
        &
    \end{array} \qquad \qquad \qquad 
    \begin{array}{ccl}
        w_1: & \ge_{w_1,1}: & f_3\ \underline{f_2}\ f_1\ f_4 \\
        & \ge_{w_1,2}: & \underline{f_4}\ f_2\ f_1\ f_3 \\
        & \ge_{w_1,3}: & f_3\ \underline{f_4}\ f_1\ f_2 \\
        w_2: & \ge_{w_2,1}: & f_3\ \underline{f_1}\ f_2\ f_4 \\
        & \ge_{w_2,2}: & f_4\ \underline{f_2}\ f_1\ f_3 \\
        w_3: & \ge_{w_3,1}: & f_1\ \underline{f_3}\ f_4\ f_2 \\
        & \ge_{w_3,2}: & f_2\ \underline{f_3}\ f_4\ f_1 \\
        & \ge_{w_3,3}: & f_1\ f_2\ \underline{f_4}\ f_3 \\
        w_4: & \ge_{w_4,1}: & \underline{f_1}\ f_3\ f_4\ f_2 \\
        & \ge_{w_4,2}: & f_2\ \underline{f_3}\ f_4\ f_1 \\
        & \ge_{w_4,3}: & \underline{f_1}\ f_2\ f_4\ f_3
    \end{array}\]
    
    Consider the stable matching $\mu'= (\{w_2, w_4\}, \{w_1, w_2\}, \{w_3, w_4\}, \{w_1, w_3\})$, where, to be concise, we list the assigned partners of firms $f_1, f_2, f_3, f_4$ in the exact order. Matched pairs are underlined above. The closure of $\mu'$ is $$\bar X(\mu')=\{\{w_2,w_3,w_4\}, \{w_1,w_2,w_3,w_4\}, \{w_1,w_2,w_3,w_4\}, \{w_1,w_2,w_3\}\}.$$ In the following, we describe the iterations of the \textuptt{break-marriage($\mu',f_1,w_2$)} procedure. The rejected firms are bolded.
    {\renewcommand{\arraystretch}{1.1}
    \[\begin{array}{c|clclclclc}
        &\;& s=0 &\;& s=1 &\;& s=2 &\;& s=3 &\; \\
        \hline
        X_{f_1}^{(s)}\ &\;& \{w_3,w_4\} &\;& \{w_3,w_4\} &\;& \{w_3,w_4\} &\;& \{w_3,w_4\} &\; \\
        X_{f_2}^{(s)}\ &\;& \{w_1,w_2,w_3,w_4\} &\;& \{w_1,w_2,w_3,w_4\} &\;& \{w_1,w_3,w_4\} &\;& \{w_1,w_3,w_4\} &\; \\ 
        X_{f_3}^{(s)}\ &\;& \{w_1,w_2,w_3,w_4\} &\;& \{w_1,w_2,w_3,w_4\} &\;& \{w_1,w_2,w_3,w_4\} &\;& \{w_1,w_2,w_3\} &\; \\
        X_{f_4}^{(s)}\ &\;& \{w_1,w_2,w_3\} &\;& \{w_1,w_2\} &\;& \{w_1,w_2\} &\;& \{w_1,w_2\} &\; \\
        \hline
        X_{w_1}^{(s)}\ &\;& \{f_2,f_4\} &\;& \{f_2,f_4\} &\;& \{f_2,f_4\} &\;& \{f_2,f_4\} &\; \\
        X_{w_2}^{(s)}\ &\;& \{f_2\} &\;& \{\bm{f_2}, f_4\} &\;& \{f_4\} &\;& \{f_3,f_4\} &\; \\
        X_{w_3}^{(s)}\ &\;& \{f_1,f_3, \bm{f_4}\} &\;& \{f_1,f_3\}  &\;& \{f_1,f_3\}  &\;& \{f_1,f_3\} &\; \\ 
        X_{w_4}^{(s)}\ &\;& \{f_1,f_3\} &\;& \{f_1,f_3\}  &\;& \{f_1,f_2, \bm{f_3}\} &\;& \{f_1,f_2\} &\; \\
        \hline
        Y_{w_1}^{(s)}\ &\;& \{f_2,f_4\} &\;& \{f_2,f_4\} &\;& \{f_2,f_4\} &\;& \{f_2,f_4\} &\; \\
        Y_{w_2}^{(s)}\ &\;& \{f_2\} &\;& \boxed{\{f_4\}} &\;& \{f_4\} &\;& \{f_3,f_4\} &\; \\
        Y_{w_3}^{(s)}\ &\;& \{f_1,f_3\} &\;& \{f_1,f_3\} &\;& \{f_1,f_3\} &\;& \{f_1,f_3\} &\; \\ 
        Y_{w_4}^{(s)}\ &\;& \{f_1,f_3\} &\;& \{f_1,f_3\} &\;& \{f_1,f_2\} &\;& \{f_1,f_2\}
    \end{array}\]
    }%
    
    The output matching is $\bar\mu= (\{w_3, w_4\}, \{w_1, w_4\}, \{w_2, w_3\}, \{w_1, w_2\})$, which, one can check, is stable. Note the step highlighted in box above where $Y_{w_2}^{(1)} = \C_{w_2}(\{f_2,f_4\} \cup\{f_1\}) \setminus \{f_1\}= \{f_1,f_4\} \setminus \{f_1\}= \{f_4\}$. If instead, $w_2$ used the original (i.e., same as in the DA algorithm) acceptance rule and accepted both $f_2$ and $f_4$, the algorithm would prematurely stopped after $s=1$, without leading to a stable matching.
\end{example}

We are now ready to present the algorithm that finds an immediate descendant for any given stable matching, using the break-marriage procedure. The details of the algorithm are presented in Algorithm~\ref{alg:immediate-descendant}. 

\begin{algorithm}[ht]
    \caption{Immediate descendant of stable matching $\mu'\neq \mu_W$} \label{alg:immediate-descendant}
	\begin{algorithmic}[1]
	    \normalsize
	    \Input{$\mu', \mu_W$}
		\State initialize $\T\gets \emptyset$
		\For{\textbf{each} $(f',w')\in \mu'\setminus \mu_W$}
		\State run the \textuptt{break-marriage($\mu',f',w'$)} procedure
		\IfThen{the procedure is successful}{add the output matching $\bar\mu$ to $\T$}
		\EndFor
		\State let $\mu^*$ be a matching in $\T$ 
		\For{\textbf{each} $\mu\in \T\setminus \{\mu^*\}$}
		\IfThen{$\mu\succeq\mu^*$}{update $\mu^*\gets \mu$}
		\EndFor \Comment{$\mu^*$ is a maximal matching from $\T$}
		\Output $\mu^*$
	\end{algorithmic}
\end{algorithm}

\begin{theorem} \label{thm:algo-descendant-correct}
    The output $\mu^*$ of Algorithm~\ref{alg:immediate-descendant} is an immediate descendant of $\mu'$ in the stable matching lattice $(\SS, \succeq)$.
\end{theorem}

\begin{proof}
    First note that due to Lemma~\ref{lem:bm-success}, all matchings in the set $\T$ constructed by Algorithm~\ref{alg:immediate-descendant} are stable matchings and $\mu'\succeq \mu$ for all $\mu\in \T$. Moreover, we claim that $\T\neq\emptyset$. Let $\mu_1\in \SS$ such that $\mu'$ is an immediate predecessor of $\mu_1$ in $(\SS,\succeq^\star)$. Such a stable matching $\mu_1$ exists because $\mu'\neq \mu_W$. Because of Lemma~\ref{lem:out-cannot-be-in}, we have $\mu' \setminus \mu_1 \subseteq \mu' \setminus \mu_W$ and thus by Theorem~\ref{thm:bm-immediate}, we have $\mu_1\in \T$. Hence, $\T\neq\emptyset$ as desired. Now, to prove the theorem, assume by contradiction that the output matching $\mu^*$ is not an immediate descendant of $\mu'$ in $(\SS, \succeq)$. Then, there exists a stable matching $\mu$ such that $\mu'\succ \mu\succ \mu^*$. By Lemma~\ref{lem:out-cannot-be-in}, for every firm-worker pair $(f',w')\in \mu'\setminus \mu$, we also have $(f',w')\notin \mu_W$. Thus, $\mu\in \T$ due to Theorem~\ref{thm:bm-immediate}. However, this means that $\mu^*$ is not a maximal matching from $\T$, which is a contradiction. 
\end{proof}

Finally, putting everything together, Algorithm~\ref{alg:max-chain} finds a maximal chain of the stable matching lattice, as well as the set of rotations. Its correctness follows from Theorem~\ref{thm:algo-descendant-correct}, Theorem~\ref{thm:chain-md}, and Theorem~\ref{thm:rp-isom-birkhoff}.

\begin{algorithm}[ht]
    \caption{A maximal chain of $(\SS,\succeq)$ and the set of rotations $\Pi$} \label{alg:max-chain}
	\begin{algorithmic}[1]
	    \normalsize
		\Input $\mu_F$ and $\mu_W$
		\State initialize counter $k\gets 0$ and $C_k\gets \mu_F$
		\While{$C_k\neq \mu_W$}
		\State run Algorithm~\ref{alg:immediate-descendant} with input $C_k$ and $\mu_W$, and let $\mu^*$ be its output
		\State update counter $k\gets k+1$ and $C_k\gets \mu^*$ 
		\EndWhile
		\Output maximal chain $C_0, C_1, \cdots, C_k$; and $\Pi= \{\rho_i \coloneqq \rho(C_{i-1}, C_i): i\in [k]\}$.
	\end{algorithmic}
\end{algorithm}

\subsection{Finding irreducible elements via maximal chains}\label{sec:find-from-chain}

The goal of this section is to prove the following. Note that the result below holds for any ring of sets. 

\begin{theorem} \label{thm:alg-correct-irreducible-ros}
    Consider a ring of sets $(\H,\subseteq)$ with base set $B$. Let $C_0, C_1, \cdots, C_k$ be a maximal chain of $(\H,\subseteq)$ and let $K_i\coloneqq C_i\setminus C_{i-1}$ for all $i\in [k]$. For $H\subseteq B$, let \textuptt{ros-membership} denote the running time of an algorithm that decides if $H\in \H$.
    There exists an algorithm with running time $O(k^2 \textuptt{ros-membership})$ that takes $C_0$, $C_1$, $\cdots$, $C_k$ as input and outputs, for each minimal difference $K_i$, a set of indices $\Lambda(K_i)$ such that $I(K_i) = \bigcup\{K_j: j\in \Lambda(K_i)\}\cup C_0$. In particular, this algorithm can be used to obtain the partial order $\sqsupseteq$ over $\D(\H)$.
\end{theorem}

We start with the theorem below, which gives an alternative definition of the partial order $\sqsupseteq$.

\begin{theorem}[Theorem 2.4.4, \cite{gusfield1989stable}] \label{thm:pomd-by-chain}
    Let $K_1, K_2\in \D(\H)$. Then, $K_1 \sqsupseteq K_2$ if and only if $K_1$ appears before $K_2$ on every maximal chain in $(\H,\subseteq)$.
\end{theorem}

We now present the algorithm stated in Theorem~\ref{thm:alg-correct-irreducible-ros} in  Algorithm~\ref{alg:irreducible-ros}. The idea is as follows. In order to find $I(K_i)$ (i.e., the minimal element in $\H$ that contains $K_i$), the algorithm tries to remove from the set $C_i$ as many items as possible, while keeping $C_i\in \H$. That is, the algorithm removes from $C_i$ all minimal differences $K\in \{K_1, K_2, \cdots, K_i\}$ such that $K\not \sqsupseteq K_i$. As we show in the proof of Theorem~\ref{thm:alg-correct-irreducible-ros}, the resulting set is $I(K_i)$. A demonstration of this algorithm is given in Example~\ref{eg:irreducible-ros} on the ring of sets from Example~\ref{eg:ring-of-sets}.

\begin{algorithm}[ht]
	\caption{} \label{alg:irreducible-ros}
	\begin{algorithmic}[1]
	    \normalsize
	    \Input A maximal chain $C_0, C_1, \cdots, C_k$ of $(\H,\subseteq)$.
	    \For{$i=1,2,\cdots,k$}
        \State{define $K_i\gets C_i\setminus C_{i-1}$}
        \State{initialize $H\gets C_i$ and $\Lambda(K_i)\gets \{1,2,\cdots,i\}$}
		\For{$j=i-1,i-2,\cdots, 1$}
		\If{$H\setminus K_j\in \H$} \label{step:memebership}
		\State{update $H \gets H \setminus K_j$ and $\Lambda(K_i) \gets \Lambda(K_i) \setminus \{j\}$}  
		\EndIf
		\EndFor
		\EndFor
		\Output $\Lambda(K_i)$ for all $i\in [k]$
	\end{algorithmic}
\end{algorithm}

\begin{example} \label{eg:irreducible-ros}
    Consider the ring of sets given in Example~\ref{eg:ring-of-sets}, and assume Algorithm~\ref{alg:irreducible-ros} takes in the maximal chain $C_0=H_1$, $C_1=H_2$, $C_2=H_4$, $C_3=H_6$, $C_4=H_7$. Then, $K_1=\{b\}$, $K_2=\{c\}$, $K_3=\{d,e\}$ and $K_4=\{f\}$. Now, image we would like to obtain $\Lambda(K_3)$. From Figure~\ref{fig:ros-rep}, it is clear that $I(K_3)=H_5$ and thus, $\Lambda(K_3)=\{2,3\}$. During Algorithm~\ref{alg:irreducible-ros}, at the outer \textbf{for} loop with $i=3$, $H$ is initialized to be $C_3=H_6$. In the first iteration of the inner \textbf{for} loop, since $H_6\setminus K_2=\{a,b,d,e\} \notin \H$, $\Lambda(K_3)$ remains $\{1,2,3\}$. Next, $H$ is updated to be $H_6\setminus K_1= H_5$ and $\Lambda(K_3)$ is updated to be $\{2,3\}$. The output is as expected.
\end{example}

We now give the proof of Theorem~\ref{thm:alg-correct-irreducible-ros}. 

\begin{proof}[Proof of Theorem~\ref{thm:alg-correct-irreducible-ros}]
    It is clear that the running time of Algorithm~\ref{alg:irreducible-ros} is $O(k^2 \textuptt{ros-membership})$. Suppose first the output of Algorithm~\ref{alg:irreducible-ros} is correct, that is, $I(K_i) = \bigcup\{K_j: j\in \Lambda(K_i)\}\cup C_0$. Then, for two minimal differences $K_{i_1}, K_{i_2}\in \D(\H)$, $K_{i_1} \sqsupseteq K_{i_2}$ if and only if $\Lambda(K_{i_1})\subseteq \Lambda(K_{i_2})$ by definition of $\sqsupseteq$. Hence, the partial order $\sqsupseteq$ can be obtained in time $O(k^2)$ from the output of Algorithm~\ref{alg:irreducible-ros}. It remains to show the correctness of Algorithm~\ref{alg:irreducible-ros}. Fix a value of $i \in [k]$ and for the following, consider the $i^\textup{th}$ iteration of the outer {\bf for} loop of the algorithm. Let $\{j_1,j_2, \cdots, j_M\}$ be an enumeration of $\Lambda(K_i)$ at the end of the iteration such that $j_1< j_2< \cdots< j_M$. Note that $j_M=i$. We start by showing the following claim.
    
    \begin{claim} \label{claim:chain-cannot-up}
        For all $m\in [M-1]$, $K_{j_m} \sqsupseteq K_i$.
    \end{claim}
    
    \begin{proof}
        We prove this by induction on $m$, where the base case is $m=M-1$. We start with the base case. Note that $j_m$ is the first index for which the \textbf{if} statement at Line~\ref{step:memebership} is evaluated to be false. That is, $(\bigcup_{\ell=1}^{j_m} K_{\ell}) \cup K_i \cup C_0\in \H$ but $(\bigcup_{\ell=1}^{j_m-1} K_{\ell}) \cup K_i \cup C_0\notin \H$. By Lemma~\ref{lem:unq-upperset}, $\{K_1, K_2, \cdots, K_{j_m}, K_i\}$ is an upper set of $(\D(\H),\sqsupseteq)$, and by Theorem~\ref{thm:birkhoff}, $\{K_1, K_2, \cdots, K_{j_m-1}, K_i\}$ is not an upper set of $(\D(\H),\sqsupseteq)$. Since for all $j'<j_m$, $K_{j_m} \not \sqsupseteq K_{j'}$ because of Theorem~\ref{thm:pomd-by-chain}, the reason why $\{K_1, K_2, \cdots, K_{j_m-1}, K_i\}$ is not an upper set of $(\D(\H),\sqsupseteq)$ must be that $K_{j_m}\sqsupseteq K_i$. For the inductive step, assume the claim is true for all $m'> m$ and we want to show that $K_{j_m}\sqsupseteq K_i$. Note that again by Theorem~\ref{thm:birkhoff}, $\{K_1, K_2, \cdots, K_{j_m}, K_{j_{m+1}}, K_{j_{m+2}}, \cdots, K_i\}$ is an upper set of $(\D(\H), \sqsupseteq)$ but $\{K_1, K_2, \cdots, K_{j_m-1}, K_{j_{m+1}}, K_{j_{m+2}}, \cdots, K_i\}$ is not an upper set of $(\D(\H), \sqsupseteq)$. With the same argument as in the base case, since for all $j'<j_m$, $K_{j_m} \not \sqsupseteq K_{j'}$ by Theorem~\ref{thm:pomd-by-chain}, it must be that $K_{j_m}\sqsupseteq K_{j_{m'}}$ for some $m'>m$. Therefore, applying the inductive hypothesis, we have $K_{j_m}\sqsupseteq K_i$ as desired.
    \end{proof}
    
    Let $H^*$ be set $H$ at the end of $i^\textup{th}$ iteration of the outer \textbf{for} loop. Note that $H^*= \bigcup\{K_j: j\in \Lambda(K_i)\}\cup C_0$ by construction. Since $K_i\subseteq H^*$, we have $I(K_i)\subseteq C_i\subseteq H^*$ by definition. Also note that by definition, $I(K_i) \in {\cal H}$. Assume by contradiction that $H^*\neq I(K_i)$ (i.e., $H^*\not \subseteq I(K_i)$). Consider a complete chain from the minimal element $H_0$ of $(\H,\subseteq)$ to $I(K_i)$ in $(\H, \subseteq)$, whose existence is guaranteed by Theorem~\ref{thm:chain-md}. Then, at least one minimal difference from $\{K_j: j\in \Lambda(K_i)\setminus \{i\} \}$, call it $K'$, is not contained in this complete chain. However, this means $K'\not\sqsupseteq K_i$ due to Theorem~\ref{thm:chain-md}, which contradicts Claim~\ref{claim:chain-cannot-up}. Therefore, we must have $I(K_i)= H^*$.
\end{proof}

\subsection{Partial order $\succeq^\star$ over $\Pi$}  \label{sec:irreducible}

In this section, we show how to obtain the partial order $\succeq^\star$ over the rotation poset $\Pi$. Recall that as stated in Theorem~\ref{thm:alg-correct-irreducible-ros} of the previous section, there exists an algorithm that finds the partial order $\sqsupseteq$ over $\D\coloneqq \D(\P)$ when given as input a maximal chain of $\P$. Employing the isomorphism between $\SS$ and $\P$ shown in Theorem~\ref{thm:iso-ros} and that between $\D$ and $\Pi$ shown in Theorem~\ref{thm:rp-isom-birkhoff}, we adapt the algorithm so that from a maximal chain of ${\cal S}$, we obtain the partial order $\succeq^\star$ over $\Pi$. 

\begin{algorithm}[ht]
	\caption{} \label{alg:irreducible}
	\begin{algorithmic}[1]
	    \normalsize
	    \Input outputs of Algorithm~\ref{alg:max-chain} -- maximal chain $C_0, \cdots, C_k$ of $(\SS,\succeq)$ and the set of rotations $\Pi= \{\rho_i \coloneqq \rho(C_{i-1}, C_i): i\in [k]\}$ 
	    \For{$i=1,2,\cdots,k$}
		\State initialize $\mu\gets C_i$ and $\Lambda(\rho_i)\gets \{1,2,\cdots,i\}$
		\For{$j=i-1,i-2,\cdots, 1$}
		\If{$\mu\triangle \rho_j^-\triangle \rho_j^+\in \SS$}
		\State{update $\mu\gets \mu \triangle \rho_j^-\triangle \rho_j^+$ and $\Lambda(\rho_i) \gets \Lambda(\rho_i) \setminus \{j\}$} 
		\EndIf
		\EndFor
		\EndFor
		\Output $\Lambda(\rho_i)$ for all $i\in [k]$
	\end{algorithmic}
\end{algorithm}

\begin{theorem} \label{thm:irreducible-compare-I} 
    Let $\Lambda(\rho)$ and $\Lambda(\rho')$ be the outputs of Algorithm~\ref{alg:irreducible} for rotations $\rho, \rho'\in \Pi$, respectively. Then, $\rho \succeq^\star \rho'$ if and only if $\Lambda(\rho) \subseteq \Lambda(\rho')$.
\end{theorem}

\begin{proof}
     To distinguish between the inputs of Algorithm~\ref{alg:irreducible} and Algorithm~\ref{alg:irreducible-ros}, we let $\mu_0, \mu_1, \cdots, \mu_k$ denote the maximal chain in the input of Algorithm~\ref{alg:irreducible}. Consider the outputs of Algorithm~\ref{alg:irreducible-ros} with inputs $C_i=P(\mu_i)$ for all $i\in [k]\cup \{0\}$. Then, because of the isomorphism between $(\SS,\succeq)$ and $(\P,\subseteq)$ and the isomorphism between $(\Pi,\succeq^\star)$ and $(\D,\sqsupseteq)$ as respectively stated in Theorem~\ref{thm:iso-ros} and Theorem~\ref{thm:rp-isom-birkhoff}, $K_i= Q(\rho_i)$ and $\Lambda(\rho_i)= \Lambda(K_i)$ for all $i\in [k]$, where $K_i= C_i\setminus C_{i-1}$ as defined in Algorithm~\ref{alg:irreducible-ros}. Thus, together with Theorem~\ref{thm:alg-correct-irreducible-ros},
     $$\rho \succeq^\star \rho' \Leftrightarrow Q(\rho) \sqsupseteq Q(\rho') \Leftrightarrow \Lambda(Q(\rho)) \subseteq \Lambda(Q(\rho')) \Leftrightarrow \Lambda(\rho) \subseteq \Lambda(\rho'),$$
     concluding the proof.
\end{proof}

\begin{example} \label{eg:irreducible}
    Consider the instance given in Example~\ref{eg:rotation-poset} and assume the maximal chain we obtained is $C_0=\mu_F, C_1=\mu_2, C_2=\mu_3, C_3=\mu_W$ so that we exactly have $\rho_i= \rho(C_{i-1}, C_i)$ for all $i\in [3]$ as denoted in Example~\ref{eg:rotation-poset}. Imagine we want to compare $\rho_1$ and $\rho_2$. As shown in Figure~\ref{fig:eg-rp}, $\rho_1 \succeq^\star \rho_2$. First, consider the outer \textbf{for} loop of Algorithm~\ref{alg:irreducible} with $i=1$. Then the body of the inner \textbf{for} loop is not executed and immediately we have $\Lambda(\rho_1)=\{1\}$. Next, consider the outer \textbf{for} loop of Algorithm~\ref{alg:irreducible} with $i=2$. Then, $\mu$ is initialized to be $\mu_3$ and $\Lambda(\rho_2)$ is initialized to be $\{1,2\}$. During the first and only iteration of the inner \textbf{for} loop, since $\mu_3\triangle \rho_1^-\triangle \rho_1^+$ is not a stable matching, $\Lambda(\rho_2)$ is not updated and remains $\{1,2\}$. Finally, $\Lambda(\rho_1)\subseteq \Lambda(\rho_2)$ and thus, $\rho_1 \succeq^\star \rho_2$ as expected.
\end{example}

\subsection{Summary and time complexity analysis} \label{sec:algo-summary}

The complete procedure to build the rotation poset is summarized in Algorithm~\ref{alg:complete-rp}. 

\begin{algorithm}[ht]
    \caption{Construction of the rotation poset $(\Pi, \succeq^\star)$} \label{alg:complete-rp}
 	\begin{algorithmic}[1]
 	    \normalsize
 	    \parState{Run Algorithm~\ref{alg:DA}'s, firm-proposing and worker-proposing, to obtain $\mu_F$ and $\mu_W$.}
 	    \parState{Run Algorithm~\ref{alg:max-chain} to obtain a maximal chain $C_0, C_1, \cdots, C_k$ of the stable matching lattice $(\SS, \succeq)$, and the set of rotations $\Pi \equiv \{\rho_1, \rho_2, \cdots, \rho_k\}$.}
 	    \parState{Run Algorithm~\ref{alg:irreducible} to obtain the sets $\Lambda(\rho_i)$ for each rotation $\rho_i\in \Pi$.}
 	    \parState{Define the partial order relation $\succeq^\star$: for $\rho_i, \rho_j\in \Pi$, $\rho_i \succeq^\star \rho_j \Leftrightarrow \Lambda(\rho_i) \subseteq \Lambda(\rho_j)$.}
    \end{algorithmic}
\end{algorithm}

The rest of the section focuses on time complexity analysis. 
\begin{theorem}
    Algorithm~\ref{alg:complete-rp} runs in time $|W|^3 |F|^3 \textuptt{oracle-call}$.
\end{theorem}

\paragraph{DA algorithm (Algorithm~\ref{alg:DA}).} Because of Lemma~\ref{lem:offer-remain-open} and Lemma~\ref{lem:rejection-final}, Algorithm~\ref{alg:DA} can be implemented as in Algorithm~\ref{alg:DA-efficient} to reduce the number of \textuptt{oracle-calls}. In particular, during each \textbf{repeat} loop, only firms that are rejected in the previous step (i.e., in $\bar F$) and only workers who receive new proposals (i.e., in $\bar W$) need to invoke their choice functions. Therefore, the \textbf{for} loop at Line~\ref{step:da-proposal} is entered at most $|F||W|$ times, and similarly, the \textbf{for} loop at Line~\ref{step:da-rejection} is entered at most $|F||W|$ times. That is, the total number of \textuptt{oracle-calls} is $O(|F||W|)$. Moreover, and for each firm-worker pair $(f,w)$, $w$ is removed from $X_f$ at most once and $f$ is added to $X_w$ at most once. That is, Line~\ref{step:da-proposal-ind} (resp.~Line~\ref{step:da-rejection-ind}) is repeated at most $|F||W|$ times. Therefore, the running time of the DA algorithm is $O(|F||W|\textuptt{oracle-call})$.

\begin{algorithm}[ht]
    \caption{Efficient implementation of Algorithm~\ref{alg:DA}} \label{alg:DA-efficient}
 	\begin{algorithmic}[1]
 	    \normalsize
 	    \State set $\bar F\gets F$ and $\bar W\gets \emptyset$
     	\ForInline{\textbf{each} firm $f$}{initialize $X_f \gets W(f)$ and $Y_f^{\textup{prev}} \gets \emptyset$ }
     	\ForInline{\textbf{each} worker $w$}{initialize $X_w\gets \emptyset$ and $Y_w^{\textup{prev}}\gets \emptyset$ }
     	\Repeat
     	\For{\textbf{each} firm $f\in \bar F$} \label{step:da-proposal}
     	\State $A_f \gets \C_f(X_f)$
     	\For{\textbf{each} worker $w\in A_f \setminus Y_f^{\textup{prev}}$}
     	\State update $X_w\gets X_w\cup \{f\}$ and $\bar W\gets \bar W\cup \{w\}$ \label{step:da-proposal-ind}
     	\EndFor
     	\State update $Y_f^{\textup{prev}} \gets A_f$
     	\EndFor
     	\State re-set $\bar F\gets \emptyset$
     	\For{\textbf{each} worker $w\in \bar W$} \label{step:da-rejection}
        \State $X_w\gets \C(X_w)$
        \For{\textbf{each} firm $f\in Y_w^{\textup{prev}} \setminus X_w$}
        \State update $X_f \gets X_f \setminus \{w\}$ and $\bar F\gets \bar F\cup \{f\}$ \label{step:da-rejection-ind}
        \EndFor
        \State update $Y_w^{\textup{prev}} \gets X_w$
     	\EndFor
     	\State re-set $\bar W\gets \emptyset$
     	\Until{$\bar F=\emptyset$}
        \Output{matching $\bar\mu$ with $\bar\mu(w)=Y_w^{\textup{prev}}$ for every worker $w$; closure $\tilde X(\bar\mu)$ with $\tilde X_f(\bar\mu)= X_f$ for every firm $f$}
    \end{algorithmic}
\end{algorithm}

\paragraph{Break-marriage procedure (Algorithm~\ref{alg:BM}).} Since the core steps (i.e., the loops) of the break-marriage procedure is the same as that of the DA algorithm, the running time of the break-marriage procedure is $O(|F||W|\textuptt{oracle-call})$, with the same arguments as above.

\paragraph{Immediate descendant (Algorithm~\ref{alg:immediate-descendant}).} Recall that $\bar q_f$ denotes the number of workers matched to firm $f$ under any stable matching (see the~\ref{eq:equal-quota} property). Let $\Upsilon\coloneqq \sum_{f\in F} \bar q_f$ denote the number of worker-firm pairs in any stable matching. Then, Algorithm~\ref{alg:BM} is run for at most $\Upsilon$ times. In addition, finding one maximal element $\mu^*$ from $\T$ requires at most $\Upsilon$ comparisons of pairs of stable matchings, each of which requires $|F|$ \textuptt{oracle-calls} by Part (3) of Lemma~\ref{lem:time-basic-ops}. All together, since $\Upsilon\le |F||W|$, the running time of Algorithm~\ref{alg:immediate-descendant} is $O(|F|^2 |W|^2 \textuptt{oracle-call})$.

\paragraph{Maximal chain (Algorithm~\ref{alg:max-chain}).} Since the length of a maximal chain of $\P$, and equivalently of $\SS$ due to Theorem~\ref{thm:iso-ros}, is at most the size of its base set due to Lemma~\ref{lem:center-uniq} and Theorem~\ref{thm:chain-md}, Algorithm~\ref{alg:immediate-descendant} is repeated for at most $|F||W|$ times. Thus, the running time of Algorithm~\ref{alg:max-chain} is $O(|F|^3 |W|^3 \textuptt{oracle-call})$. 

\paragraph{Partial order $\succeq^\star$ (Algorithm~\ref{alg:irreducible}).} Recall that checking if a matching is stable requires $O(|F||W|)$ \textuptt{oracle-calls} by Part (2) of Lemma~\ref{lem:time-basic-ops}. Thus, \texttt{ros-membership} is $O(|F||W| \textuptt{oracle-call})$. Since $k$ is at most $|F||W|$ as argued above, the running time of Algorithm~\ref{alg:irreducible} is $O(|F|^3 |W|^3 \textuptt{oracle-call})$. 

\paragraph{Rotation poset $(\Pi,\succeq^\star)$ (Algorithm~\ref{alg:complete-rp}).} Summing up the time of running Algorithm~\ref{alg:DA} twice, then Algorithm~\ref{alg:max-chain}, and lastly Algorithm~\ref{alg:irreducible}, the time complexity for building $(\Pi,\succeq^\star)$ is $O(|F|^3 |W|^3 \textuptt{oracle-call})$.

\section{The convex hull of lattice elements} \label{sec:poly}

Consider a poset $(Y,\succeq^\star)$. Its associated \emph{order polytope} is defined as
$$\O(Y,\succeq^\star) \coloneqq \{y\in [0,1]^Y: y_i\ge y_j, \ \forall i,j\in Y \textup{ s.t. } i\succeq^\star j \}.$$
A characterization of vertices and facets of $\O(X,\succeq^\star)$ is given in~\cite{stanley1986two}.

\begin{theorem}[\cite{stanley1986two}] \label{thm:order-polytope}
    The vertices of $\O(Y,\succeq^\star)$ are the characteristic vectors of upper sets of $Y$. The facets of $\O(Y,\succeq^\star)$ are all and only the following: $y_i \geq 0$ if $i$ is a minimal element of the poset; $y_i\leq 1$ if $i$ is a maximal element of the poset; $y_i\geq y_j$ if $i$ is an immediate predecessor of $j$.
\end{theorem}

\begin{proof}[Proof of Theorem~\ref{thm:poly-intro}]
    Let $(Y,\succeq^\star)$ affinely represent $(X,\succeq)$ via functions $\psi$ and $g(u)=Au+ x^0$. We claim that 
    \begin{equation} \label{eq:affine-map2}
    \begin{aligned}
        \conv(\X)&\coloneqq \conv(\{\Chi^{\mu}: \mu\in X\}) = \{x^0\} \oplus A\cdot \O(Y,\succeq^\star)\\
        &=\{x \in \R^X : x = x^0 + Ay, y \in \O(Y,\succeq^\star)\}, 
    \end{aligned}
    \end{equation}
    where $\oplus$ denotes the Minkowski sum operator. Indeed, by definition of affine representation and the fact that both polytopes, $\conv(\X)$ and $\O(Y, \succeq^\star)$, have $0/1$ vertices, $g$ defines a bijection between vertices of these two polytopes. Convexity then implies~\eqref{eq:affine-map2}. As $\O(Y,\succeq^\star)$ has $O(|Y|^2)$ facets shown in Theorem~\ref{thm:order-polytope}, we conclude the first statement from Theorem~\ref{thm:poly-intro}. 
    
    Now suppose that $A$ has full column rank. This implies that $\conv(\X)$ is affinely isomorphic to $\O(Y,\succeq^\star)$. Hence, there is a one-to-one correspondence between facets of $\O(Y,\succeq^\star)$ and facets of $\conv(\X)$, concluding the proof. 
\end{proof}

Following the proof of Theorem~\ref{thm:poly-intro}, when a poset ${\cal B}=(Y,\succeq^\star)$ affinely represent a lattice ${\cal L}=({\cal X},\succeq^*)$ via a function $g(u)=Au + x^0$, with $A$ having full column rank, many properties of $\conv({\cal X})$ can be derived from the analogous properties of ${\cal O}(Y,\succeq^*)$. For instance, the following immediately follows from the fact that ${\cal O}(Y,\succeq^*)$ is full-dimensional.

\begin{corollary}
    Let $\B=(Y,\succeq^\star)$ affinely represent the lattice ${\cal L}=({\cal X},\succeq)$ via functions $\psi$ and $g(u)=Au+ x^0$, with $A$ having full column rank. Then the dimension of $\conv({\cal X})$ is equal to the number of elements in ${\cal B}$.
\end{corollary}

Example~\ref{eg:polytope} shows that statements above need not hold when $A$ does not have full column-rank.

\begin{example} \label{eg:polytope}
    Consider the lattice $(\X, \succeq)$ and its representation poset $(Y,\succeq^\star)$ from Example~\ref{ex:not-aff-isomorphic}. Note that $$\conv(\X) = \{x\in [0,1]^4: x_1=1,\ x_2+x_3=1\}.$$
    Thus, $\conv(\X)$ has dimension $2$. On the other hand, $\O(Y,\succeq^\star)$ has dimension $3$. So the two polytopes are not affinely isomorphic. Polytopes $\conv(\X)$ and $\O(Y,\succeq^\star)$ are shown in Figure~\ref{fig:polytope}.
    \begin{figure}[ht]
        \centering
        \begin{subfigure}{.48\textwidth}
        \centering
        \begin{tikzpicture}[x=1cm, y=1.5cm, z=-0.6cm]
            \draw [thick, -{Latex[length=2mm]}] (0,0,0) -- (1.7,0,0) node [at end, right] {$x_2$};
            \draw [thick, -{Latex[length=2mm]}] (0,0,0) -- (0,1.3,0) node [at end, left] {$x_4$};
            \draw [thick, -{Latex[length=2mm]}] (0,0,0) -- (0,0,1.7) node [at end, left] {$x_3$};
            \draw[fill=gray!30, fill opacity=0.6] (1,0,0) -- (0,0,1) -- (0,1,1) -- (1,1,0) -- (1,0,0);
            \node[draw, circle, fill, scale=.5] (1) at (1,0,0) {};
            \node[draw, rectangle, fill, scale=.7] (2) at (0,0,1) {};
            \node[draw, star, star points=9, fill, scale=.5] (3) at (0,1,1) {};
            \node[draw, regular polygon,regular polygon sides=5, fill, scale=.5] (4) at (1,1,0) {};
        \end{tikzpicture}
        \caption{$\conv(\X)$ in the space of $x_1=1$}
        \end{subfigure}%
        \begin{subfigure}{.48\textwidth}
        \centering
        \begin{tikzpicture}[x=2cm, y=1.5cm, z=1cm]
            \draw [thick, -{Latex[length=2mm]}] (0,0,0) -- (2,0,0) node [at end, right] {$y_1$};
            \draw [thick, -{Latex[length=2mm]}] (0,0,0) -- (0,2,0) node [at end, left] {$y_3$};
            \draw [thick, -{Latex[length=2mm]}] (0,0,0) -- (0,0,2) node [at end, above] {$y_2$};
            \draw[fill=gray!30!yellow, fill opacity=0.2] (0,0,0) -- (1,0,0) -- (1,0,1) -- (0,0,0);
            \draw[fill=gray!30!red, fill opacity=0.2] (0,0,0) -- (1,0,0) -- (1,1,1) -- (0,0,0);
            \draw[fill=gray!30!blue, fill opacity=0.2] (1,1,1) -- (1,0,0) -- (1,0,1) -- (1,1,1);
            \draw[fill=gray!30!green, fill opacity=0.2] (0,0,0) -- (1,1,1) -- (1,0,1) -- (0,0,0);
            \node[draw, circle, fill, scale=.5] (a) at (0,0,0) {};
            \node[draw, rectangle, fill, scale=.7] (b) at (1,0,0) {};
            \node[draw, regular polygon,regular polygon sides=5, fill, scale=.5] (c) at (1,0,1) {};
            \node[draw, star, star points=9, fill, scale=.5] (d) at (1,1,1) {};
        \end{tikzpicture}
        \caption{$\O(Y,\succeq^\star)$}
        \end{subfigure}%
        \caption{Polytopes for Example~\ref{eg:polytope}.} \label{fig:polytope}
    \end{figure}
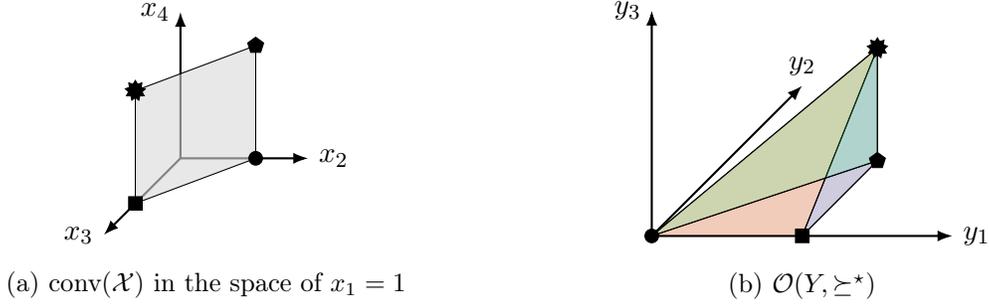
    
    More generally, one can easily construct a ``trivial'' distributive lattice $(\X, \succeq)$ such that the number of facets of $\O(Y,\succeq^\star)$ gives no useful information on the number of facets of $\conv(\X)$, where $(Y,\succeq^\star)$ is a poset that affinely represents $(\X, \succeq)$. In fact, the vertices of any $0/1$ polytope can be arbitrarily arranged in a chain to form a distributive lattice $(\X, \succeq)$. A poset $\O(Y,\succeq^\star)$ that affinely represents $(\X, \succeq)$ is given by a chain with $|Y|=|\X|-1$. It is easy to see that $\O(Y,\succeq^\star)$ is a simplex and has therefore $|Y|+1=|\X|$ facets. However, $\conv(\X)$ could have much more (or much less) facets than the number of its vertices.
\end{example}

\section{Representations of choice functions and algorithms}\label{sec:representation-and-algo}

Recall our previous observation that a choice function may be defined on all the (exponentially many) subsets of agents from the opposite side. The oracle model bypasses the computational concerns of representing choice functions explicitly. However, one drawback of this model is that it requires multiple rounds of communication between the ``central planner'' and each agent in the market. This, from an application point of view, is time-consuming: one of the major improvements brought about by the implementations of the Deferred Acceptance algorithm when applied, e.g., to the New York City school system, lies in the fact that it does not require multiple rounds of communication between the agents and the central planner~\cite{abdulkadirouglu2005new}.

This observation leads to the following practically relevant and theoretically intriguing questions: is there a way to represent choice functions ``compactly'', and do our algorithms perform efficiently in such a model? A natural starting point is the MC-representation defined in Section~\ref{sec:mc-representation}. We show in Section~\ref{sec:mc-rep-algo} that the time complexity of our algorithms in the model where choice functions are given through their MC-representation is polynomial in the input size (where now the input includes the MC-representations). However, the MC-representation of a choice function may need a number of preference relations that are  exponential in the number of agents (see Remark~\ref{rmk:mc-rep-exp}).

It is therefore interesting to investigate whether there are other ways to represent choice functions that is of size polynomial in the number of agents. Via a counting argument, we give a negative answer to this question in Section~\ref{sec:cardinal-monotone-cf-count} for choice functions that are substitutable, consistent, and cardinal monotone (see Theorem~\ref{thm:CMCF-count} and Remark~\ref{rmk:CMCF-rep}). We remark that our argument leaves it open whether a similar result holds if we replace cardinal monotonicity with quota-filling.

\subsection{Algorithms with MC-representation} \label{sec:mc-rep-algo}

In this section, we show how to modify the algorithms and analyze their time complexities when agents' choice functions are explicitly given via the MC-representations.

In Algorithm~\ref{alg:DA} and Algorithm~\ref{alg:BM}, instead of relying on an oracle model, we need to compute the outcomes of choice functions $\C_a(S)$ for agent $a\in F\cup W$ and subset of acceptable partners $S$. Using results in Section~\ref{sec:mc-representation}, $\C_a(S)$ can be obtained as a set of maximizers: $\{\max(S,\ge_{a,i}): i\in [p(\C_a)]\}$. Since each $\max(S,\ge_{a,i})$ requires $O(\max(|F|,|W|)$ time to compute, the time-complexity for obtaining $\C_a(S)$ is $O(\max(|F|,|W|)p(\C_a))$. Thus, for all previous results in terms of time complexity, one can simply replace $O(\textuptt{oracle-calls})$ with $O(\max(|F|,|W|) \max_{a\in F\cup W}p(\C_a))$. Note that this time complexity bound is polynomial in the input size, but could be exponential in the number of agents, since $\max_{a\in F\cup W}p(\C_a)$ maybe exponential in the number of the agents as discussed in Remark~\ref{rmk:mc-rep-exp}.

\begin{remark} \label{rmk:mc-rep-exp}
    \cite{dougan2021capacity} constructed strict preference lists (i.e., choice functions for the \textsc{MM-Model}) whose MC-representation needs exponentially many preference relations. Since such choice functions is a special case of the quota-filling choice functions, in general the MC-representation of quota-filling choice functions is not polynomial in the number of agents.
\end{remark}

\subsection{On the number of  substitutable, consistent, and cardinal monotone choice functions} \label{sec:cardinal-monotone-cf-count}

In this section, the domain of all choice functions is the family of subsets of $X$, with $|X|=n$. The simplest choice functions ${\cal C}$ appears in the \textsc{SM-Model}, where there is a single underlying strict preference list. The number of such choice functions is 
$$\sum_{i=0}^n \binom{n}{i} i! = \sum_{i=0}^n \frac{n!}{(n-i)!} =  \sum_{i=0}^n \frac{n!}{(n-i)!} = n!\sum_{i=0}^n\frac{1}{i!}\le e n!,$$
hence, singly exponential in $n$. On the other extreme, the number of all choice functions is doubly-exponential in $n$ (see, e.g., \cite{echenique2007counting}). We give the proof of this fact for completeness. 

\begin{theorem}
    The number of choice functions on subsets of $X$ with $|X|=n$ is $2^{n2^{n-1}}$.
\end{theorem}

\begin{proof}
    Since for each set of partners $S\subseteq X$ with $|S|=i$, $\C(S)$ can take $2^i$ possible values and there are $\binom{n}{i}$ subsets of $X$ with size $i$, the number of possible choice function is $\prod_{i=1}^n (2^i)^{\binom{n}{i}}$. Taking the logarithm with base $2$, we have
    $$\log_2\left( \prod_{i=1}^n (2^i)^{\binom{n}{i}}\right) =\sum_{i=1}^n \binom{n}{i}i= \sum_{i=1}^{n} \frac{n!}{(n-i)!(i-1)!} =n \sum_{i'=0}^{n-1} \frac{(n-1)!}{(n-1-i')!(i')!} = n 2^{n-1}.$$
\end{proof}

It has also been shown by Echenique~\cite{echenique2007counting} that when choice functions are assumed to be substitutable and consistent (i.e., path-independent), the number of choice functions remains doubly exponential in $n$. 

\begin{theorem}[\cite{echenique2007counting})]
    The number of substitutable and consistent choice functions on subsets of $X$ with $|X|=n$ is $2^{\Omega \left(\frac{2^{n-1}}{\sqrt{n-1}} \right)}$.
\end{theorem}

In the rest of the section, we show that the number of choice functions that additionally satisfies cardinal monotonicity remains doubly exponentially in $n$. The proof idea follows from that given in~\cite{echenique2007counting}.

\begin{theorem} \label{thm:CMCF-count}
    The number of substitutable, consistent, and cardinal monotone choice functions on subsets of $X$ with $|X|=n$ is $2^{\Omega \left(\frac{2^{n-1}}{\sqrt{n-1}} \right)}$.
\end{theorem}

\begin{remark} \label{rmk:CMCF-rep}
    Because of Theorem~\ref{thm:CMCF-count}, in order to encode all substitutable, consistent, and cardinal monotone choice function in binary strings, we need a number of strings that is super-polynomial in $n$, i.e., the number of agents in the market. 
\end{remark}

A family of subsets $\A\subseteq 2^X$ is an \emph{antichain} of $(2^X,\subseteq)$ if for any subsets $A,B\in \A$, they are not comparable, i.e., $A\setminus B\neq \emptyset$ and $B\setminus A\neq \emptyset$. A family of subsets $\F\subseteq 2^X$ is a \emph{filter} (i.e., \emph{lower set}) if for all $F\in \F$, $F'\supseteq F$ implies $F'\in \F$. Moreover, we say filter $\F$ is a \emph{filter at $x$} if for all $F\in \F$, we have $x\in F$. Note that $\emptyset$ is a filter at $x$.

\begin{theorem}[\cite{echenique2007counting}] \label{thm:antichain-subcf}
    There is an injective function mapping collections of antichains $\bm{\A}=\{\A_x: x\in X\}$ where each $\A_x$ is an antichain of the poset $(2^{X\setminus \{x\}}, \subseteq)$ to substitutable choice functions. The image of $\bm\A$ is defines as follows: for all $C\subseteq X$, $$\C(S)\coloneqq \{x\in S: S\notin \T_x\},$$ where  $$\T_x\coloneqq \{B\subseteq X: A\cup \{x\} \subseteq B \hbox{ for some $A \in \A_x$}\}.$$
    Moreover, $\T_x$ is a filter at $x$ for all $x\in X$.
\end{theorem}

Because of Theorem~\ref{thm:antichain-subcf}, let $\C[\bm\A]$ denote the substitutable choice function corresponding to the collection of antichains $\bm\A$ constructed by the statement of the theorem.

\begin{lemma} \label{lem:CMCF-construct}
    Let $(Y,W)$ be a partition of $X$ with $W=\{w\}$. Let $\bm\A=\{\A_x: x\in X\}$ be a collection of antichains such that (i) for all $x\in Y$, $\A_x=\emptyset$ and (ii) $\A_w$ is an antichain of $(2^Y,\subseteq)$. Then $\C[\bm\A]$ is consistent and cardinal monotone.
\end{lemma}

\begin{proof}
    We abbreviate $\C:=\C[\bm\A]$. Let $\T_x$ be as defined in the statement of Theorem~\ref{thm:antichain-subcf}. That is, $\T_x= \emptyset$ for all $x \in Y$ and $\T_w$ is a filter at $w$. Hence, note that $S\cap Y\subseteq \C(S)$ for all $S \subseteq X$ ($\sharp$).
    
    Let $T\subseteq X$. We consider first the case when $w \notin T$. Then, $\C(T) = T$ because of ($\sharp$). Let $S\subseteq X$ be such that $\C(T)\subseteq S\subseteq T$. Then it must be that $S=T$ and it follows immediately that $\C(T)=\C(S)$. In addition, for all $S\subseteq T$, we also have $S\subseteq Y$ and thus, using ($\sharp$) again, $|\C(S)|= |S|\le |T| = |\C(T)|$. 
    
    We next consider the case when $ w\in T$. Then, either $\C(T) = T$ or $\C(T) = T\setminus \{w\}$, again because of ($\sharp$). We start with the consistency property. Assume we are in the former case, and let $S\subseteq X$ be such that $\C(T)\subseteq S\subseteq T$. Since $T=\C(T)$, we have $S=T$ and thus $\C(T)=\C(S)$. Now assume we are in the latter case: $\C(T)=T\setminus \{w\}$. If $S\subseteq X$ satisfies $\C(T)\subseteq S\subseteq T$, we either have $S=T$ or $S=T\setminus \{w\}$. Regardless, we have $\C(S)=\C(T)$. Lastly, we show the cardinality monotonicity property, and we consider both cases at once. For all $S\subsetneq T$, we either have $\C(S)=S$ or $\C(S)=S\setminus \{w\}$ due to ($\sharp$). Either way, $|\C(S)|\le |S| \le |T|-1\le |\C(T)|$. Hence, $\C$ is both consistent and cardinal monotone, concluding the proof. 
\end{proof}

Thus, a lower bound to the number of substitutable, consistent, and cardinal monotone choice functions can be obtained by counting the number of antichains. The problem of counting the number of antichains of a poset is called the \emph{Dedekind's problem}. Let $\mathcal{N}(k)$ denote the collection of antichains of poset $(2^{[k]},\subseteq)$. The following result is well-known and we include the proof for completeness.

\begin{lemma} \label{lem:dedekind-lb}
    $|\mathcal{N}(k)| \ge 2^{\binom{k}{\lfloor k/2 \rfloor}} = 2^{\Theta(2^k/ \sqrt{k})}$.
\end{lemma}

\begin{proof}
    Consider any two distinct subsets $A,B\subseteq X$ with $|A|=|B|$, then it must be that $A\setminus B\neq \emptyset$ and $B\setminus A\neq \emptyset$. Thus, a collection of subsets, each with the same size, is an antichain of $(2^{[k]},\subseteq)$. Therefore, the number of antichains of $(2^{[k]}, \subseteq)$ is at least the number of subsets of $\{A\subseteq X: |A|=\lfloor k/2 \rfloor\}$, which is exactly $2^{\binom{k}{\lfloor k/2 \rfloor}}$ since there are $\binom{k}{\lfloor k/2 \rfloor}$ subsets of $X$ with size $\lfloor k/2 \rfloor$. The last equality follows from Stirling's approximation.
\end{proof}
 
We now present the proof for Theorem~\ref{thm:CMCF-count}.

\begin{proof}[Proof of Theorem~\ref{thm:CMCF-count}]
    Let $(Y,W)$ be a partition of $X$ with $|Y|=n-1$ and $|W|=1$, as in the statement of Lemma~\ref{lem:CMCF-construct}. By Lemma~\ref{lem:dedekind-lb}, the possible choices of antichains $\A_x$ for $x\in W$ is at least $\mathcal{N}(n-1)$. Hence, the number of $\bm\A$ (i.e., collection of antichains) in the statement of Lemma~\ref{lem:CMCF-construct} is also at least $\mathcal{N}(n-1)$. Finally, together with Theorem~\ref{thm:antichain-subcf}, we have that the number of substitutable, consistent, and cardinal monotone choice functions is again at least $\mathcal{N}(n-1)$.
\end{proof}

\section{Concluding remarks} \label{sec:concluding}

Our results show that approaching stable matching problems by regarding their feasible regions as a distributive lattice leads to efficient optimization algorithms and a polyhedral description of the associated convex sets. Our study leaves some questions open and it poses research directions which we think are worth exploring. 

First, it is not clear if algorithms from Section~\ref{sec:algo} extend to the \textsc{CM-Model} -- or even beyond --  and if conversely the lower bound from Section~\ref{sec:representation-and-algo} extends to choice functions that are quota-filling. Second, there has been some recent work showing how feasible regions of certain problems in combinatorial optimization can be seen as a distributive lattice~\cite{garg2020predicate}. This fact, combined with our approach, may lead to (known or new) efficient algorithms for optimizing linear functions over the associated polytopes.

\medskip\noindent\textbf{Acknowledgements.} Yuri Faenza acknowledges support from the NSF Award 2046146 \emph{Career: An Algorithmic Theory of Matching Markets}. Xuan Zhang thanks the Cheung Kong Graduate School of Business (CKGSB) for their fellowship support. Part of this work was done when Yuri Faenza was a Visiting Scientist in the \emph{Online and Matching-Based Market Design program} at the Simons Institute in Fall 2019. The authors wish to thank Vijay Garg for pointing out to them the reference~\cite{garg2020predicate}, as well as the members of the OpLog division at UBC for stimulating discussions when the material from the current paper was presented there.

\bibliography{ref} 

\begin{thebibliography}{41}
\providecommand{\natexlab}[1]{#1}
\providecommand{\url}[1]{\texttt{#1}}
\expandafter\ifx\csname urlstyle\endcsname\relax
  \providecommand{\doi}[1]{doi: #1}\else
  \providecommand{\doi}{doi: \begingroup \urlstyle{rm}\Url}\fi

\bibitem[Abdulkadiro{\u{g}}lu and S{\"o}nmez(2003)]{abdulkadirouglu2003school}
Atila Abdulkadiro{\u{g}}lu and Tayfun S{\"o}nmez.
\newblock School choice: A mechanism design approach.
\newblock \emph{American economic review}, 93\penalty0 (3):\penalty0 729--747,
  2003.

\bibitem[Abdulkadiro{\u{g}}lu et~al.(2005)Abdulkadiro{\u{g}}lu, Pathak, and
  Roth]{abdulkadirouglu2005new}
Atila Abdulkadiro{\u{g}}lu, Parag~A Pathak, and Alvin~E Roth.
\newblock The {N}ew {Y}ork {C}ity {H}igh {S}chool {M}atch.
\newblock \emph{American Economic Review}, 95\penalty0 (2):\penalty0 364--367,
  2005.

\bibitem[Aizerman and Malishevski(1981)]{aizerman1981general}
Mark Aizerman and Andrew Malishevski.
\newblock General theory of best variants choice: Some aspects.
\newblock \emph{IEEE Transactions on Automatic Control}, 26\penalty0
  (5):\penalty0 1030--1040, 1981.

\bibitem[Alkan(2001)]{alkan2001preferences}
Ahmet Alkan.
\newblock On preferences over subsets and the lattice structure of stable
  matchings.
\newblock \emph{Review of Economic Design}, 6\penalty0 (1):\penalty0 99--111,
  2001.

\bibitem[Alkan(2002)]{alkan2002class}
Ahmet Alkan.
\newblock A class of multipartner matching markets with a strong lattice
  structure.
\newblock \emph{Economic Theory}, 19\penalty0 (4):\penalty0 737--746, 2002.

\bibitem[Aprile et~al.(2018)Aprile, Cevallos, and Faenza]{aprile20182}
Manuel Aprile, Alfonso Cevallos, and Yuri Faenza.
\newblock On 2-level polytopes arising in combinatorial settings.
\newblock \emph{SIAM Journal on Discrete Mathematics}, 32\penalty0
  (3):\penalty0 1857--1886, 2018.

\bibitem[Ayg{\"u}n and S{\"o}nmez(2013)]{aygun2013matching}
Orhan Ayg{\"u}n and Tayfun S{\"o}nmez.
\newblock Matching with contracts: Comment.
\newblock \emph{American Economic Review}, 103\penalty0 (5):\penalty0 2050--51,
  2013.

\bibitem[Ayg{\"u}n and Turhan(2016)]{aygun2016dynamic}
Orhan Ayg{\"u}n and Bertan Turhan.
\newblock Dynamic reserves in matching markets: Theory and applications.
\newblock \emph{Available at SSRN 2743000}, 2016.

\bibitem[Ba{\"\i}ou and Balinski(2000{\natexlab{a}})]{baiou2000many}
Mourad Ba{\"\i}ou and Michel Balinski.
\newblock Many-to-many matching: stable polyandrous polygamy (or polygamous
  polyandry).
\newblock \emph{Discrete Applied Mathematics}, 101\penalty0 (1-3):\penalty0
  1--12, 2000{\natexlab{a}}.

\bibitem[Ba{\"\i}ou and Balinski(2000{\natexlab{b}})]{baiou2000stable}
Mourad Ba{\"\i}ou and Michel Balinski.
\newblock The stable admissions polytope.
\newblock \emph{Mathematical programming}, 87\penalty0 (3):\penalty0 427--439,
  2000{\natexlab{b}}.

\bibitem[Bansal et~al.(2007)Bansal, Agrawal, and
  Malhotra]{bansal2007polynomial}
Vipul Bansal, Aseem Agrawal, and Varun~S Malhotra.
\newblock Polynomial time algorithm for an optimal stable assignment with
  multiple partners.
\newblock \emph{Theoretical Computer Science}, 379\penalty0 (3):\penalty0
  317--328, 2007.

\bibitem[Birkhoff(1937)]{birkhoff1937rings}
Garrett Birkhoff.
\newblock Rings of sets.
\newblock \emph{Duke Mathematical Journal}, 3\penalty0 (3):\penalty0 443--454,
  1937.

\bibitem[Blair(1988)]{blair1988lattice}
Charles Blair.
\newblock The lattice structure of the set of stable matchings with multiple
  partners.
\newblock \emph{Mathematics of operations research}, 13\penalty0 (4):\penalty0
  619--628, 1988.

\bibitem[Chambers and Yenmez(2017)]{chambers2017choice}
Christopher~P Chambers and M~Bumin Yenmez.
\newblock Choice and matching.
\newblock \emph{American Economic Journal: Microeconomics}, 9\penalty0
  (3):\penalty0 126--47, 2017.

\bibitem[Do{\u{g}}an et~al.(2021)Do{\u{g}}an, Do{\u{g}}an, and
  Y{\i}ld{\i}z]{dougan2021capacity}
Battal Do{\u{g}}an, Serhat Do{\u{g}}an, and Kemal Y{\i}ld{\i}z.
\newblock On capacity-filling and substitutable choice rules.
\newblock \emph{Mathematics of Operations Research}, 2021.

\bibitem[Echenique(2007)]{echenique2007counting}
Federico Echenique.
\newblock Counting combinatorial choice rules.
\newblock \emph{Games and Economic Behavior}, 58\penalty0 (2):\penalty0
  231--245, 2007.

\bibitem[Echenique and Yenmez(2015)]{echenique2015control}
Federico Echenique and M~Bumin Yenmez.
\newblock How to control controlled school choice.
\newblock \emph{American Economic Review}, 105\penalty0 (8):\penalty0 2679--94,
  2015.

\bibitem[Fleiner(2003)]{fleiner2003stable}
Tam{\'a}s Fleiner.
\newblock On the stable b-matching polytope.
\newblock \emph{Mathematical Social Sciences}, 46\penalty0 (2):\penalty0
  149--158, 2003.

\bibitem[Gale and Shapley(1962)]{gale1962college}
David Gale and Lloyd~S Shapley.
\newblock College admissions and the stability of marriage.
\newblock \emph{The American Mathematical Monthly}, 69\penalty0 (1):\penalty0
  9--15, 1962.

\bibitem[Garg(2020)]{garg2020predicate}
Vijay~K Garg.
\newblock Predicate detection to solve combinatorial optimization problems.
\newblock In \emph{Proceedings of the 32nd ACM Symposium on Parallelism in
  Algorithms and Architectures}, pages 235--245, 2020.

\bibitem[Gusfield and Irving(1989)]{gusfield1989stable}
Dan Gusfield and Robert~W Irving.
\newblock \emph{The stable marriage problem: structure and algorithms}.
\newblock MIT press, 1989.

\bibitem[Hatfield and Milgrom(2005)]{hatfield2005matching}
John~William Hatfield and Paul~R Milgrom.
\newblock Matching with contracts.
\newblock \emph{American Economic Review}, 95\penalty0 (4):\penalty0 913--935,
  2005.

\bibitem[Irving and Leather(1986)]{irving1986complexity}
Robert~W Irving and Paul Leather.
\newblock The complexity of counting stable marriages.
\newblock \emph{SIAM Journal on Computing}, 15\penalty0 (3):\penalty0 655--667,
  1986.

\bibitem[Irving et~al.(1987)Irving, Leather, and Gusfield]{irving1987efficient}
Robert~W Irving, Paul Leather, and Dan Gusfield.
\newblock An efficient algorithm for the “optimal” stable marriage.
\newblock \emph{Journal of the ACM (JACM)}, 34\penalty0 (3):\penalty0 532--543,
  1987.

\bibitem[Kamada and Kojima(2015)]{kamada2015efficient}
Yuichiro Kamada and Fuhito Kojima.
\newblock Efficient matching under distributional constraints: Theory and
  applications.
\newblock \emph{American Economic Review}, 105\penalty0 (1):\penalty0 67--99,
  2015.

\bibitem[Kelso~Jr and Crawford(1982)]{kelso1982job}
Alexander~S Kelso~Jr and Vincent~P Crawford.
\newblock Job matching, coalition formation, and gross substitutes.
\newblock \emph{Econometrica: Journal of the Econometric Society}, pages
  1483--1504, 1982.

\bibitem[Knuth(1976)]{knuth1976marriages}
Donald~Ervin Knuth.
\newblock Marriages stables.
\newblock \emph{Technical report}, 1976.

\bibitem[Manlove(2013)]{manlove2013algorithmics}
David Manlove.
\newblock \emph{Algorithmics of matching under preferences}, volume~2.
\newblock World Scientific, 2013.

\bibitem[Mart{\'i}nez et~al.(2004)Mart{\'i}nez, Mass{\'o}, Neme, and
  Oviedo]{martinez2004algorithm}
Ruth Mart{\'i}nez, Jordi Mass{\'o}, Alejandro Neme, and Jorge Oviedo.
\newblock An algorithm to compute the full set of many-to-many stable
  matchings.
\newblock \emph{Mathematical Social Sciences}, 47\penalty0 (2):\penalty0
  187--210, 2004.

\bibitem[McVitie and Wilson(1971)]{mcvitie1971stable}
David~G McVitie and Leslie~B Wilson.
\newblock The stable marriage problem.
\newblock \emph{Communications of the ACM}, 14\penalty0 (7):\penalty0 486--490,
  1971.

\bibitem[Nguyen and Vohra(2019)]{nguyen2019stable}
Th{\`a}nh Nguyen and Rakesh Vohra.
\newblock Stable matching with proportionality constraints.
\newblock \emph{Operations Research}, 2019.

\bibitem[Picard(1976)]{picard1976maximal}
Jean-Claude Picard.
\newblock Maximal closure of a graph and applications to combinatorial
  problems.
\newblock \emph{Management science}, 22\penalty0 (11):\penalty0 1268--1272,
  1976.

\bibitem[Roth(1984{\natexlab{a}})]{roth1984evolution}
Alvin~E Roth.
\newblock The evolution of the labor market for medical interns and residents:
  a case study in game theory.
\newblock \emph{Journal of political Economy}, 92\penalty0 (6):\penalty0
  991--1016, 1984{\natexlab{a}}.

\bibitem[Roth(1984{\natexlab{b}})]{roth1984stability}
Alvin~E Roth.
\newblock Stability and polarization of interests in job matching.
\newblock \emph{Econometrica: Journal of the Econometric Society}, pages
  47--57, 1984{\natexlab{b}}.

\bibitem[Roth(1986)]{roth1986allocation}
Alvin~E Roth.
\newblock On the allocation of residents to rural hospitals: a general property
  of two-sided matching markets.
\newblock \emph{Econometrica: Journal of the Econometric Society}, pages
  425--427, 1986.

\bibitem[Roth et~al.(1993)Roth, Rothblum, and Vande~Vate]{roth1993stable}
Alvin~E Roth, Uriel~G Rothblum, and John~H Vande~Vate.
\newblock Stable matchings, optimal assignments, and linear programming.
\newblock \emph{Mathematics of operations research}, 18\penalty0 (4):\penalty0
  803--828, 1993.

\bibitem[Rothblum(1992)]{rothblum1992characterization}
Uriel~G Rothblum.
\newblock Characterization of stable matchings as extreme points of a polytope.
\newblock \emph{Mathematical Programming}, 54\penalty0 (1-3):\penalty0 57--67,
  1992.

\bibitem[Schrijver(2003)]{schrijver2003combinatorial}
Alexander Schrijver.
\newblock \emph{Combinatorial optimization: polyhedra and efficiency},
  volume~24.
\newblock Springer Science \& Business Media, 2003.

\bibitem[Stanley(1986)]{stanley1986two}
Richard~P Stanley.
\newblock Two poset polytopes.
\newblock \emph{Discrete \& Computational Geometry}, 1\penalty0 (1):\penalty0
  9--23, 1986.

\bibitem[Tomoeda(2018)]{tomoeda2018finding}
Kentaro Tomoeda.
\newblock Finding a stable matching under type-specific minimum quotas.
\newblock \emph{Journal of Economic Theory}, 176:\penalty0 81--117, 2018.

\bibitem[Vate(1989)]{vate1989linear}
John H~Vande Vate.
\newblock Linear programming brings marital bliss.
\newblock \emph{Operations Research Letters}, 8\penalty0 (3):\penalty0
  147--153, 1989.

\end{thebibliography}
\bibliographystyle{plainnat}

\end{document}